\documentclass[12pt]{article}

\usepackage{amsthm}
\usepackage{amssymb}
\usepackage{amsmath}
\usepackage{epsfig}
\usepackage{psfrag}
\usepackage{amscd}
\usepackage{graphicx}
\usepackage{epstopdf}
\newtheorem{theorem}{Theorem}[section]
\newtheorem{corollary}[theorem]{Corollary}
\newtheorem{proposition}[theorem]{Proposition}
\newtheorem{lemma}[theorem]{Lemma}

\theoremstyle{definition}

\newtheorem{remark}[theorem]{Remark}

\newtheorem{definition}[theorem]{Definition}

\newcommand{\eqdef}{\;{:=}\;}
\newcommand\PSL{\operatorname{PSL}}

\def\R{\mathbb R}
\def\C{\mathbb C}

\def\Q{\mathbb Q}

\def\bigmid{\ \rule[-3.5mm]{0.1mm}{9mm}\ }

\newcommand{\CC}{{\mathbb C}}

\newcommand{\RR}{{\mathbb R}}

\begin{document}

\title{New obstructions to symplectic embeddings}

\author{R. Hind \and E. Kerman}

\date{\today}

\maketitle

\begin{abstract}
In this paper we establish new restrictions on the symplectic embeddings of basic shapes in symplectic vector spaces.
By refining an  embedding technique due to Guth, we also show that they are sharp.
\end{abstract}

\section{Introduction and main results}

Consider $\R^{2n}$ equipped with coordinates $(x_1, \ldots, x_n, y_1, \ldots, y_n)$ and its standard symplectic
form $\omega_0 = \sum_{i=1}^n dx_i \wedge dy_i.$
Let $B^{2k}(R)$ denote the closed ball of radius $R$ in $\RR^{2k}$.  Gromov's Nonsqueezing Theorem states
that there is no symplectic embedding of $B^{2n}(1)$ into $B^{2}(R) \times \R^{2(n-1)}$ for $R<1$, \cite{gr}. We are interested here in intermediate
nonsqueezing  phenomena for domains with a large factor. Accordingly, we will assume throughout that  $n \geq 3$. The basic motivating  problem
for this work, which remains open, is the following.

\medskip
\noindent{\bf Question 1.}
{\em What, if any, is the smallest value of $R>0$ such that there exists a symplectic embedding of $B^2(1) \times \RR^{2(n-1)}$ into $B^4(R) \times \RR^{2(n-2)}$? }
\medskip

\noindent  Prior to the current paper, the most that could be said is that if such an emebedding exists, then $R$ must be at least $\sqrt{2}$. This bound is implied by the second Ekeland-Hofer capacity from \cite{ekehof}.

In \cite{guth}, L. Guth constructs new symplectic embeddings of polydiscs which represent a major breakthrough in our understanding of the following
{\em bounded} version of Question 1.

\medskip
\noindent{\bf Question 2.}
{\em What, if any, is the smallest value of $R>0$ such that there exists a symplectic embedding of $B^2(1) \times B^{2(n-1)}(S)$ into $B^4(R) \times \RR^{2(n-2)}$  for arbitrarily large $S>0$?}
\medskip

\noindent Among other things, Guth's work settles the existence issue here. \footnote{Before Guth's work it was commonly thought that
the answer to Questions 1 and 2 was that  no such $R$ exists.}
In the setting of Question 2, the second Ekeland-Hofer capacity again implies that $R$ must be at least  $\sqrt{2}$.
The following improvement of this bound is the main result of this paper.

\begin{theorem}
\label{thm} For any $0<R< \sqrt{3}$ there are no symplectic embeddings of
$B^2(1) \times B^{2(n-1)}(S)$ into $B^4(R) \times
\RR^{2(n-2)}$ when $S$ is sufficiently large.
\end{theorem}

In fact we prove a slightly stronger result.  For convenience, identify $\RR^{2n}$ with $\CC^n$ using
the complex coordinates $z_j = x_j +iy_j$. Let
$$E(1,S, \dots S)=\left\{(z_1, \ldots z_n) \in \CC^n \bigmid |z_1|^2  + \frac{|z_2|^2}{S^2} + \ldots +\frac{|z_n|^2}{S^2} \leq1 \right\}$$
and denote by $\CC P^2(R)$  the complex projective plane
equipped with the symplectic form $R^2\omega_{FS}$ where  $\omega_{FS}$ is the standard Fubini-Study symplectic form.

\begin{theorem}
\label{thm2} For any $0<R< \sqrt{3}$ there are no symplectic embeddings of
$E(1,S, \dots S)$ into $\CC P^2(R)
\times \RR^{2(n-2)}$ when $S$ is sufficiently large.
\end{theorem}

\bigskip

\noindent{\bf A brief outline of the proof.}
As in \cite{gr},  Theorem \ref{thm2} is proved via an existence theorem for certain holomorphic curves.
Suppose that for any $S>0$ there  is a symplectic embedding  $\phi(S)$ of $E(1,S, \dots, S)$ into $\CC P^2(R)
\times \RR^{2(n-2)}$.
It suffices to show that  for every positive integer $d< S/\sqrt{3}$, there exists a holomorphic plane of degree $d$
in the (negative) symplectic completion of  $(\CC P^2(R)
\times \RR^{2(n-2)}) \smallsetminus \phi(E(1,S, \dots, S))$ whose negative end covers  the
shortest simple Reeb orbit on the boundary of  $\phi(E(1,S, \dots, S))$ a total of $3d-1$ times. In particular, the
symplectic area of such a curve is both positive and equal to $d\pi R^2 - (3d-1)\pi$ and so the existence of these curves implies that $R^2 > (3d-1)/d$
for all $d>0$.

To prove this existence result we use a  cobordism argument to reduce it to an equivalent
problem for curves in a four dimensional symplectic manifold. Starting with well known results concerning holomorphic spheres of degree $d$ in a blow-up of $\CC P^2(R)$,  we then
settle this alternative existence problem
using  automatic regularity theorems,  the compactness
theorem for splittings from \cite{BEHWZ}, and several new techniques, many of which involve families of holomorphic curves
with varying point constraints. In fact, we prove the existence problem in dimension four first,  and then we apply it in the manner described above.

\bigskip

\noindent{\bf Sharpness.}
We also prove that Theorem \ref{thm} is sharp in the following sense.

\begin{theorem}\label{embed2} For any $R>\sqrt{3}$ there exist symplectic embeddings of
$B^2(1) \times B^{2(n-1)}(S)$ into $B^4 (R)\times \RR^{2(n-2)}$ for all $S>0$.
\end{theorem}

\noindent
The proof of Theorem \ref{embed2} involves a refinement of Guth's  embedding procedure from \cite{guth}.

\begin{remark}
We do not know whether there exists a symplectic embedding of $B^2(1) \times B^{2(n-1)}(S)$ into $B^4 (\sqrt{3})\times \RR^{2(n-2)}$ for all large $S$.\footnote{This has since been settled by Pelayo and V\~{u} Ngoc who construct a symplectic embedding of $B^2(1) \times \mathbb{R}^{2(n-1)}$ into $B^4 (\sqrt{3})\times \RR^{2(n-2)}$ in \cite{pn}.}
\end{remark}

\subsection{A related result}

Our approach also allows us to settle the following related problem.

\medskip
\noindent{\bf Question 3.}
{\em What are  the smallest values of $R_1\leq R_2$ for which there are  symplectic embeddings of
$B^2(1) \times B^{2(n-1)}(S)$ into $B^2(R_1) \times B^2(R_2)
\times \RR^{2(n-2)}$ for all $S>0$?}
\medskip

More precisely, by compactifyng  $B^2(R_1) \times
B^2(R_2)$ to  $\CC P^1(R_1) \times \CC P^1(R_2)$, and replacing the study of
holomorphic curves of high degree, say $d$, in $\CC P^2$ by curves
of degree  $(d,1)$ in $\CC P^1 \times \CC P^1$, an identical
argument to the one used to prove Theorem \ref{thm} yields the following result.

\begin{theorem}\label{eh} If $R_1 < \sqrt{2}$ then there are no symplectic embeddings
of $B^2(1) \times B^{2(n-1)}(S)$ into $B^2(R_1)\times
B^2(R_2) \times \RR^{2(n-2)}$ when $S$ is sufficiently large.
\end{theorem}

This restriction is again stronger than those imposed by  the Ekeland-Hofer capacities which imply that  no such embeddings
exist, for  large enough $S>0$,  when $R_1< 1$. Moreover, Guth's embedding results again imply (this time directly) that
Theorem \ref{eh} is sharp in the
following sense.

\begin{theorem}(\cite{guth})\label{embed} For any $R>\sqrt{2}$ there exist symplectic embeddings
of  $B^2(1) \times B^{2(n-1)}(S)$ into $B^2(R)\times
B^2(R)\times \RR^{2(n-2)}$ for all $S>0$.
\end{theorem}

\subsection{Organization} The next section contains some background material  and the proof of  the crucial  existence result, Theorem \ref{key}, which concerns  special
 holomorphic curves in certain four dimensional symplectic manifolds. Theorem \ref{thm2} is then proved in Section $3$. Our proof of Theorem \ref{eh} is completely similar
and is thus omitted. In Section $4$, we construct the embeddings of Theorems \ref{embed2} and \ref{embed}.

\subsection{Acknowledgements} The authors would like to thank Dusa McDuff for helpful comments and suggestions. We also thank an anonymous referee for pointing out  gaps in previous proofs of Proposition \ref{path2} and Proposition \ref{compactt}, as well as for many other helpful remarks  incorporated into the text.

\section{On holomorphic curves in dimension four} \label{dim4}

Throughout this section we fix a real number $R>1$ and a positive integer $d$. As described below, we also fix a symplectically embedded ellipsoid $E$  in $\CC P^2(R)$, the complex projective plane
equipped with the symplectic form $R^2 \omega_{FS}$. The main result in this section is Theorem \ref{key} which establishes the existence of certain holomorphic planesÊ of degree $d$  in the negative
symplectic completion of $\CC P^2 \smallsetminus E$. Starting with a well known moduli space of seed curves in $\CC P^2 \# \overline{\CC P^2}$,  the symplectic manifold obtained from $\CC P^2(R)$ by blowing up a ball inside of $E$, we detect the desired curves as part of the possible limits which occur when one {\it stretches the neck} along the boundary of $E$.

\subsection{An embedded ellipsoid} \label{ellipsoid}

Let us first recall some basic facts concerning standard ellisoids in $\RR^4 =\CC^2$. For $a<b$, let $E(a,b)$ denote the ellipsoid
$$E(a,b)=\left\{(z_1,z_2) \in \CC^2 \bigmid \frac{|z_1|^2}{a^2} + \frac{|z_2|^2}{b^2} \leq1 \right\}.$$ If $\frac{a^2}{b^2} \in \RR \smallsetminus \Q$,
then the standard Liouville form on $\CC^2$ restricts to
$\partial E(a,b)$ as a contact form $\alpha_{E(a,b)}$ which has exactly two simple Reeb orbits, $\gamma_1$ and $\gamma_2$, which lie in the planes $\{z_2=0\}$ and
$ \{z_1=0\}$, respectively. The action of $\gamma_1$ is
 $\pi a^2$ and the action of $\gamma_2$ is $\pi b^2$.
The Conley-Zehnder indices of these orbits, with respect to the natural
trivialization of the $z_1$ and $z_2$
coordinate planes, are
given by $\mu(\gamma_1) =3$ and $\mu(\gamma_2)= 2+2 \Big\lfloor
\frac{b^2}{a^2} \Big\rfloor +1$. More generally, if $\gamma^{(r)}_j$ is the $r$-fold cover of
$\gamma_j$, then we have
\begin{equation}
\label{gamma1}
\mu(\gamma^{(r)}_1) =2r +2\Big\lfloor  \frac{ra^2}{b^2}  \Big\rfloor +1
\end{equation}
and
\begin{equation}
\label{gamma2}
\mu(\gamma^{(r)}_2) =2r +2\Big\lfloor \frac{rb^2}{a^2}  \Big\rfloor +1.
\end{equation}
These trivialization along closed Reeb orbits  will be used throughout, see Section \ref{finer} below.

Now for $S>\sqrt{2}$ it follows from Theorem 2 of \cite{schl} that  the ellipsoid $E(1/S,1)$ can be symplectically embedded into the interior of $B^4(1)$ and hence $\CC P^2(R)$.
For a fixed degree $d$ we now fix such an embedding of $E(1/S,1)$ for an $S$ satisfying
\begin{equation}
\label{tford}
S > \sqrt{3d}.
\end{equation}
In what follows, we will denote $E(1/S,1)$ by $E$ and will identify
$E$ with its image in $\CC P^2(R)$.

\subsection{Moduli spaces of seed holomorphic curves}\label{ddeeff}

For an $r< 1/S$, let  $(\CC P^2 \# \overline{\CC P^2}, \omega_{R,r})$ be the symplectic manifold obtained from $\CC P^2(R)$ by blowing up the ball $B^4(r)$ inside $E$.
\footnote{The choice of the radius $r$ of this ball  will be used explicitly later, see Lemma \ref{h}.}
Let ${\mathcal J}$ be the space of smooth  almost-complex
structures on $\CC P^2 \# \overline{\CC P^2}$  which are tame with respect to  $\omega_{R,r}$ and denote the homology class of the exceptional divisor in $\CC P^2 \# \overline{\CC P^2}$ by $\mathcal{E}$ and the homology class of a line in $\CC P^2$ by $\mathcal{L}$. 

For a fixed $J$ in ${\mathcal J}$ and an ordered collection of  points $p_1, \dots, p_{2d}$ in $\CC P^2 \# \overline{\CC P^2}$,
we consider the moduli space ${\mathcal M}_d(J, p_1, \dots, p_{2d})$ defined as
$$
\Big\{ (f,(y_1, \dots,y_{2d}))  \mid \bar{\partial}_{\scriptscriptstyle{J}}f =0,\,
f(y_i)=p_i,\, [f] = d\mathcal{L} - (d-1)\mathcal{E}\Big\} / G,
$$
where $f$ is in $C^{\infty}(S^2,\CC P^2 \# \overline{\CC P^2})$,  $(y_1, \dots y_{2d})$ is a $2d$-tuple of pairwise distinct points in $S^2$, and $G$ is the
reparameterization group $\PSL(2,\CC)$ of the domain.

We will also consider the moduli space of unconstrained holomorphic spheres associated to an almost-complex structure $J \in {\mathcal J}$ and a homology class $A \in H_2(\CC P^2 \# \overline{\CC P^2})$ defined as $${\mathcal M}_A(J) =
\Big\{ f \in C^{\infty}(S^2,\CC P^2 \# \overline{\CC P^2})  \mid \bar{\partial}_{\scriptscriptstyle{J}}f =0,\, [f] = A \Big\} / G.
$$

\begin{definition} \label{general1}
A collection of points $\{p_i\}_{i=1}^{2d}$ is said to be in {\it general position relative to $J$} if no somewhere injective $J$-holomorphic sphere of virtual index $2k$ has image intersecting more than $k$ of the points. In particular, there are no somewhere injective $J$-holomorphic spheres of negative virtual index.
\end{definition}

The main result of this section is
\begin{proposition} \label{compactCP23}
Suppose that the  points $\{p_i\}_{i=1}^{2d}$ are in general position relative to $J$. Then the moduli space ${\mathcal M}_d(J, p_1, \dots, p_{2d})$ consists of a single class represented by an embedded, regular holomorphic sphere.
\end{proposition}


\subsubsection{The proof of Proposition \ref{compactCP23} }

We begin with a few technical preliminaries.



\begin{proposition} \label{closedindex} The space ${\mathcal M}_d(J, p_1, \dots, p_{2d})$ has virtual dimension $0$.
\end{proposition}

\begin{proof}
The index formula for closed holomorphic curves is well known, see for example the book \cite{msa}.
\end{proof}

\begin{lemma} \label{surjective}
Suppose that the  points $\{p_i\}_{i=1}^{2d}$ are in general position relative to $J$. If $f$ is any curve representing a class in the space ${\mathcal M}_d(J, p_1, \dots, p_{2d})$ then $f$ is embedded and the linearized Cauchy-Riemann operator at $f$ is surjective.
The canonical orientation of every class in  ${\mathcal M}_d(J, p_1, \dots, p_{2d})$ is $+1$.
\end{lemma}

\begin{proof}
As the class $d\mathcal{L}- (d-1)\mathcal{E}$ is primitive, holomorphic spheres representing this homology class are necessarily somewhere injective.

The adjunction formula, Theorem $2.6.3$ in \cite{msa},  now implies that $f$ is actually embedded. It now follows from automatic regularity, see \cite{hls} or Lemma $3.3.3$ of \cite{msa}, that $f$ is regular in ${\mathcal M}_d(J, p_1, \dots, p_{2d})$. Indeed, as shown in Section $3$ of \cite{hls}, any element $s$ of the kernel of the relevant linearized operator can be represented as a section of the normal bundle $\nu$ of the image of $f$ in $\CC P^2 \# \overline{\CC P^2}$ which is {\it holomorphic} with respect to a certain complex structure for which the zero section is also holomorphic. But as the normal bundle has Chern class $c_1(\nu) = 2d-1$ and $s$ has zeros at $p_1, \dots, p_{2d}$, it follows from positivity of intersection that $s$ must vanish identically. Thus the kernel is trivial and as the virtual index is $0$ the curve is regular. We thank the referee for this argument.


For $f$ as above, let $D_f$ denote the Cauchy-Riemann operator at $f$. The canonical orientation of the class $[f] \in {\mathcal M}_d(J, p_1, \dots, p_{2d})$
is determined by the $\mod 2$ spectral flow of a family of  linear operators  of the form $D_f + P^t$ where the $P^t$ are compact and the end point $D_f + P^1$ is complex linear and bijective, see Remark $3.2.5$ of \cite{msa}. The automatic regularity argument above applies to the operators $D_f + P^t$ as well. Since these operators all have Fredholm index zero,  their kernels must all be trivial and hence the spectral flow has no crossings. From this is follows that the  canonical orientation of any class $[f]$ is $(-1)^0=+1$.
\end{proof}

\begin{lemma} \label{regular1} Assume that $J \in \mathcal{J}$ is regular for every moduli space of holomorphic spheres
${\mathcal M}_A(J)$ \footnote{The set of such almost-complex structures  is of the second category, see for instance \cite{msa}, Theorem $3.1.5$.}, that is, the linearized Cauchy-Riemann operator is surjective at all somewhere injective curves.
The sets of $2d$-tuples of points which are not in general position  relative to $J$ is of codimension $2$. In other words, such sets of points in $(\CC P^2 \# \overline{\CC P^2})^{2d}$ lie in the image of countably many smooth maps from manifolds of dimension at most $8d-2$.
\end{lemma}

\begin{proof}
Let $\mathcal{B}$ be a moduli space of somewhere injective holomorphic spheres of virtual index $2k$. Let $\mathcal{B}^{k+1}$ denote the corresponding moduli space of curves equipped with $k+1$ marked points. There is a smooth evaluation map $\mathcal{B}^{k+1} \to (\CC P^2 \# \overline{\CC P^2})^{k+1}$. Since $J$ is regular, if $\mathcal{B}$ is nonempty then $k \ge 0$ and $\mathcal{B}^{k+1}$ has dimension $2k+2(k+1)$ while $(\CC P^2 \# \overline{\CC P^2})^{k+1}$ has dimension $4(k+1)$. The points $\{p_i\}_{i=1}^{2d}$ are not in general position only if some subset of size $k+1$ lies in the image of one of these evaluation maps, and so our result follows.
\end{proof}

An identical argument gives the following. For a family of almost-complex structures $\{J_t\}$ and for $t \in [0,1]$ we  define the moduli spaces $${\mathcal M}_A(\{J_t\}) =
\Big\{ ([f],t)  \mid [f] \in {\mathcal M}_A(J_t) \Big\}.
$$ We now consider tuples of points of the form  $$(p_1, \dots, p_{2d} ,t)  \in (\CC P^2 \# \overline{\CC P^2})^{2d} \times [0,1].$$

\begin{lemma} \label{regular11}Assume that $\{J_t\}$ is regular for every moduli space of holomorphic spheres
${\mathcal M}_A(\{J_t\})$. \footnote{Again, the set of such families of almost-complex structures  is of the second category, see for instance \cite{msa}, Theorem $3.1.7$.}
The set of points $(p_1, \dots, p_{2d} ,t)$ such that $\{p_i\}_{i=1}^{2d}$ is not in general position  relative to the corresponding $J_t$ is of codimension $2$ inside $(\CC P^2 \# \overline{\CC P^2})^{2d} \times [0,1]$. In other words, these points in $(\CC P^2 \# \overline{\CC P^2})^{2d} \times [0,1]$ lie in the image of countably many smooth maps from manifolds of dimension at most $8d-1$.
\end{lemma}

\begin{corollary} \label{regularcor11}
The set of maps $\bar{p} \in C^{\infty} ([0,1], (\CC P^2 \# \overline{\CC P^2})^{2d})$ such that $\bar{p}(t) = \{p_i(t)\}_{i=1}^{2d}$ is in general position relative to $J_t$ for all $t$, is of second category.
\end{corollary}

\begin{lemma} \label{compactCP21}
Suppose that the points $\{p_i\}_{i=1}^{2d}$ are in general position relative to $J$. Then the moduli space ${\mathcal M}_d(J, p_1, \dots, p_{2d})$ is compact. That is, any sequence of holomorphic curves  representing classes in ${\mathcal M}_d(J, p_1, \dots, p_{2d})$ has a convergent sequence which converges $C^{\infty}$-uniformly modulo reparameterization.
\end{lemma}

\begin{proof}
Holomorphic spheres representing the homology class $d\mathcal{L} - (d-1)\mathcal{E}$ have unconstrained deformation index $4d$. By Gromov compactness, for any sequence as above there exists a subsequence converging in a certain sense to a cusp curve. This cusp curve consists of a number, say $l$, of holomorphic spheres intersecting at nodal points. Suppose that there are $N$ of these nodes. Letting $2k_i$ be the deformation index of the $i$th sphere, we then have the formula $$2N+\sum_{i=1}^l 2k_i = 4d.$$ Now, each of our constraint points must lie on at least one of the holomorphic spheres, and each of the spheres multiply covers a somewhere injective sphere of lower index (strictly lower if the cover is nontrivial). Therefore by the assumption of general position we have $\sum_{i=1}^l k_i \ge 2d$. Combining the two formulas we see that $N=0$ and so the cusp curve is in fact a single (somewhere injective) holomorphic sphere and Gromov compactness gives $C^{\infty}$-uniform convergence as required.
\end{proof}

Similarly we have the following.

\begin{lemma} \label{compactCP22}
Suppose that a sequence of point constraints $\{p^{(j)}_i\}_{i=1}^{2d}$ converges to $\{p_i\}_{i=1}^{2d}$, and let $C_j$ be a class in ${\mathcal M}_d(J, p^{(j)}_1, \dots, p^{(j)}_{2d})$ for all $j$. Then if the $\{p_i\}_{i=1}^{2d}$ are in general position for $J$ a subsequence of the $C_j$ converges to a  class  $C \in {\mathcal M}_d(J, p_1, \dots, p_{2d})$.
\end{lemma}

At this point we continue with the proof of Proposition \ref{compactCP23}.  Choose $J$ and the constraints $\{p_i\}_{i=1}^{2d}$ as in the statement of  Proposition \ref{compactCP23}. By Lemma \ref{surjective} and  Lemma \ref{compactCP21} there are finitely many classes in ${\mathcal M}_d(J, p_1, \dots, p_{2d})$, all of which have canonical orientation $+1$.   We now show that the  number of these classes is independent of the choice of $J$ and the constraint points. Let $J'$ and  $\{p'_i\}_{i=1}^{2d}$ be another
good set of data. Let $\{J_t\}$ be a smooth family of regular  almost complex structures, as in Lemma \ref{regular11}, such that  $J_0=J$ and $J_1=J'$, and let $\bar{p} \in C^{\infty}([0,1], (\CC P^2 \# \overline{\CC P^2})^{2d})$ be such that $\bar{p}(0)=\{p_i\}$ and $\bar{p}(1)=\{p'_i\}$.
Set
$${\mathcal N}(J_t, \bar{p})=\{(C,t)|C \in {\mathcal M}_d(J_t, \bar{p}(t))
\text{ and } t \in [0,1]\}.$$
We may assume that $\bar{p}$ and $J_t$ are chosen such that ${\mathcal N}(J_t, \bar{p})$ is a smooth oriented $1$-dimensional manifold.  By Corollary \ref{regularcor11} we may also assume that $\bar{p}(t)$ is in general position relative to $J_t$ for all $t \in [0,1]$, and so by Lemma \ref{compactCP22}  the space ${\mathcal N}(J_t, \bar{p})$ is compact. Hence, ${\mathcal N}(J_t, \bar{p})$ is an oriented cobordism between ${\mathcal M}_d(J_t, \bar{p}(0))$ and ${\mathcal M}_d(J_t, \bar{p}(1))$. The  number of classes in these spaces, when counted with sign, must therefore be equal. Since these signs are all $+1$,  these moduli spaces  have the same absolute number of classes.

To show that this number of classes is one we consider the case when  $J$ is integrable. Then $\CC P^2 \# \overline{\CC P^2}$ is a holomorphic $\CC P^1$-bundle over a $\CC P^1$ of degree $1$,  i.e., there is a holomorphic projection to the line at infinity in $\CC P^2$ whose fibres are lines intersecting the exceptional divisor. There are meromorphic sections of this bundle with $d$ zeros (corresponding to intersections with the line at infinity) and $d-1$ poles (corresponding to $d-1$ intersections with the exceptional divisor). These are holomorphic spheres in our homology class. Picking our $2d$ points to lie on such a curve we have a nonempty moduli space. But the self-intersection number of curves in this class is $2d-1$ and so by positivity of intersection our section will in fact be the unique curve passing through these points as required.

\subsection{Stretching the neck along $\partial E$} \label{split}


Recall that we have a symplectically embedded ellipsoid $E \subset B^4(R) \subset \CC P^2(R)$ and we have blown-up a ball $B^4(r) \subset E$, where $r<1/S$, to obtain the manifold $\CC P^2 \# \overline{\CC P^2}$. We now choose our constraint points $p_i$ to be in $E \smallsetminus B^4(r) \subset \CC P^2 \# \overline{\CC P^2}$, split  the manifold $\CC P^2 \# \overline{\CC P^2}$ along the boundary $\partial E$, as in \cite{BEHWZ}, and analyze the limits of holomorphic spheres representing  classes in ${\mathcal M}_d(J, p_1, \dots, p_{2d})$ under this process. In the present section we describe the relevant details of the splitting procedure and the immediate implications of the compactness theorem of \cite{BEHWZ}. In the section which follows  we refine these compactness results using some of the special features of our setting.

It will useful to first restrict the class of of almost complex structures on $\CC P^2 \# \overline{\CC P^2}$ that we consider. Let $\Sigma$ be the exceptional divisor in $\CC P^2 \# \overline{\CC P^2}$ and let $\CC P^1(\infty)$ be the line at infinity in $\C P^2(R)$. Denote by $\mathcal{J}^{\star}$ the subset of $\mathcal{J} $ consisting  of almost-complex structures  for which $\Sigma$ and $\CC P^1(\infty)$ are complex, and which are standard in a fixed open  neighborhood $U_{\Sigma}$ of $\Sigma$.
In the remainder of Section 2 we will restrict ourselves to the almost complex structures in $\mathcal{J}^{\star}$. The fact that $\Sigma$ and $\CC P^1(\infty)$ can now be assumed to be holomorphic will allow us to use the positivity of intersection theorem to control their intersections with other holomorphic curves. This strategy is used several times (see, for example, the first paragraph of Section \ref{propindex}).
The restriction on the neighborhood of $\Sigma$ is used in the proofs of Lemma \ref{rclose} and Proposition \ref{oneend}. It is also important to note that in restricting ourselves to $\mathcal{J}^{\star}$ we are not introducing any new difficulties. For example, the various genericity statements in Section \ref{ddeeff} continue to hold when $\mathcal{J}$ is replaced by  $\mathcal{J}^{\star}$  since, when thought of as holomorphic curves, $\Sigma$ and $\CC P^1(\infty)$ are automatically regular, and by positivity of intersection multiple covers of $\Sigma$ itself are the only curves contained in the fixed neighborhood of $\Sigma$. Moreover, as $\partial E$ is disjoint from both $\Sigma$ and $\CC P^1(\infty)$, the restriction to $\mathcal{J}^{\star}$ does not influence the neck stretching procedure, the details of which we now recall.



Let $X$ be the vector field defined near $\partial E$ as  the symplectic dual
of the Liouville form. An almost-complex structure $J$ on $\CC P^2 \# \overline{\CC P^2}$ is said to be compatible with $E$ if the contact structure
$\{ \alpha_E = 0\}$ on $\partial E$ is equal to $T(\partial E) \bigcap  JT(\partial E)$, and $J X$ is equal to the
Reeb vector field of $\alpha_E$. Denote by ${\mathcal J}_E^{\star}$ the set of $J \in {\mathcal J}^{\star}$ which are compatible with $E$.

For every natural number $N \in \mathbb{N}$, the union of the three pieces
$$ \left( E \# \overline{\CC P^2}, e^{-N}\omega_{R,r}\right), \, \left( \partial E \times [-N, N], d(e^{\tau} \alpha_E)\right),\, \text{ and } \left( \CC P^2 \smallsetminus E , e^{N}\omega_{R,r}\right),$$
attached along their appropriate boundary components, is a symplectic manifold $((\CC P^2 \# \overline{\CC P^2})^N, \omega_{R,r}^N)$. We recall that $\alpha_E$ is the restriction of the standard Liouville $1$-form to $\partial E$, see section \ref{ellipsoid}, and $\tau$ here denotes the coordinate on $[-N,N]$. For a $J$ in ${\mathcal J}_E^{\star}$, let $J^N$ be the
continuous  almost-complex structure on  $(\CC P^2 \# \overline{\CC P^2})^N$ which equals $J$ on the disjoint union of $(\CC P^2 \smallsetminus E) $ and $E \# \overline{\CC P^2}$, and is translation invariant on $\partial E \times [-N, N]$. To make each of the $J^N$ smooth one must perturb $J$ near $\partial E$.  As noted in \S 3.4 of \cite{BEHWZ}, the choice of this perturbation  is irrelevant for the compactness theorem below.  Accordingly, we will henceforth assume that this choice has been made, and that the almost-complex structures $J^N$ are smooth.

Fix  constraint points $p_i$  in $E \smallsetminus B^4(r) \subset \CC P^2 \# \overline{\CC P^2}$ and a $J$ in ${\mathcal J}_E^{\star}$. For each $N \in \mathbb{N}$ let $C_N$  be a class in ${\mathcal M}_d(J^N, p_1, \dots, p_{2d})$ and fix a representative curve $f_N$ for $C_N$. The following compactness theorem is proved by Bourgeois, Eliashberg, Hofer, Wysocki and Zehnder in \cite{BEHWZ}.
\begin{theorem}(Theorem 10.6 of \cite{BEHWZ}) \label{ccc}
There exists a subsequence of the $f_N$ which  converges to a
holomorphic building, $\mathbf{ F}$.
\end{theorem}
We now describe the aspects of  this theorem which are relevant to our purposes. The reader is referred to \cite{BEHWZ} for the precise definitions, statements and proofs
of the general compactness theorem for holomorphic curves under splittings.

\medskip

\noindent{\bf The limits.}  We begin by describing a  holomorphic building $\mathbf{F}$  which arises as the limit of a subsequence of the  curves $f_N$ (see Chapter 9 of \cite{BEHWZ}).
The domain of  $\mathbf{F}$ is  a nodal Riemann sphere  $(\mathcal{S},j)$ with punctures. The building $\mathbf{F}$ then consists of a collection of finite energy holomorphic maps from the collection of
punctured spheres of $\mathcal{S}\smallsetminus \{\mathrm{nodes}\}$ to one of the
three symplectic manifolds:
\begin{itemize}
\item
$
\left( E^{\infty}_+ , \omega^{\infty}_+\right)=  \left( E \# \overline{\CC P^2}, \omega_{R,r}\right) \cup \left( \partial E \times [0,\infty), d (e^{\tau} \alpha_E)\right),
$
 \item
$
(SE, \omega_{SE})=\left( \partial E \times \mathbb{R}, d(e^{\tau} \alpha_E)\right),$
\item
 $\left( (\CC P^2 \smallsetminus E)^{\infty}_- , \omega^{\infty}_-\right)= \left(\CC P^2 \smallsetminus E , \omega_{R,r} \right) \cup  \left( \partial E \times (-\infty, 0], d (e^{\tau} \alpha_E)\right).
$
\end{itemize}
These target manifolds are equipped with compatible almost-complex structures  which are induced by $J$ and are translation
invariant on the subsets which are cylinders over $\partial E$.
The curves of $\mathbf{ F}$ are then  holomorphic with respect to these almost-complex structures and the complex structures on
the punctured spheres  induced by the structure $j$ on $\mathcal{S}$.

Since each curve of $\mathbf{ F}$ has finite energy, they are all asymptotically cylindrical near each of their punctures, to some multiple of either $\gamma_1$ or $\gamma_2$. If $F$ is a curve of $\mathbf{F}$ with image in $(\CC P^2 \smallsetminus E)^{\infty}_-$, then the punctures in its
domain  are all negative, i.e  as $z\in S^2$ approaches a puncture on the domain of $F$,  $F(z)$ takes values in
$\partial E \times (-\infty , 0]$ and its $(-\infty,0]$-component  converges to $-\infty$.
Similarly,  each  curve  of $\mathbf{F}$ with image in $E^{\infty}_+$  has only  positive punctures, and
curves  of $\mathbf{F}$ with  image  in $SE$ have both negative and positive punctures (but by the maximum principle not
only negative punctures).

The limiting  building $\mathbf{F}$ is also equipped with a {\it level structure}.
For a building  $\mathbf{F}$ of level $k$, this structure is encoded
by a labeling of the punctured Riemann spheres of  $\mathcal{S}\smallsetminus \{\mathrm{nodes}\}$
by  integers from $0$ to $k+1$, called levels,  such that the levels of two components which share
a node,  differ at most by $1$. Let $\mathcal{S}_r$ denote  the  union of components of level $r$ and denote by
$v_r$ the holomorphic curve of $\mathbf{ F}$ with (possibly disconnected) domain $\mathcal{S}_r$.  Then $v_0:\mathcal{S}_0\to E^{\infty}_+$,
$v_r:\mathcal{S}_r\to SE$, for $1\le r \le k$, and
$v_{k+1}:\mathcal{S}_{k+1}\to (\CC P^2 \smallsetminus E) ^{\infty}_-$.   Moreover, each node shared by
$\mathcal{S}_r$ and $\mathcal{S}_{r+1}$ is a positive puncture for $v_r$ and a negative puncture for $v_{r+1}$, each asymptotic to
the same Reeb orbit. As well,  $v_r$ extends continuously
across each node within $\mathcal{S}_r$. As part of the definition of a limiting building from \cite{BEHWZ} it is assumed that none of the curves $v_r$, for $1\le r \le k$, consist entirely of trivial cylinders over Reeb orbits. (Although, $\mathbf{F}$ itself can include some trivial cylinders.)
With this, the level structure of a specific limit is well defined (whereas the limit itself is not, see Remark \ref{uniquelimit}).

Lastly, we recall that the curves of $\mathbf{ F}$  have  two collective properties. The first of these is the fact that the sum, over components, of their virtual indices
is equal to $0$, the deformation  index of the curves $f_N$. (Formulas for the virtual indices of the curves of  $\mathbf{ F}$ are described in detail in Section \ref{finer}.)
To state the second collective property,  we must  first recall the definition of a
{\it compactifcation} of a curve of  $\mathbf{ F}$.  This definition  depends on the target of the curve.  For a curve $G$ of $\mathbf{ F}$  with image in $SE=\partial E \times \mathbb{R}$ one can write $G =(g,a)$ where
$g$ maps the domain of $G$ to $\partial E$.  The map $g$ then extends to a continuous map $\overline{G}$, the compactification of $G$,
which takes the oriented blow-up of the domain of $G$ to $\partial E$ such that the circle corresponding to each puncture is mapped to the closed Reeb orbit
on $\partial E$ which determines the asymptotic behavior of $G$ near that puncture, (see Section $4.3$ of \cite{BEHWZ} for a description of the oriented blow-up, and Proposition 5.10 of \cite{BEHWZ} for the asymptotic behavior).
Let  $F$ be a curve of $\mathbf{F}$ with image in $(\CC P^2 \smallsetminus E)^{\infty}_-$. As described in Section 3.2 of \cite{BEHWZ}, one can identify
$(\CC P^2 \smallsetminus E) ^{\infty}_-$ with $\CC P^2 \smallsetminus E$
via a diffeomorphism $\Psi_-$ which is the identity map away from an arbitrarily small tubular neighborhood of
$\partial(\overline{\CC P^2 \smallsetminus E})$ in $\overline{\CC P^2 \smallsetminus E}$.
One can then extend the map $\Psi_- \circ F$ to a  smooth map $\overline{F}$  which takes the
oriented blow-up of the domain of $F$ to $\overline{\CC P^2 \smallsetminus E}$ such that  each boundary circle
goes to the appropriate closed Reeb orbit on $\partial E$. A choice of this extension $\overline{F}$ is a compactification of $F$. Similarly, for a curve $H$ of $\mathbf{ F}$  with image in
$E ^{\infty}_+$, one can use a diffeomorphism
$
\Psi_+ \colon E ^{\infty}_+ \to E\# \overline{\CC P^2} \smallsetminus \partial E
$
to define a compactification $\overline{H}$  of  $H$ as a smooth extension of $\Psi_+ \circ H$ which takes the oriented blow-up of the domain of $H$
to $E \# \overline{\CC P^2}$, and again takes each  boundary circle  to the corresponding  closed Reeb orbit on $\partial E$. If one fixes a compactification for each of the curves in $\mathbf{F}$, then these
maps must fit together to form a continuous map $ \overline{\mathbf{ F}} \colon S^2 \to \CC P^2 \# \overline{\CC P^2}$. The map is smooth away from the boundary circles and there $\omega_{R,r}$ pulls back to a nonnegative multiple of an area form on $S^2$. The pull-back degenerates only on the parts of $S^2$ mapping onto Reeb orbits in $\partial E$, corresponding to curves in $SE$ which cover the trivial cylinders.

Given the asymptotic convergence of finite energy holomorphic maps the following is well defined in terms of the notions described above. It gives a meaning to symplectic area directly in terms of our original symplectic form, and of course differs from the areas defined by the symplectic forms $\omega_+^{\infty}$, $\omega_{SE}$ and $\omega_-^{\infty}$ defined above.

\begin{definition} \label{symarea} The {\it symplectic area} of a finite energy holomorphic map $G=(g,a)$ with image in $SE$, $F$ with image in $(\CC P^2 \smallsetminus E)^{\infty}_-$, or $H$ with image in $E ^{\infty}_+$ is defined by $\int_{S} g^*\omega_{R,r}$, $\int_{S} \overline{F}^* \omega_{R,r}$ or $\int_{S} \overline{H}^* \omega_{R,r}$, respectively. In the first integral $\omega_{R,r}$ represents the restriction of $\omega_{R,r}$ to $\partial E$, and in all three integrals the Riemann surface $S$  denotes the domain of the map.
\end{definition}

\medskip

\noindent{\bf The convergence.}
Let $\mathbf{ F}$ be a holomorphic building of level $k$, as above,  whose domain is the
Riemann surface with nodes $(\mathcal{S},j)$. If $\mathbf{F}$ is a limit of the holomorphic spheres $f_N$ in the sense of \cite{BEHWZ},
 then  there
exist maps $\sigma_N:S^2\to \mathcal{S}$ and sequences $s^r_N\in \Bbb R$, $r=1,\ldots,k$, such that:
\renewcommand{\theenumi}{(\roman{enumi})}
\begin{enumerate}
  \item The $\sigma_N$ are diffeomorphisms except that they may collapse a finite collection of circles
in $S^2$ to nodes in $\mathcal{S}$. Moreover, $\sigma_{N*}i$ converges to $j$ away from the nodes of $\mathcal{S}$, where $i$ is the standard complex structure on $S^2$.
  \item The sequences of maps $f_N\circ \sigma_N^{-1}:\mathcal{S}_0\to E^{\infty}_+$ and
  $f_N\circ \sigma_N^{-1}:\mathcal{S}_{k+1}\to (\CC P^2 \smallsetminus E) ^{\infty}_-$ converge in the $C_{loc}^{\infty}$-topology to the maps $v_0$ and $v_{k+1}$, respectively. For $1\le r\le k$ the maps
$\psi^{s^r_N} \circ f_N \circ \sigma_N^{-1}:\mathcal{S}_r\to SE$  converge to $v_r$ in the
$C_{loc}^{\infty}$-topology where $\psi^{s^r_N}$ is the diffeomorphism of $SE = \partial E \times \R$
which translates the $\R$-component by $s^r_N$.
\end{enumerate}

Here, as is necessary, we are identifying
$\partial E \times (-N,N) \subset (\CC P^2 \# \overline{\CC P^2})^{N}$ with an increasing sequence of domains in $SE$,
$E \# \overline{\CC P^2} \cup \partial E \times (-N,N) \subset (\CC P^2 \# \overline{\CC P^2})^{N}$ with an increasing sequence of domains
in $E^{\infty}_+$, and $\CC P^2 \smallsetminus E \cup \partial E \times (-N,N] \subset (\CC P^2 \# \overline{\CC P^2})^{N}$ with an increasing
sequence of domains in $(\CC P^2 \smallsetminus E) ^{\infty}_-$.

\begin{remark}
\label{uniquelimit}
As described above the limiting building $\mathbf{F}$ of a convergent sequence of holomorphic spheres is certainly not unique. The holomorphic maps defining $\mathbf{F}$ could be reparameterized, and the maps to $SE$ of a fixed level can all be translated the same distance. As in the closed case constant maps, or ghost bubbles, could be defined on additional components of $\mathcal{S}$. Finally, in this setting, it is also possible to add additional levels to $\mathbf{F}$ consisting only of trivial cylinders, that is, maps from the punctured plane which are unbranched covers of the cylinder over a closed Reeb orbit. These ambiguities can be avoided as in \cite{BEHWZ} by working only with equivalence classes of {\it stable} buildings.

\end{remark}


%
\bigskip

\noindent{\bf Notational convention.}   Sometimes it will be useful to view a holomorphic building  $\mathbf{F}$
as a collection of holomorphic curves each of whose domains  are single components of $\mathcal{S}\smallsetminus \{\mathrm{nodes}\}$. These curves, whose domains are all punctured spheres, will be denoted by a capital letter such as $F$.  It will also be necessary to view  $\mathbf{F}$ as being comprised of curves  with a fixed level but with possibly disconnected domains. As in the previous section, such curves will be denoted by a lower case letter such as $v$.

\subsection{Finer restrictions on the limiting buildings}\label{finer}

Exploiting the special nature of the Reeb flow of $\alpha_E$ on $\partial E$, together with standard regularity results, see for example \cite{msa} Chapter $3$,  we obtain more restrictions
on the curves, with connected domains, that comprise our limiting holomorphic buildings. Along the way we prove some similar restrictions for some more general classes of curves which will be used later (see Lemma \ref{new2} and Lemma \ref{new3}).

\medskip

\noindent{\bf Virtual indices.} We first derive formulas for the virtual deformation indices of any finite energy holomorphic curve with genus zero and image in one of three possible targets
resulting from the splitting procedure. We note, in advance, the following fact which will be exploited later
repeatedly.
\begin{lemma}\label{even}
All finite energy holomorphic curves with genus zero and image in either $(\CC P^2 \smallsetminus E)^{\infty}_-$, $SE$, or  $E^{\infty}_+$ represent moduli spaces with {\bf even} virtual indices.
\end{lemma}

This will be seen directly from the index formulas given in Propositions \ref{indexCP2}, \ref{indexSE} and \ref{indexE} below.

Before computing indices we first must  fix a convention for defining the Chern number of a finite energy curve $F$ whose target is some ambient  symplectic manifold, say $X$, with cylindrical ends. Such curves do not generally give closed cycles in $X$. However, they are asymptotic at their punctures to closed Reeb orbits and  hence can be compactified to give $2$-dimensional cycles with boundary, which can always be perturbed to be defined by immersions. The definition of the Chern number here then depends upon a choice of trivialization of $TX$ along these boundary orbits. We fix these trivializations to be compatible with the one used in Section \ref{ellipsoid} to define our Conley-Zehnder indices.

To be precise, in all cases considered in this paper our Reeb orbits $\gamma$ will lie in the boundary of an ellipsoid $E$ in $\CC ^n$. Thus we have a standard symplectic trivialization of $T(X)|_{\gamma}$ coming from restriction of the standard trivialization on $\CC^n$. The linearization of the Reeb flow along $\gamma$ (extended to act trivially on the normal vector to $\partial E|_{\gamma}$) induces a family of symplectomorphisms $\eta_t \in {\mathrm Symp}(\CC^n)$ for $0 \le t \le L$, where $L$ is the length of the Reeb orbit. We then define the Conley-Zehnder index $\mu(\gamma)$ following \cite{rs}, as was done above in Section \ref{ellipsoid}. Suppose that $X$ has real dimension $2n$ and is equipped with an almost-complex structure $J$. The determinant line bundle $\Lambda^n(X,J)$ has a standard section $S$ over $E$, again using our trivialization. Then $c_1(F)$ can be defined to be the number of zeros (counted with multiplicity) of a section of $\Lambda^n(X,J)|_F$ which agrees with $S$ over $\gamma$.

\medskip

Consider now a finite energy holomorphic curve $F$  with genus zero and image in $(\CC P^2 \smallsetminus E)^{\infty}_-$.
As described above, the punctures of $F$ are all negative. Suppose that $F$ has $s^-_1 \ge 0$ negative ends asymptotic to multiples of $\gamma_1$,  and $s^-_2 \ge 0$ negative ends asymptotic to multiples of $\gamma_2$.  Say that the $i^{th}$ negative end covering  $\gamma_1$ does so $a^-_i$ times,  and the $i^{th}$ negative end covering $\gamma_2$ does so $b^-_i$ times.

\begin{proposition} \label{indexCP2} The virtual deformation index of $F$ (in the moduli space of finite energy curves with the same asymptotics, modulo reparameterization) is
\begin{equation}
\label{indexf}
\mathrm{index}(F)=-2+ 2c_1(F)-2\sum_{i=1}^{s^-_1}( a^-_i + \lfloor a^-_i / S^2 \rfloor) - 2\sum_{i=1}^{s^-_2} ( b^-_i + \lfloor b^-_i  S^2 \rfloor).
\end{equation}
\end{proposition}

\begin{remark}\label{fchern}
Applying our conventions,  the Chern number appearing in \eqref{indexf} is the same as the usual Chern number, $\langle c_1(\CC P^2,J), [\widehat{F}] \rangle$, where $[\widehat{F}]$ is the homology class represented by the cycle $\widehat{F}$ formed by gluing the  appropriate discs in $E$ to the compactification of $F$. By Poincar\'{e} duality this is just three times the intersection number of $F$ with the line at infinity in the complex projective space. Hence we will often say that a curve in $(\CC P^2 \smallsetminus E)^{\infty}_-$ with Chern number $3d$ has {\it degree} $d$.
\end{remark}

\begin{proof}
The general index formula for genus zero finite energy curves, taken for example from \cite{EGH}, is
\begin{equation}
\label{indexff}
\mathrm{index}(F)=(-1)(2-s^-_1-s^-_2) + 2c_1(F)-\sum_{i=1}^{s^-_1}\mu(\gamma_1^{(a^-_i)}) - \sum_{i=1}^{s^-_2} \mu( \gamma_2^{(b^-_i)}).
\end{equation}
By the iteration formulas \eqref{gamma1} and \eqref{gamma2}, this formula simplifies to the one in our statement.
\end{proof}

Now, let $G$ be a  finite energy curve of genus zero in $SE$.
Suppose  that $G$ has $s_1^+$ positive ends asymptotic to multiples of $\gamma_1$ with the $i^{th}$ such end covering this orbit $a^+_i$ times, and $s_2^+$ positive ends asymptotic to multiples of $\gamma_2$ with the $i^{th}$ such end covering $\gamma_2$ a total of $b^+_i$ times. Suppose also that $G$ has $s_1^-$ negative ends asymptotic to multiples of $\gamma_1$ with the $i^{th}$ such end covering this orbit $a^-_i$ times, and $G$ has $s_2^-$ negative ends asymptotic to multiples of $\gamma_2$ with the $i^{th}$ such end covering $\gamma_2$ a total of $b^-_i$ times.

\begin{proposition} \label{indexSE} The virtual deformation index of $G$ is equal to
\begin{eqnarray*}
\mathrm{index}(G) & = & 2(s_1^- + s_2^- -1) + 2\sum_{i=1}^{s_1^+}( a^+_i + \lfloor a^+_i / S^2 \rfloor) + 2\sum_{i=1}^{s_2^+}( b^+_i + \lfloor b^+_i  S^2 \rfloor) \\
{} & {} & - 2\sum_{i=1}^{s_1^-}( a^-_i + \lfloor a^-_i / S^2 \rfloor) -2 \sum_{i=1}^{s_2^-}( b^-_i + \lfloor b^-_i  S^2 \rfloor).
\end{eqnarray*}
\end{proposition}

\begin{proof}This follows from the general formulas in the same way as Proposition \ref{indexCP2}.  In this case, for our trivialization the Chern number term vanishes for curves in $SE$.
\end{proof}

Finally we consider a genus zero finite energy curve $H$ mapping to $E^{\infty}_+$. The curve now has only positive ends. Suppose that $H$ has $s^+_1$ positive ends asymptotic to multiples of $\gamma_1$ with  the $i^{th}$ such  end covering  this orbit  $a^+_i$ times and $H$ has $s^+_2$ positive ends asymptotic to multiples of $\gamma_2$ with  the $i^{th}$ such  end covering  this orbit  $b^+_i$ times. We must also specify the relative homology class of $H$. This is determined by the intersection number $H \cdot\Sigma$. With our  trivializations for the Conley-Zehnder indices the Chern number term becomes $-H \cdot \Sigma$ and we derive the following formula.

\begin{proposition} \label{indexE}
The virtual deformation index of $H$,  $\mathrm{index}(H)$, is equal to
\begin{equation*}
2(s^+_1+s^+_2-1)-2H \cdot \Sigma + 2\sum_{i=1}^{s^+_1}( a^+_i + \lfloor a^+_i / S^2 \rfloor) + 2\sum_{i=1}^{s^+_2} ( b^+_i + \lfloor b^+_i  S^2 \rfloor).
\end{equation*}
If $H$, as above, represents a class in the moduli space of finite energy curves contrained to pass through $M$ points  then its virtual index decreases by $2M$.
\end{proposition}

\medskip

\noindent{\bf Curves in our limiting buildings.} We now establish restrictions for those curves which may appear as part of the limiting buildings from Section \ref{split}, as well as
other classes of curves with similar properties (see Lemma \ref{new2} and Lemma \ref{new3}).

First, in analogy with Definition \ref{general1} we describe what is meant by saying that a collection of points in $E^{\infty}_+$ is in general position relative to the  almost-complex structure on $E^{\infty}_+$
determined by  $J \in {\mathcal J}_E^{\star}$.

\begin{definition} \label{general2}
A collection of points $\{p_i\}_{i=1}^{2d}$ in $E^{\infty}_+$ is in {\it general position relative to $J$} if no somewhere injective finite energy holomorphic curve of genus zero and  virtual index $2k$ has image intersecting more than $k$ of the points.
\end{definition}

As noted in Lemma  \ref{even} we need only consider curves of even index. We also recall that the exceptional divisor $\Sigma \subset E^{\infty}_+$ is chosen to be holomorphic and has index $0$, and so the definition implies that if the points are in general position then none of the $p_i$  lie on $\Sigma$.

\begin{lemma} \label{regular2} The set of all $J \in {\mathcal J}_E^{\star}$ which are regular for genus zero finite energy curves (that is, $J$ such that the linearized deformation operator is surjective at all somewhere injective curves) in any of our three target manifolds is a subset of the second category.
For a fixed regular $J \in {\mathcal J}_E^{\star}$ the sets of points which are not in general position is of codimension $2$.
\end{lemma}

\begin{proof}
The proof of the first statement is identical to that contained in \cite{msa}, Chapter $3$.
The proof of the second statement is identical to that of Lemma \ref{regular1}.
\end{proof}

\begin{proposition}\label{index} Let $\mathbf{ F}$ be a holomorphic building  which is the limit of a sequence of curves $f_N$ representing classes in ${\mathcal M}_d(J^N, p_1, \dots, p_{2d})$. For a regular $J$ in ${\mathcal J}_E^{\star}$, points  $\{p_i\}_{i=1}^{2d}$ in general position relative to $J$, and a blow-up radius $r$ sufficiently close to $1/S$ we have:
\renewcommand{\theenumi}{(\roman{enumi})}
\begin{enumerate}
\item Each curve  of  $\mathbf{F}$ with connected domain and image in $(\CC P^2 \smallsetminus E)^{\infty}_-$ is
somewhere injective, regular, and has deformation index equal to zero. Its (negative) ends are all asymptotic to multiples of $\gamma_1$ and, in total, they cover $\gamma_1$ at most $3d-1$ times.
 \item Each curve of  $\mathbf{F}$ with connected domain and image in $SE$ is a multiple cover of a holomorphic cylinder over $\gamma_1$, and has virtual index zero.
 \item Each curve of $\mathbf{F}$ with connected domain and image  in $E^{\infty}_+$  has deformation index equal to zero (with the point constraints), and its (positive) ends are all asymptotic to multiples of $\gamma_1$.
\end{enumerate}
\end{proposition}

\begin{remark} \label{new1}
As closed curves in $\CC P^2$ all have positive deformation index, part $(i)$ of Proposition \ref{index} implies immediately that all curves of $\mathbf{F}$ with image in $(\CC P^2 \smallsetminus E)^{\infty}_-$ must have at least one negative end, that is, not all punctures are removable singularities. If a curve with only removable singularities appears in $E^{\infty}_+$ then up to covers it must be the exceptional divisor $\Sigma$. Our proof will also exclude these possible components (see the argument at the very end of the proof). Indeed, if we remove this curve then some component of the remaining building must have a negative constrained index (since the remaining, possibly disconnected, building still has $2d$ point constraints but now its intersection number with $\Sigma$ is strictly greater than $d-1$). Curves in $E^{\infty}_+$ without marked points and which miss the exceptional divisor $\Sigma$ all have deformation index at least $2$, see Proposition \ref{indexE}, and so we see that each curve in $E^{\infty}_+$ must intersect either a constraint point or $\Sigma$.
\end{remark}

\subsubsection{The proof of Proposition \ref{index} and some related results}\label{propindex}
We start our analysis at the top level.
Let $F$ be a genus zero finite energy curve of $\mathbf{F}$ with image in $(\CC P^2 \smallsetminus E)^{\infty}_-$.
Suppose that $F$ has $s^-_1 \ge 0$ negative ends asymptotic to multiples of $\gamma_1$,  and $s^-_2 \ge 0$ negative ends asymptotic to multiples of $\gamma_2$.  Say that the $i^{th}$ negative end covering  $\gamma_1$ does so $a^-_i$ times,  and the $i^{th}$ negative end covering $\gamma_2$ does so $b^-_i$ times. We recall from Remark \ref{fchern} that the Chern number of such components, with respect to our trivialization, is just three times the intersection number with the line at infinity $\CC P^1(\infty)$. This intersection number is nonnegative for each curve of $\mathbf{F}$ (since for $J$ in $\mathcal{J}_E^{\star}$ the line at infinity, $\CC P^1(\infty)$, is holomorphic) and the sum of these intersection numbers over all the curves of $\mathbf{F}$ is $d$ (since $\mathbf{F}$ is a limit of holomorphic spheres having intersection number $d$ with $\CC P^1(\infty)$). Hence $0 \le c_1(F) \le 3d$.
As well, the bound on $S$ from \eqref{tford} implies that $\lfloor b^-_i S^2\rfloor \geq  b^-_i 3d$, and so by Proposition \ref{indexCP2}
$$
\mathrm{index}(F) \leq -2+ 2c_1(F) -2\sum_{i=1}^{s^-_1} a^-_i  -(2 +6d)\sum_{i=1}^{s^-_2}  b^-_i.
$$
These inequalities lead immediately to  the following result.
\begin{lemma}
\label{fpos}
If a holomorphic curve $F$ of $\mathbf{F}$ with image in $(\CC P^2 \smallsetminus E)^{\infty}_-$ has nonnegative virtual index, then
$c_1(F)>0$, $s^-_2 = 0$ and $\sum_{i=1}^{s^-_1} a^-_i \leq c_1(F)-1$. Moreover,
$$
\mathrm{index}(F) = -2+ 2c_1(F) -2\sum_{i=1}^{s^-_1} a^-_i .
$$
\end{lemma}
If $F$ is  somewhere injective, then regularity with respect to generic almost-complex structures follows from \cite{msa}, Chapter $3$. In this case we can immediately apply Lemma \ref{fpos}. To deal with the possibility that $F$ is not somewhere injective we require the following result.

\begin{proposition}
 \label{multiplecover}
Let $X$ be a symplectic manifold perhaps having cylindrical ends and equipped with a compatible almost-complex structure as described above.
A finite energy holomorphic curve $u: Z \to X$ is either somewhere injective or there exists a proper holomorphic map $\phi: Z \to Z'$ and a somewhere injective curve $u': Z' \to X$ such that $u= u' \circ \phi$.
\end{proposition}

\begin{proof} This follows almost exactly as in the case of closed holomorphic curves, see for example Proposition $2.5.1$ of \cite{msa}, at least if we assume that the Reeb flow is nondegenerate. Then our finite energy curves converge asymptotically to cylinders over closed Reeb orbits and in particular have only finitely many critical points, see \cite{hofa}. The proof has already been adapted to finite energy planes in \cite{hofi}, Theorem $6.2$, in which case the map $\phi$ is shown to be polynomial. At least in the nondegenerate case this proof applies equally well to finite energy curves with multiple ends and perhaps higher genus (although of course in the higher genus case $\phi$ may not necessarily be polynomial).
\end{proof}

So, if the curve $F$ is multiply covered, then it is a $p$-fold cover of a simple curve $\widetilde{F}$. Suppose that $\widetilde{F}$ has $\tilde{s}^-_1 \ge 0$ negative ends asymptotic to multiples of $\gamma_1$,  and $\tilde{s}^-_2 \ge 0$ negative ends asymptotic to multiples of $\gamma_2$.  Say that the $i^{th}$ negative end covering  $\gamma_1$ does so $\tilde{a}^-_i$ times.  It follows from the regularity
of $J$, that $\mathrm{index}(\widetilde{F}) \geq 0$. Lemma \ref{fpos} then implies that
$c_1(\widetilde{F})>0$, $\tilde{s}^-_2 = 0$, $\sum_{i=1}^{\tilde{s}^-_1} \tilde{a}^-_i \leq c_1(\widetilde{F})-1$, and
$$
\mathrm{index}(\widetilde{F}) = -2+ 2c_1(\widetilde{F}) -2\sum_{i=1}^{\tilde{s}^-_1} \tilde{a}^-_i .
$$
Hence, for $F$ we have $s^-_2=0$,  $$\sum_{i=1}^{{s}^-_1} {a}^-_i \leq p \sum_{i=1}^{\tilde{s}^-_1} \tilde{a}^-_i \leq c_1(F) -p, $$ and thus
\begin{equation}
\label{nonneg}
\mathrm{index}(F) \geq -2+ p(\mathrm{index}(\widetilde{F})+2) \geq 0.
\end{equation}
The hypothesis of Lemma \ref{fpos} therefore holds for each curve $F$ of $\mathbf{ F}$ with image in $(\CC P^2 \smallsetminus E)^{\infty}_-$ and we have the following result.
\begin{lemma}
\label{f}  Let $F$ be a curve of $\mathbf{F}$ with  image in $(\CC P^2 \smallsetminus E)^{\infty}_-$. Then the virtual index of $F$ is nonnegative and is strictly positive when $F$ is a multiple cover.
The ends of $F$ are all asymptotic to some multiple of $\gamma_1$, $c_1(F)>0$, and (in the notation above) the total multiplicity of all the negative ends of $F$ is
\begin{equation}
\label{neg}
\sum_{i=1}^{s^-_1} a^-_i \le c_1(F) -1.
\end{equation}
Futhermore, the total multiplicity of the negative ends of all such curves is at most $3d-1$, with equality only if there is a single component of $\mathbf{F}$ with  image in $(\CC P^2 \smallsetminus E)^{\infty}_-$.
\end{lemma}

\begin{proof}
Only the last assertion remains to proved. It follows immediately from inequality \eqref{neg} and the observation that if one sums the Chern number terms $c_1(F)$ over all curves $F$ appearing $\mathbf{ F}$ one gets $3d$.
\end{proof}

Arguing  as above we get the following useful variation of Lemma \ref{f} for curves not necessarily in  $\mathbf{F}$.
\begin{lemma} \label{new2}
Let $F$ be a finite energy curve  in $(\CC P^2 \smallsetminus E)^{\infty}_-$ of genus zero such that  $c_1(F)=e \le 3d$. Then the virtual index of $F$ is nonnegative and is strictly positive when $F$ is a multiple cover.
Furthermore, $c_1(F)>0$, the ends of $F$ are all asymptotic to some multiple of $\gamma_1$, and the total multiplicity of all negative ends is at most $3e-1$.
\end{lemma}

Now consider a curve $G$ of $\mathbf{F}$ whose image lies in the symplectization $SE$. Suppose also that
$G$ has the highest  level among such curves, $k$.
Since none of the curves of $\mathbf{F}$ in $(\CC P^2 \smallsetminus E)^{\infty}_-$ have negative ends asymptotic to multiples of $\gamma_2$,  it follows from the existence of the map $\overline{\mathbf{ F}}$ that the positive ends of $G$ can only be asymptotic to multiples of $\gamma_1$. Suppose that $G$ has $s_1^+$ such ends, and that the $i^{th}$ one covers $\gamma_1$ a total of $a^+_i$ times.  As established above, curves of $\mathbf{F}$ in $(\CC P^2 \smallsetminus E)^{\infty}_-$ have in total at most $3d-1$ negative ends when counted with multiplicity. Hence
\begin{equation}
\label{bound2}
\sum_{i=1}^{s_1^+}a^+_i  \leq 3d-1 < S^2.
\end{equation}
Suppose that $G$ has $s_1^-$ negative ends asymptotic to multiples of $\gamma_1$ with the $i^{th}$ such end covering this orbit $b^-_i$ times, and $G$ has $s_2^-$ negative ends asymptotic to multiples of $\gamma_2$ with the $i^{th}$ such end covering $\gamma_2$ a total of $c^-_i$ times.
Then, by Stokes' Theorem we have
\begin{eqnarray*}
0 & \leq & \int_{G} d \alpha_E \\
{} & = & \frac{\pi}{S^2} \left(\sum_{i=1}^{s_1^+}a^+_i - \sum_{i=1}^{s_1^-} b^-_i\right)  -  \pi \left(\sum_{i=1}^{s_2^-} c^-_i\right) \\
{} & \leq & \frac{\pi}{S^2}\sum_{i=1}^{s_1^+}a^+_i - \pi \sum_{i=1}^{s_2^-} c^-_i \\
{} & \leq & \frac{\pi (3d-1)}{S^2}  - \pi \sum_{i=1}^{s_2^-} c^-_i.
\end{eqnarray*}
Our choice of $S$ satisfying \eqref{tford} implies that
$$
\sum_{i=1}^{s_2^-} c^-_i \leq \frac{3d-1}{S^2} <1,
$$
and so $s^-_2 = 0$. Integrating $d\alpha_E$ over $G$ once again, we now have
\begin{equation}
\label{more}
\sum_{i=1}^{s_1^+}a^+_i - \sum_{i=1}^{s_1^-} b^-_i  \geq 0.
\end{equation}
Hence the total number of positive ends of $G$, counted with multiplicity, is no less than the total  number of its negative ends and by \eqref{bound2} and \eqref{more} and Proposition \ref{indexSE} we have
\begin{eqnarray*}
\mathrm{index}(G) & = & -2 + s_1^+ -s_1^- + \sum_{i=1}^{s_1^+}\mu(\gamma_1^{(a^+_i)}) - \sum_{i=1}^{s_1^-} \mu(\gamma_1^{(b^-_i)}) \\
{} & = & -2 + 2s_1^+  + 2\sum_{i=1}^{s_1^+}(a^+_i + \lfloor a^+_i /S^2 \rfloor) - 2\sum_{i=1}^{s_1^-} (b^-_i + \lfloor b^-_i / S^2 \rfloor)\\
{} & = & 2(s_1^+ - 1) + 2 \left( \sum_{i=1}^{s_1^+} a^+_i  - \sum_{i=1}^{s_1^-} b^-_i \right) \\
{} & \geq & 0.
\end{eqnarray*}
Hence, the virtual index of $G$ is strictly positive unless $s_1^+=1$ and $\sum_{i=1}^{s_1^+}a^+_i = \sum_{i=1}^{s_1^-} b^-_i$.
This condition is equivalent to the curve being a multiple cover of a cylinder over $\gamma_1$.
As $G$ has no negative ends asymptotic to $\gamma_2$ the same conclusions apply by induction to lower level curves mapping to $SE$.
To summarize, we have
\begin{lemma}
\label{g}
Let $G$ be a curve of $\mathbf{ F}$ with image  in the symplectization $SE$. The positive and negative ends of $G$ are all  asymptotic to some multiple of $\gamma_1$
and the positive ends cover $\gamma_1$ at least as many times as the negative ends. The virtual index of $G$ is nonnegative and is strictly positive unless
$G$ has one positive end and is a multiple cover of a cylinder over $\gamma_1$.
\end{lemma}

Again, the same proof yields the following useful result for curves not necessarily in $\mathbf{ F}$.

\begin{lemma} \label{new3}
Let $G$ be finite energy curve in $SE$ of genus zero whose positive ends are all asymptotic to multiples of $\gamma_1$ and have total multiplicity at most $3d-1$. Then the negative ends of $G$ are also asymptotic to multiples of $\gamma_1$, the positive ends cover $\gamma_1$ at least as many times as the negative ends, and the virtual index of $G$ is nonnegative and is strictly positive unless
$G$ has one positive end and is a multiple cover of a cylinder over $\gamma_1$.
\end{lemma}

Finally, we consider  a curve $H$ of $\mathbf{F}$ whose image is  in $E^{\infty}_+$. Our analysis of the curves $G$ above implies that none of the positive ends of $H$ are asymptotic to multiples of $\gamma_2$. Suppose that $H$ has   $s^+_1$ positive ends asymptotic to multiples of $\gamma_1$ with  the $i^{th}$ such  end covering  this orbit  $a^+_i$ times. As before we let $M$ be the number of point constraints on the corresponding moduli space. Let $N$ be the intersection number of the curve with the exceptional divisor $\Sigma$.
The fact that the negative ends of the collection of curves  of $\mathbf{ F}$ in $(\CC P^2 \smallsetminus E)^{\infty}_-$
cover $\gamma_1$ at most $3d-1$ times, together with the fact that positive ends of the curves in $SE$ cover $\gamma_1$  at least as many times as their  negative ends, implies that
\begin{equation*}
\label{ }
\sum_{i=1}^{s^+_1} a^+_i \leq 3d-1 < S^2.
\end{equation*}
Hence the index formula in Proposition \ref{indexE} reduces to
\begin{equation*}
\label{ }
\mathrm{index}(H) = 2(s^+_1-M-N-1)  + 2\sum_{i=1}^{s^+_1} a^+_i.
\end{equation*}

Arguing as above, the curve $H$ can be realized as a  $p$-fold cover of a somewhere injective curve $\widetilde{H}$ with the same point constraints. Suppose that this curve has $\tilde{s}^+_1$ positive ends with the $i^{th}$ having multiplicity $\tilde{a}^+_i$. Then our assumption that the points are in general position implies that the constrained index
\begin{equation*}
\label{ }
\mathrm{index}(\widetilde{H}) =  2(\tilde{s}^+_1 - M -\tilde{N} -1)  + 2\sum_{i=1}^{\tilde{s}^+_1} \tilde{a}^+_i   \geq 0
\end{equation*}
where $p\tilde{N}=N$.
Since $\sum_{i=1}^{s^+_1} a^+_i=p \sum_{i=1}^{\tilde{s}^+_1} \tilde{a}^+_i $ and $s^+_1 \geq \tilde{s}^+_1$, we then have

\begin{equation*}
\label{ }
\mathrm{index}(H)  \geq  \mathrm{index}(\widetilde{H}) -2\widetilde{N} (p-1)+2(p-1)\sum_{i=1}^{\tilde{s}^+_1} \tilde{a}^+_i
\end{equation*}
and hence
\begin{equation}
\label{new12}
\mathrm{index}(H)  \geq  \mathrm{index}(\widetilde{H}) +2(p-1)\left(\sum_{i=1}^{\tilde{s}^+_1} \tilde{a}^+_i - \widetilde{N}\right).
\end{equation}

As with the curves of $\mathbf{ F}$ with images in the other two targets, we would like to conclude that the curves like $H$
always have nonnegative virtual index. It is clear from inequality \eqref{new12}, and regularity, that this requires us to show that the term
$\sum_{i=1}^{\tilde{s}^+_1} \tilde{a}^+_i - \widetilde{N}$ is nonnegative. It is precisely at this point where we need to use the freedom to choose the blow-up radius $r$.

\begin{lemma}\label{rclose}
If the blow-up radius $r$ is sufficiently close to $1/S$, then for every  $H$ (and $\widetilde{H}$) as above we have
$$\sum_{i=1}^{\tilde{s}^+_1} \tilde{a}^+_i - \widetilde{N} \geq 0.$$
\end{lemma}

\begin{proof}
Since $J$ belongs to $\mathcal{J}^{\star}_E$, the exceptional divisor $\Sigma$ is holomorphic and $J$ is standard on the neighborhood $U_{\Sigma}$ of $\Sigma$.
By positivity of intersection then the set $\widetilde{H}^{-1}(\Sigma)$ is finite. Denote  the restriction of  $\widetilde{H}$ to the complement of $\widetilde{H}^{-1}(\Sigma)$ by  $\widetilde{H}'$.  Symplectically we can identify $E \smallsetminus \Sigma$ with $E \smallsetminus B^4(r)$ (recalling that $\Sigma$ was the result of blowing up a ball of radius $r$). Holomorphically, we can identify  $U_{\Sigma} \smallsetminus \Sigma$ with an open neighborhood of the origin in $\C^2$. Under these identifications, our restricted map  $\widetilde{H}'$ extends to the oriented blow-up at each of its punctures, as a map (still denoted by $\widetilde{H}'$) which takes a boundary circle to a cover of a Hopf circle on $\partial B^4(r)$. (The oriented blow-up is again defined as in \cite{BEHWZ} section $4.3$.) More precisely, if a puncture corresponds to an intersection point with $\Sigma$ of multiplicity $m$,  then the corresponding boundary circle gets mapped by $\widetilde{H}'$ to the $m$-fold cover of a Hopf circle. Gluing the $m$-times cover of a complex line through $0 \in B^4(r)$ to this $m$-fold covered Hopf circle, and repeating this for each puncture,  we obtain a continuous map $\widetilde{H}''$, holomorphic away from the Hopf circles, whose domain is given by gluing disks to boundary circles of the domain of  $\widetilde{H}'$, and whose target is now the positive symplectic completion of $E$,
$$
(E, \omega_0) \cup (\partial E \times [0, \infty), d(e^{\tau} \alpha_E)).
$$
As in section \ref{split} we can associate to $\widetilde{H}''$ a map with domain the oriented blow-up $\overline{S}$ of the domain $S$ of $\widetilde{H}''$ and target $E$. Let us denote this map simply by $\overline{H}$ and compute its symplectic area with respect to the standard symplectic form $\omega_0$ on $E$ (rather than a form on $E \# \overline{\CC P^2}$ as in Definition \ref{symarea}). The form $\omega_0$ is just the restriction of the standard form on $\RR^4$ and has a primitive $\lambda_0$ whose restriction to $\partial E$ we always denote by $\alpha_E$.

By Stokes' Theorem we have
\begin{equation*}
\label{ }
\int_{S} \overline{H}^* \omega_0 = \int_{\partial \overline{S}} \overline{H}^* \alpha_E = \frac{\pi}{S^2} \sum_{i=1}^{\tilde{s}^+_1} \tilde{a}^+_i.
\end{equation*}
By construction, we also have
\begin{equation*}
\label{ }
\int_{S} \overline{H}^* \omega_0 = \int_{S'} \overline{H}^* \omega_0 +  \widetilde{N}\pi r^2.
\end{equation*}
In the second integral $S'$ denotes $ \overline{H}^{-1} (E \smallsetminus B^4(r))$. The formula holds as the disks we glue all have area $\pi r^2$.

Since $\int_{S'} \overline{H}^* \omega_0 \geq 0$ this implies,
\begin{equation}
\label{rnear}
\frac{\pi}{S^2}\sum_{i=1}^{\tilde{s}^+_1} \tilde{a}^+_i \ge \widetilde{N}\pi r^2.
\end{equation}
Now $d-1\geq N \geq \widetilde{N}$ and so choosing $r$ sufficiently close to $1/S$, say within $1/S^4$ of $1/S$, we can then conclude
from \eqref{rnear} that for any curve $H$ (and $\widetilde{H}$) as above, we have $$\sum_{i=1}^{\tilde{s}^+_1} \tilde{a}^+_i \ge \widetilde{N}.$$
\end{proof}

Together, equation \eqref{new12} and Lemma \ref{rclose} imply that
\begin{equation}
\label{coversofH}
\mathrm{index}(H) \ge \mathrm{index}(\widetilde{H}) \ge 0.
\end{equation} We then  have

\begin{lemma}
\label{h}
If the blow-up radius $r$ is sufficiently close to $1/S$ then for any curve $H$ of $\mathbf{F}$ with image in $E^{\infty}_+$, the positive ends of $H$ are asymptotic to some multiple of $\gamma_1$ and the virtual index of $H$ is nonnegative.
\end{lemma}

Note that the deformation  index of the curves in  ${\mathcal M}_d(J^N, p_1, \dots, p_{2d})$ is zero,
and so the sum of the virtual indices of the curves of $\mathbf{ F}$, plus twice the number of nodes connecting components of the same level, is also zero.
It follows from  Lemmas \ref{f}, \ref{g}, and \ref{h},  that every curve of $\mathbf{F}$ must  have virtual index zero (as we have shown that none can have negative virtual index).
The same three  lemmas then yield the three statements of Proposition \ref{index}. Furthermore, for our sum to be zero we also see that there can be no nodes connecting components of the same level, and in particular each component has a positive or negative puncture (that is, a nonremovable singularity). This was already observed in Remark \ref{new1}.

\medskip

\subsubsection{Monotonicity and still finer restrictions.}
For any integer $l>0$, one can symplectically embed $l$ disjoint open balls of radius one into the ellipsoid $E(1, \sqrt{l})$, \cite{tr}.
Hence, for $S$ satisfying \eqref{tford}, and given a symplectically embedded $E=E(1/S,1) \subset B^4(R) \subset \CC P^2(R)$, we may choose the constraint points $p_1, \dots, p_{2d}$ to lie at the center of disjoint embedded Darboux balls in $E$ which have radii equal to the blow-up radius $r < 1/S$ and whose closures lie in the interior of $E$. We may also assume that these balls are all disjoint from the ball $B^4(r) \subset E$ that we have blown-up to obtain our symplectic manifold $(\CC P^2 \# \overline{\CC P^2}, \omega_{R,r})$.

Let  $J \in {\mathcal J}_E^{\star}$ be an almost-complex structure
which agrees with the standard complex structure in the balls around the points $p_i$ and induces almost-complex structures  on
$(\CC P^2 \smallsetminus E)^{\infty}_-$, $SE$, and $E^{\infty}_+$ which are regular for somewhere injective, finite energy curves of genus
zero. As no holomorphic curves have image contained entirely in these balls such $J$ exist. In fact, the almost complex structures in ${\mathcal J}_E^{\star}$ which satisfy both of these conditions form a subset of the second category in the set of structures which satisfy just the first. Assume now that  the  point constraints are in general position relative to $J$. Consider  again a sequence of curves  $f_N$ which represent classes  in ${\mathcal M}_d(J^N, p_1, \dots, p_{3d-1})$ and converges, in the sense of \cite{BEHWZ},
 to a holomorphic building $\mathbf{ F}$.   For this new choice of $J$ we get the following refinements of Proposition \ref{index}.

\begin{lemma}\label{limit} With $J$ as above and $r$ sufficiently close to $1/S$, the limit $\mathbf{F}$ contains exactly one curve whose domain is connected and whose image lies in $(\CC P^2 \smallsetminus E)^{\infty}_-$. This curve, $F$, has exactly $3d-1$ negative ends when counted with multiplicity. Moreover, the curves of  $\mathbf{F}$ with connected domain and image in $SE$ each have exactly one positive puncture,  and the curves of  $\mathbf{F}$ with connected domain and image in $E^{\infty}_+$ are all holomorphic planes.
\end{lemma}

\begin{proof}

Our choice of $J$ forces the  negative ends of the  curves  of $\mathbf{ F}$ with image in $(\CC P^2 \smallsetminus E)^{\infty}_-$
to cover $\gamma_1$ at least $3d-1$ times and hence, by Lemma \ref{f}, exactly $3d-1$ times.
To see this, let $G_E$ be the collection of curves of $\mathbf{F}$ with images in
either $SE$ or $E^{\infty}_+ $. Denote by $\overline{G}_E$ the map to $E \# \overline{\C P^2}$ formed by fitting together compactifications of the curves of $G_E$.  By the existence of the map $\overline{\mathbf{ F}} \colon S^2 \to \CC P^2 \# \overline{\C P^2} $, the ends of $\overline{G}_E$ cover $\gamma_1$ the same number of times, say $t$,
as the  curves  of $\mathbf{ F}$ with image in $(\CC P^2 \smallsetminus E)^{\infty}_-$.

Arguing as in the proof of Lemma \ref{rclose}, we see that the symplectic area (see Definition \ref{symarea}) of $\overline{G}_E$ is $t \pi / S^2 - (d-1)\pi r^2$
where the negative terms correspond to the $d-1$ intesections of $\overline{G}_E$ with  $\Sigma$. The monotonicity theorem
for holomorphic curves and our choice of $J$ implies that the intersections of $\overline{G}_E$ with the balls centered at the points $p_i$ each have symplectic area bounded below by $\pi r^2$. Hence, $t \pi / S^2 - (d-1)\pi r^2 \ge 2d \pi r^2$ and for $r$ sufficiently close to $1/S$ we can conclude that $t \ge 3d-1$.

In summary, the curves  of $\mathbf{F}$ with connected domain and image in $(\CC P^2 \smallsetminus E)^{\infty}_-$  have a  total Chern number of $3d$ and, collectively, their  negative ends cover
$\gamma_1$ exactly $3d-1$ times. As shown in the proof of Proposition \ref{index}, the deformation index
of a curve $F$ in $\mathbf{ F}$ whose image lies in  $(\CC P^2 \smallsetminus E)^{\infty}_-$ is given by
$$
\mathrm{index}(F) = -2+ 2c_1(F) -2\sum_{i=1}^{s^-_1} a^-_i .
$$ where $s^-_1$ represents the number of negative
ends and $a^-_i$ is the number of times the $i^{th}$
negative end covers $\gamma_1$. If there are $K$ such
curves, then their
 total deformation
index is  $$-2K +2(3d) - 2(3d-1) = -2K+2.$$
 Since the total index
must be nonnegative we have $K=1$. Hence, there is exactly one curve, say $F$,
of $\mathbf{ F}$ in $(\CC P^2 \smallsetminus E)^{\infty}_-$, and $F$ has index zero and exactly $3d-1$ negative ends,
when counted with multiplicity.

The remaining statements  of Lemma \ref{limit} follows easily from the first as $F$ is a limit of genus zero curves.
\end{proof}

Henceforth, we will always assume that the blow-up radius has been chosen sufficiently close to $1/S$ for Proposition \ref{index}
and Lemma \ref{limit} to hold.


\subsubsection{Limits of finite energy curves in $(\CC P^2 \smallsetminus E)^{\infty}_-$}

The compactness theorems of \cite{BEHWZ} also cover appropriate sequences of  finite energy curves  mapping to $(\CC P^2 \smallsetminus E)^{\infty}_-$, and in our specific setting we can again obtain some finer restrictions on the possible limits.


\begin{proposition} \label{anotherlimit}
Let $F_i$ be a sequence of finite energy holomorphic curves with connected domains and image in $(\CC P^2 \smallsetminus E)^{\infty}_- $,  such that each $F_i$ has genus zero, degree $k<d$, index $0$, and exactly $s$ negative ends (not counting multiplicities) each asymptotic to $\gamma_1$. There is a subsequence of the $F_i$ which converges in the sense  of \cite{BEHWZ} to a pseudo-holomorphic building consisting of curves mapping to either $(\CC P^2 \smallsetminus E)^{\infty}_-$ or $SE$. Moreover, the top level curve of any such limit, say $\mathbf{F}$, consists of  a single curve  in $(\CC P^2 \smallsetminus E)^{\infty}_- $ with a connected domain. This curve has  degree $k$ and $s' \le s$ negative ends. If $s'=s$ then we may assume that $\mathbf{F}$ has no additional curves with image in $SE$.
\end{proposition}

\begin{proof}

The existence of a convergent subsequence follows immediately from Theorem 10.2 of \cite{BEHWZ}, so for simplicity let us assume that the $F_i$ converge to $\mathbf{F}$ in the sense of \cite{BEHWZ}.  The nature of this convergence has, among others, the following four implications: (i)  the curves of $\mathbf{F}$ with image in $(\CC P^2 \smallsetminus E)^{\infty}_-$ have total degree $k<d$ (first Chern number $3k<3d$); (ii) the compactifications of the curves of $\mathbf{F}$  fit together to form a continuous map  $\overline{\mathbf{F}}$ from a connected compact surface with genus zero and $s$ boundary circles to $\overline{\CC P^2 \smallsetminus E}$; (iii) the sum of the virtual indices of all curves  of $\mathbf{F}$ is equal to $0$, the virtual index of the $F_i$; and (iv) $\mathbf{F}$ has a level structure and the lowest level curves of $\mathbf{F}$ (with connected domains) have, collectively,  the same asymptotic behavior as the $F_i$, i.e., a total of $s$ negative ends (not counting multiplicities) each asymptotic to $\gamma_1$.

Lemma \ref{new2} together with property (i) above, implies that  the curves of $\mathbf{F}$ with image  in $(\CC P^2 \smallsetminus E)^{\infty}_- $, like the $F_i$ themselves, have nonnegative deformation index and their (negative) ends are all asymptotic to multiples of $\gamma_1$ and cover  $\gamma_1$, in total, at most  $3k-1$ times.

This asymptotic behavior  of the the curves of $\mathbf{F}$ with image  in $(\CC P^2 \smallsetminus E)^{\infty}_- $ together with the existence of the map $\overline{\mathbf{F}}$ implies that  each positive end of the top level curves of $\mathbf{F}$ with image in $SE$ must also be asymptotic to multiples of $\gamma_1$. By Lemma \ref{new3} it then follows that the negative ends of the top level curves with image in $SE$ are also asymptotic to multiples of $\gamma_1$. By iteration on the level, we see that for each curve of $\mathbf{F}$ with image in $SE$ both the positive and negative ends are asymptotic to some multiple of $\gamma_1$. Lemma \ref{new3} also implies that
the virtual index of each of these curves is nonnegative and is strictly positive unless the curve has one positive end and is a multiple cover of a cylinder over $\gamma_1$.

It now follows from the discussion above  and property (iii)  that each curve in $\mathbf{F}$  has  index $0$, and that each  curve of $\mathbf{F}$  with image in $SE$ has a single positive end and is a multiple cover of a cylinder over $\gamma_1$. Together with the fact that  the map $\overline{\mathbf{F}}$  has a connected domain, this implies that there is a single curve, say $F$, of  $\mathbf{F}$ with connected domain and image in $(\CC P^2 \smallsetminus E)^{\infty}_- $. Again, by the discussion above we also see that $F$ has degree $k$.

Finally, to verify the statement on the asymptotics of $F$, we utilize property (iv). If the collection of  lowest level curves  of $\mathbf{F}$ is $F$ itself, then (iv) implies that $s'$, the number of negative ends of $F$, is equal to $s$ and we are done. Suppose then that the lowest level curves  of $\mathbf{F}$ instead map to $SE$. As described above, each of these curves has exactly one positive end and is a multiple cover of a cylinder over $\gamma_1$. Hence, the existence if $\overline{\mathbf{F}}$ implies that $s' \le  s$. If $s'=s$, then all curves  of $\mathbf{F}$ with image in $SE$ are unbranched multiple covers of the trivial cylinder over $\gamma_1$. In this case, one can omit these curves from the limit and our sequence still converges to $\mathbf{F} =F$ (see Remark \ref{uniquelimit}).


\end{proof}

\subsection{Holomorphic curves with varying point constraints} \label{varypoint}

In this section we prove two results involving families of (moduli spaces of)  curves
with varying  point constraints.

\subsubsection{Curves in $\CC P^2 \# \overline{\CC P^2}$} \label{firstmove}

Let $J^N$ be the almost-complex structure on $((\CC P^2 \# \overline{\CC P^2})^{N}, \omega_{R,r}^{N})$ defined by  $J \in {\mathcal J}_E^{\star}$
in the manner  described in Section \ref{split}.

Consider the space of smooth paths
$$
\Lambda(E \smallsetminus \Sigma) = C^{\infty}([0,1], (E \smallsetminus \Sigma)^{2d}).
$$
For  $\bar{p}(t) = \{p_i(t)\}_{i=1}^{2d} \in \Lambda(E \smallsetminus \Sigma)$, a $J$ in ${\mathcal J}_E^{\star}$
and an $N \in \mathbb{N}$ set
$${\mathcal N}(J^N, \bar{p})=\{(C,t)|C \in {\mathcal M}_d(J^N, p_1(t), \dots, p_{2d}(t))
\text{ and } t \in [0,1]\}.$$


\begin{proposition}\label{path1} Suppose that for a $J$ in ${\mathcal J}_E^{\star}$, the almost-complex structure $J^N$ is regular for every moduli space of holomorphic spheres $\mathcal{M}_A(J^N)$. Then there exists a subset of $\Lambda(E \smallsetminus \Sigma)$ of the second category such that the corresponding moduli space
${\mathcal N}(J^N, \bar{p})$ is a compact $1$-dimensional manifold and the
natural projection onto $[0,1]$ is a diffeomorphism.
\end{proposition}

\begin{proof} By Corollary \ref{regularcor11}, there is a set of the second category such that for all $t$ the $\{p_i(t)\}_{i=1}^{2d}$ are in general position with respect to $J^N$. By Lemma \ref{compactCP22} this implies that ${\mathcal N}(J^N, \bar{p})$ is a compact $1$-dimensional manifold. Finally, by Proposition \ref{compactCP23}, the space ${\mathcal N}(J^N, \bar{p})$ intersects the fibres of the projection transversally in single classes and so the projection is indeed a diffeomorphism.

\end{proof}


\subsubsection{Stretching the neck}

Given an almost-complex structure $J^N$ and path  $\bar{p}$ as in Proposition \ref{path1}, let $f_N(t)$ be a smooth family of $J^N$-holomorphic curves such that $[f_N(t)]$ is the unique class of holomorphic spheres for which $([f_N(t)],t) \in {\mathcal N}(J^N, \bar{p})$.

\begin{proposition}
\label{path2} There is a subset of $ {\mathcal J}_E^{\star}$ of the second category such that for each $J$ in this subset the induced $J^N$ are all regular in the sense of Lemma \ref{regular1} and the corresponding limiting almost-complex structures on
$(\CC P^2 \smallsetminus E)^{\infty}_- $, $SE$  and $E_+^{\infty}$ are regular for somewhere injective finite energy pseudo-holomorphic curves of genus
zero. For such $J$ there exist a subset of paths in $\Lambda(E \smallsetminus \Sigma)$ of the second category such that the conclusions of Proposition \ref{path1} hold for all $N$. Furthermore, if subsequences $f_{N_j}(0)$ and $f_{N_j}(1)$ both
converge, then the top level  curves of the two corresponding limiting holomorphic buildings (that is, the two collections of limiting curves  which map into
$(\CC P^2 \smallsetminus E)^{\infty}_- $)  have identical images.
\end{proposition}

\begin{proof}
The first statements on regular almost-complex structures and paths follow from Lemmas \ref{regular1} and \ref{regular2}.

In particular by Lemma \ref{regular2} we may assume that for paths $\bar{p}$ in our subset the corresponding points $\{p_i(t)\}$ are in general position for all $t$ for the limiting almost-complex structure on $E^{\infty}_+$.
We conclude that the restrictions  of Proposition \ref{index} hold for any holomorphic building which arises as the limit of a subsequence of curves of the form
$f_{N}(t_{N})$.

Let $\mathbf{ F}^0$ and $\mathbf{ F}^1$ be the limiting buildings of $f_{N_j}(0)$ and $f_{N_j}(1)$, respectively. Denote the
top level curves of $\mathbf{ F}^0$ and $\mathbf{ F}^1$ by $v^0=\{F^0_1, \ldots, F^0_{m_0}\}$ and $v^1=\{F^1_1, \ldots, F^1_{M_1^{\epsilon}}\}$.  Each of the holomorphic curves $F^0_i$ and $F^1_j$ here has a connected domain of genus zero, finite energy, index zero, and is regular. Arguing  by contradiction, we assume that the images of $v^0$ and $v^1$ are distinct.

For any sequence of points $t_{N_j} \in [0,1]$  we may assume that, after passing to  a further subsequence, the $t_{N_j}$ converge and the corresponding holomorphic spheres $f_{N_j}(t_{N_j})$ converge to a holomorphic building  such that, by Proposition \ref{index},  the curves of this building with  image in $(\CC P^2 \smallsetminus E)^{\infty}_- $  have finite energy, genus zero and index zero.   Choose a sequence $t_{N_{j_k}}$ in $[0,1]$ such that $t_{N_{j_k}} \to t_{\infty}$ and  the curves $f_{N_{j_k}}(t_{N_{j_k}})$ converge to a holomorphic building $\mathbf{G}$ whose top level curve $v_G$ is of the form $\{G_1, \ldots, G_{m_0}\}$, where the set of degrees  of the $G_i$ is equal to the set of degrees  of  the $F^0_i$. (For example, we can choose $t_{N_{j_k}} =0$). Assume further that the total number of negative ends of $v_G$ (not counting multiplicity) is maximal amongst all such limits. This last condition implies that $v_G$ is isolated in the following sense.

\begin{lemma}\label{isolate}
If $\mathcal{V}$ is the set of all top level curves of  limiting buildings of sequences  of the form $f_{N_j}(t_{N_j})$, then there is no nontrivial sequence (that is, a sequence in which infinitely many terms are distinct) $v_j$ in $\mathcal{V}$ which converges in the sense of \cite{BEHWZ} to a holomorphic building whose top level curve is $v_G$.
\end{lemma}

\begin{proof}
Suppose
that such a sequence $v_j$ exits. Without loss of generality we may assume that  $v_j =\{G_j^1, \ldots, G_j^m\}$  where the asymptotic behavior and the degree of  $G_j^i$ do not depend on $j$. Applying Proposition \ref{anotherlimit} to each
sequence  $G^i_j$ it immediately follows that $m = m_0$, and the set of degrees of the component curves of the $v_j$ must match the set of degrees of the component curves of $v_G$. Relabeling,  if necessary, we may assume that the top level of the
limit of $G^i_j$ is $G_i$. Now the curves $G_i$ are regular and have index zero, and thus are isolated among curves with the same degree and asymptotic behaviour. Therefore, if the $v_i$ is a nontrivial convergent sequence the asymptotic behavior of at least one of the $G^i_j$ must be different than that of $G_i$ and the limiting building of $G^i_j$ must have nontrivial components in $SE$. In this case, Proposition \ref{anotherlimit} now implies that the $G^i_j$ have more negative ends than $G_i$ and hence  the total number of negative limits of the $v_j$ must be strictly larger then the total  number of negative ends of  $v_G$. The existence of any such $v_j$ contradicts the maximality condition for $v_G$.
\end{proof}

By assumption, the images of  $v^0$ and  $v^1$ are not equal  and so the image $v_G$ must differ from one of them. As the argument will be identical in both cases, let us assume that the image of $v_G$ is not equal to the image of $v^1$.

Let $\mathcal{I}_G$ and $\mathcal{I}_1$ be the intersection of  $\overline{\CC P^2 \smallsetminus E}$  with the images of
$v_G$ and $v^1$, respectively. It then follows from the unique continuation principle that $\mathcal{I}_G$ and $\mathcal{I}_1$
 are  distinct compact subsets of $\overline{\CC P^2 \smallsetminus E}$.
Fixing  a metric on $\overline{\CC P^2 \smallsetminus E}$ we consider the corresponding Hausdorff metric on compact subsets of $\overline{\CC P^2 \smallsetminus E}$.
For large $k$, the sets   $f_{N_{j_k}}(t)(S^2) \cap \left(\overline{\CC P^2 \smallsetminus E}\right)$ are arbitrarily close to $\mathcal{I}_G$ for $t$ near $t_{N_{j_k}}$, and arbitrarily close to $\mathcal{I}_1$ for $t$ near $1$.
Since $\mathcal{I}_G \neq \mathcal{I}_1$, we can then choose,  for all large $k$, a time $t^{\epsilon}_{N_{j_k}} \in (t_{N_{j_k}}, 1)$
such that $f_{N_{j_k}}(t^{\epsilon}_{N_{j_k}})(S^2) \cap \left(\overline{\CC P^2 \smallsetminus E}\right)$ is a fixed distance, say $\epsilon>0$,  from $\mathcal{I}_G$.  Passing to a subsequence, if necessary,
we may then assume that the $t^{\epsilon}_{N_{j_k}}$ converge to some $t^{\epsilon}_{\infty} \in [0,1]$ and that  $f_{N_{j_k}}(t^{\epsilon}_{N_{j_k}})$ converges to a holomorphic building, $\mathbf{ F}_{\epsilon}$. By the choice of the $t^{\epsilon}_{N_{j_k}}$, the image of the top level curve $v_{\epsilon}$  of $\mathbf{ F}_{\epsilon}$  intersects  $\overline{\CC P^2 \smallsetminus E}$ in a set whose distance  from $\mathcal{I}_G$ is $\epsilon$ in the Hausdorff metric. Letting $\epsilon \to 0$ we can find a  nontrivial sequence of top level curves $v_{\epsilon_l}$ converging to a building whose top level curve has image $\mathcal{I}_G$ when intersected with $\overline{\CC P^2 \smallsetminus E}$, and is somewhere injective by Proposition \ref{index}. By the unique continuation principle, this top level curve must be $v_G$ itself. This contradicts the fact that $v_G$ is isolated in the sense of Lemma \ref{isolate}.

\end{proof}


Let $f_N(t)$ and $N_j$ be a family of curves and a sequence as in the statement of Proposition \ref{path2}. The above proof gives the following information on the images of the $f_{N_j}(t)$, which is a convenient way to apply Proposition \ref{path2}. Let $K \subset (\CC P^2 \smallsetminus E)^{\infty}_-$ be a compact subset with smooth boundary $\partial K$ such that the limiting component of the $f_{N_j}(0)$ in $(\CC P^2 \smallsetminus E)^{\infty}_-$, say $v$, intersects $\partial K$ transversally. Define $D = v^{-1}(K)$ and $D_j(t) = f_{N_j}(t)^{-1}(K)$ where in the latter expression we have identified $K$ with a compact subset of $(\CC P^2 \# \overline{\CC P^2})^{N_j}$ for large $j$.

\begin{corollary} \label{closeimage}
Given any  $\epsilon>0$  there exists a $j_0$ such that if $j \ge j_0$ then for all $t \in [0,1]$ there is a diffeomorphism $g_j(t):D \to D_j(t)$ such that $f_{N_j}(t) \circ g_j(t)$ is $\epsilon$-close to $v|_D$ in a $C^{\infty}$-norm (defined  by fixing metrics on $D$ and $K$).
\end{corollary}

Indeed, if the $f_{N_j}(t)$ converge uniformly to a building with top level component identical to $v$ then we can take the $g_j$ to be a small perturbation of the maps $\sigma_{N_j}^{-1}$ implicit in the definition of convergence in Section \ref{split}. We notice that if $f_{N_j} \circ \sigma_{N_j}^{-1}$ is sufficiently $C^{\infty}$-close to $v$ and $v$ intersects $\partial K$ transversally then by the implicit function theorem  $f_{N_j} \circ \sigma_{N_j}^{-1}$ can be slightly perturbed such that the preimage of $\partial K$ coincides with that of $v$. On the other hand, if the $f_{N_j}(t)$ do not converge uniformly to a building with top level component identical to $v$ (up to reparameterization, of course) then after taking a subsequence and reparameterizing our path we will derive a contradiction to Proposition \ref{path2}.

%



\subsection{An existence theorem}

The key to the proof of Theorem \ref{thm}
is the following existence result.

\begin{theorem}
\label{key}
For any integer $d \geq 1 $ and  a suitable choice of almost-complex structure $J \in {\mathcal J}_E^{\star}$, there exists a regular, finite energy holomorphic
plane of degree $d$ in $(\CC P^2 \smallsetminus E)^{\infty}_-$ whose
negative end covers the periodic orbit $\gamma_1$ precisely $3d-1$
times.
\end{theorem}
\noindent Here, as before, $E$ is the (the image of the) ellipsoid $E(1/S,1)$ and we assume that $S>\sqrt{3d}$.

\subsubsection{The proof of Theorem \ref{key}}

\noindent{\bf An overview.}
As the proof is quite long  we will begin with an outline of its structure. The proof is divided into 7 steps. In the first of these we observe that for a suitable choice of $J$, which we fix,  Lemma \ref{limit} provides us with a viable candidate for the desired curve  in $(\C P^2 \smallsetminus E)^{\infty}_-$. This candidate, $F$, is the unique top level curve of a holomorphic building $\mathbf{F}^0$ which arises as the limit of $J$-holomorphic spheres satisfying special point constraints.

In Step 2, we observe that Proposition \ref{path2} allows us to view $F$  as part of the limit  $\mathbf{F}$ of  a sequence of curves which  are also $J$-holomorphic but which satisfy a different, essentially arbitrary, choice of point constraints. This freedom to choose point constraints  will be crucial  in what follows.  Also in Step 2, we  use Corollary \ref{closeimage} to explore some relationships between $\mathbf{F}$ and $\mathbf{F}^0$. These are expressed in terms of the behavior of $\mathbf{F}$ on the complement  of a subset of its domain of the form $D= F^{-1}(K)$ where $K$ is a fixed compact subset of  $(\C P^2 \smallsetminus E)^{\infty}_-$.

To verify that $F$ is the desired curve it suffices at this point to show that $F$ has exactly one negative end. In Step 3 we prove that, in fact, it suffices to prove that there is a marked point in each component  of the complement of $D$ in the domain of $\mathbf{F}$. This leaves us with the task of choosing  point constraints so that  the resulting limit $\mathbf{F}$ has this property.

In Step 4, we define the desired point constraints ( in {\it good position relative to $J$}) and prove that they exist.  These points are first chosen to lie in a neighborhood of the exceptional divisor $\Sigma$ where $J$ is standard. Identifying  this neighborhood with a small disk bundle $V_{\Sigma}$ in the tautological line bundle $L$ over $\Sigma = \C P^1$, we then place further restrictions on the points. These restrictions are expressed in terms of the almost-complex structures which occur in the process of splitting along $\partial E$ and/or $\partial V_{\Sigma}$  and the holomorphic curves in the resulting symplectic completions.



%
%

In Step $5$ we begin the final formal argument. Arguing by contradiction, we assume that for our special point constraints there are no marked points in some component of the complement of $D=F^{-1}(K)$ in the domain of  $\mathbf{F}$. We then discuss a few immediate implications of this assumption and define a function which measures the distance to $F$ of finite energy curves in $(\C P^2 \smallsetminus E)^{\infty}_-$, or in fact of any curves mapping to $(\C P^2 \#\overline{ \C P^2})^N$  for $N$ sufficiently large.

In Step 6, we prove that for all small $\epsilon>0$ there exist curves $f^\epsilon_N$ which map into $(\C P^2 \#\overline{ \C P^2})^N$ and are a distance $\epsilon$ away from $F$ in the sense of Step 5.
The key result is Proposition \ref{mty} whose proof is the most technical section of the total argument.


Finally, in Step $7$, we complete the proof by taking a limit of the curves $f^\epsilon_N$ as $N \to \infty$.  This limit, $\mathbf{F^{\epsilon}}$, will have total index zero.  However, the fixed distance requirement, for a generic choice of $\epsilon$,  implies that some curve of $\mathbf{F^{\epsilon}}$ which maps into $(\C P^2 \smallsetminus E)^{\infty}_-$ must have positive index. Our various regularity conditions imply that all the other curves  of  $\mathbf{F^{\epsilon}}$ have nonnegative indices. Thus, we will obtain the desired contradiction.

\bigskip

\noindent{\bf Step 1.} {\it Identifying a candidate.} We begin by fixing points $p^0_1, \dots , p^0_{2d}$ in $E \smallsetminus \Sigma$ and an almost-complex structure  $J \in {\mathcal J}_E^{\star}$ which meets all the conditions required for us to apply Lemma \ref{limit} and Proposition \ref{path2}.
In particular, we assume that the points $p^0_1, \dots , p^0_{2d}$ lie at the center of disjoint Darboux balls in $E \smallsetminus \Sigma$ of radius $r$ (arbitrarily close to $1/S$) and that $J$ agrees with the standard complex structure in
these balls. By Lemma \ref{regular2} we may assume that the $J^N$ are all regular and that $J$ induces almost-complex structures  on
$(\CC P^2 \smallsetminus E)^{\infty}_-$, $SE$, and $E^{\infty}_+$ which are regular for somewhere injective, finite energy curves of genus zero. We may also suppose that the $p^0_i$ are in general position for all of the almost-complex structures induced by $J$.

Let $f^0_{N}$ be a sequence of curves with $[f^0_{N}]={\mathcal M}_d(J^{N}, p^0_1, \dots, p^0_{2d})$, and let $\mathbf{F}^0$ be the limiting building of a convergent subsequence of the $f^0_N$. By  Lemma \ref{limit},
$\mathbf{F}^0$ has a single curve, say $F$, with image in $(\CC P^2 \smallsetminus E)^{\infty}_-$.
Moreover, the domain of $F$ is connected, its negative ends are all asymptotic to $\gamma_1$ and they cover $\gamma_1$ a total of
$3d-1$ times.  To prove that $F$ is the desired holomorphic plane it suffices to show the following.

\begin{proposition}\label{oneend}
The curve $F$ has exactly one negative end.
\end{proposition}

\begin{remark}\label{up-reg}
The point-wise restrictions we have imposed on $J$ all occur within $E \#\overline{ \CC P^2 }$. Hence, the conclusion of Theorem \ref{key} holds for
any almost-complex structure on  $(\CC P^2 \smallsetminus E)^{\infty}_-$ which is obtained by a splitting along $\partial E$, is regular for somewhere injective, finite energy curves of genus zero in $(\CC P^2 \smallsetminus E)^{\infty}_-$, and for which  $\CC P^1(\infty)$ is complex.
\end{remark}

\noindent{\bf Step 2.} {\it Reidentifying our candidate.} To prove Proposition \ref{oneend} we will need to view $F$  as the unique  top level curve in the limiting building of a different sequence of $J$-holomorphic curves which satisfy different, carefully chosen, point constraints. The following immediate consequence of Proposition \ref{path2} gives us the freedom to do so.

\begin{lemma}\label{move} Let $p_1, \dots , p_{2d}$ be any collection of points in $E \smallsetminus \Sigma$ which are in general position for all of the almost-complex structures induced by $J$. There is a convergent sequence of curves $f_{N_j}$ representing classes in  ${\mathcal M}_d(J^{N_j}, p_1, \dots, p_{2d})$ whose limiting building, $\mathbf{F}$, has a single curve with image in $(\CC P^2 \smallsetminus E)^{\infty}_-$ and this curve has the same image as $F$.
\end{lemma}
In what follows, we will simplify notation by writing $N$ instead of $N_j$ and calling the top level curve of the new building $\mathbf{F}$, $F$.

We now use Corollary \ref{closeimage} to establish a deeper relationship between the curves of $\mathbf{F}^0$ and those of the limiting building $\mathbf{ F}$.
As in  Corollary \ref{closeimage}, we first choose  a compact subset $K$ of $(\CC P^2 \smallsetminus E)^{\infty}_-$ such that $F$ (the unique top level curve of both $\mathbf{ F}^0$ and $\mathbf{ F}$) intersects $\partial K$ transversally. Here it will be useful to choose a $K$ of the form  $$K= (\C P^2 \smallsetminus E) \cup (\partial E \times [\tau,0])$$ for some $\tau<0$.
Choosing $\tau$ to be sufficiently negative we may assume that the complement of $D = F^{-1}(K)$ (in the domain of $F$) consists of a collection of once punctured disks, one for each negative end of $F$, where the boundaries of the disks are mapped to $\partial K$, and $F$ is asymptotic at each puncture to a cover of $\gamma_1$. Fixing a $\delta>0$ and with area defined as in Definition \ref{symarea} we may also assume that the total area of the punctured disks is at most $\delta$.
In other words,  the area of $F|_D$ is bounded below by $d\pi R^2 - (3d-1)\frac{\pi}{S^2} - \delta$.

Denoting the domain of $\mathbf{ F}^0$ by $\mathcal{S}^0$  and the domain of $\mathbf{ F}$  by $\mathcal{S}$, we now prove the following result.
\begin{lemma}\label{components}
There is an unambiguously defined bijection between the components of $\mathcal{S}^0 \smallsetminus D$  and the components of $\mathcal{S} \smallsetminus D$. Moreover, if $c_0$ and $c$ are corresponding components of $\mathcal{S}^0 \smallsetminus D$  and  $\mathcal{S} \smallsetminus D$, respectively, then $c_0$ and  $c$ have the same number of marked points and the collections of curves $\mathbf{ c}_0=\mathbf{ F}^0|_{c_0}$ and $\mathbf{ c}=\mathbf{ F}|_{c}$ have same total symplectic areas and intersection numbers with $\Sigma$.
\end{lemma}

\begin{proof}
Choose a path $\bar{p}(t) = \{p_i(t)\}_{i=1}^{2d} \in \Lambda(E \smallsetminus \Sigma)$ from the points $\{p^0_i\}_{i=1}^{2d}$ to the points $\{p_i\}_{i=1}^{2d}$
such that the spaces $${\mathcal N}(J^N, \bar{p})= \{([f_N(t)],t) \mid [f_N(t)] = {\mathcal M}_d(J^N, p_1(t), \dots, p_{2d}(t))\text{ and } t \in [0,1]\}$$
are as in Proposition \ref{path1}, where $f_N(0)= f_N^0$ and $f_N(1)=f_N$. Now choose a subsequence $N_j \to \infty$, such that sequences $f^0_{N_j}$ and $f_{N_j}$ both converge.

Corollary \ref{closeimage} implies that, for $j$ sufficiently large, up to reparameterization, the curves $f_{N_j}(t)$, for all $t \in [0,1]$, are all $C^{\infty}$-close to one another when restricted to the preimage of $K$. Thus we can identify the components of $S^2 \smallsetminus (f_{N_j}(0))^{-1}(K)$ with the components of $S^2 \smallsetminus (f_{N_j}(t))^{-1}(K)$, continuously in $t$.
As marked points  must map to constraint points in $E \smallsetminus \Sigma$, they cannot enter the preimage of $K$ and so corresponding components of $S^2 \smallsetminus(f^0_{N_j})^{-1}(K)$ and $S^2 \smallsetminus(f_{N_j}(t))^{-1}(K)$  have the same number of marked points for all $t \in [0,1]$. Corresponding components also have approximately the same boundary image in $K$. Thus, the symplectic areas of the images of corresponding components are arbitrarily close for large $j$, and the intersections numbers of corresponding components with $\Sigma$ are constant in $t$ for large $j$. Passing to the limit $j\to \infty$ the result follows.

\end{proof}

It will be useful to group the components of  $\mathcal{S} \smallsetminus D$ into two subsets, $\mathcal{C}$,  the collection of components which contain no marked points, and $\tilde{\mathcal{C}}$, the remaining components.  Set $\mathbf{ C}= \mathbf{ F}|_{\mathcal{C}}$.

\begin{lemma}\label{delta}
For any $\delta>0$ and $r$ sufficiently close to $1/S$  the total area of the curves in $\mathbf{C}$ is less than $2\delta$.
\end{lemma}

\begin{proof}
Let $\mathcal{C}^0$ be the collection of components of $\mathcal{S}^0 \smallsetminus D$ which contain no marked points and set $\mathbf{ C}^0 = \mathbf{ F}^0|_{\mathcal{C}^0}$.
By Lemma \ref{components} it suffices to show that the statement of the lemma holds for $\mathbf{C}^0$ in place of $\mathbf{C}$.

It follows from our choice of $K$, that the total area of $\mathbf{ F}^0|_{\mathcal{S}^0 \smallsetminus  D}$ is at most $$(d-1)\pi \left(\frac{1}{S^2}-r^2\right)+ \frac{2d\pi}{S^2} + \delta.$$
Since the point constraints $p^0_1, \dots , p^0_{2d}$ lie at the center of disjoint Darboux balls in $E \smallsetminus \Sigma$ of radius $r$ and  $J$ agrees with the standard complex structure in
these balls, we can invoke the monotonicity theorem as in Lemma \ref{limit}.
Let $\tilde{\mathcal{C}}^0$ be the  components of  $\mathcal{S} ^0\smallsetminus D$ not in $\mathcal{C}^0$ (which therefore contain all the marked points) and set $\tilde{\mathbf{ C}}^0 = \mathbf{ F}^0|_{\tilde{\mathcal{C}}^0}$.
Monotonicity then implies that the curves in $\tilde{\mathbf{ C}}^0$   have area bounded below by $2d \pi r^2$. From this and the upper bound for the area of $\mathbf{ F}^0|_{\mathcal{S}^0 \smallsetminus  D}$ above, it follows that  for $r$ sufficiently close to $1/S$ the total area of the curves in $\mathbf{C}^0$ is less than $2\delta$, as required.

\end{proof}

At this point we fix our set $K$ with $\delta <\pi r^2/3$. This is possible as long as $r^2 > \left(\frac{9d-3}{9d-2}\right)\frac{1}{S^2}.$

\bigskip

 \noindent{\bf Step 3.} {\it A sufficient condition.} A careful choice of the point constraints $\{p_i\}_{i=1}^{2d}$ will give us enough control over the limiting building $\mathbf{F}$ to establish a condition which is sufficient to imply  Proposition \ref{oneend}. In this step of the proof we describe this condition and establish its sufficiency.

\begin{proposition}\label{alt} If each component of the complement of $D=F^{-1}(K)$ in $\mathcal{S}$ contains at least one marked point, then $F$ has exactly one negative end.
\end{proposition}

\begin{proof}
Arguing by contradiction we assume that each component of the complement of $D$ contains at least one marked point and $F$ has more than one negative end.
Given this  we can find two marked points, say $y_1$ and $y_2$,  which lie in different components of $\mathcal{S} \smallsetminus D$. Choose a path  of constraints $\bar{p}(t) = \{p_i(t)\}_{i=1}^{2d} \in \Lambda( E \smallsetminus \Sigma)$ which switches $p_1$ and $p_2$ and leaves the other points fixed. More precisely, suppose that  $p_i(0)=p_i$ for all $i$, $p_1(1) =p_2$, $p_2(1) =p_1$, and
$p_i(1)=p_i$ for all $i>2$. By Lemmas \ref{regular1} and \ref{regular2}, for our fixed regular $J$ we may assume that the $p_i(t)$ are in general position for all $t$. By Proposition \ref{path1} for each $N$ there exist corresponding families of curves $f_N(t)$ with $([f_N(t)],t) \in {\mathcal N}(J^N, \bar{p})$. Moreover, our choice of the path $\bar{p}$ implies that the curves $f_N(0)$ and $f_N(1)$ intersect the same constraint points. By Proposition \ref{compactCP23}, it follows that
their images coincide.

Now choose a subsequence $N_j$ as in Proposition \ref{path2}. As described in the proof of Lemma \ref{components}, Corollary \ref{closeimage} implies that, for $j$ sufficiently large,
we can identify the components of $S^2 \smallsetminus (f_{N_j}(0))^{-1}(K)$ with the
 components of $S^2 \smallsetminus (f_{N_j}(t))^{-1}(K)$, continuously in $t$. Consider the component of $S^2 \smallsetminus (f_{N_j}(0))^{-1}(K)$ containing the marked point $y_1$. The image of this component  under $f_{N_j}(0)$ intersects $p_1$ and doesn't intersect $p_2$ (as our curves are embedded), whereas the image of the corresponding  component of $S^2 \smallsetminus (f_{N_j}(1))^{-1}(K)$ must intersect $p_2$ and not $p_1$. This implies that the curves $f_{N_j}(0)$ and $f_{N_j}(1)$ must have different images which is the desired contradiction.
\end{proof}

\bigskip

\noindent{\bf Step 4.} {\it Good point constraints.}  We now define and establish the existence of special point constraints $\{p_j\}_{j=1}^{2d}$ which will allow to prove Proposition \ref{oneend} using Proposition \ref{alt}.
The conditions imposed on these points all refer to a new splitting defined near the  exceptional divisor $\Sigma$ which we now describe.
\medskip

\noindent{\bf A splitting near $\Sigma$.}
Recall that the assumption that  $J$ belongs to $\mathcal{J}^{\star}_E$ implies  that $J$ is equal to the standard integrable complex structure in a small open neighborhood $U_{\Sigma}$ of $\Sigma$.
For a sufficiently small $\epsilon_{\Sigma}>0$ we can choose a closed subset $V_{\Sigma}$ of $U_{\Sigma}$ which can be identified with a small disk bundle in the the holomorphic line bundle $L$ of degree $-1$ over $\Sigma = \CC P^1$,
in such a way that $V_{\Sigma} \smallsetminus \Sigma$ equipped with $J$  is biholomorphic to $B^4(\epsilon_{\Sigma}) \smallsetminus \{0\} \subset \CC^2$.

For $N \in \mathbb{N}$ and $M \in [0, \infty)$, let $J^N_M$ denote the result of  stretching $J$ to length $N$ along $\partial E$ and to length $M$ along $\partial V_{\Sigma}$. Note that each $J^N_M$ is biholomorphic to $J^N$ by a biholomorphism which equals the identity away from $V_{\Sigma}$ and simply contracts the stretched ball onto the original one. We will also denote by $J_M$ the almost complex structure on $E^{\infty}_+$ given by stretching $J$ to length $M$ along $\partial V_{\Sigma}$. Again, $J_M$ is biholomorphic to $J$ and so if $J$ is regular on $E^{\infty}_+$ then so are all $J_M$.

In the limit  $M \to \infty$ (with $N$ fixed) we will obtain two new almost-complex manifolds with cylindrical ends which are not symplectizations. One of these will be  the (positive) completion of $V_{\Sigma}$. This is just the full line bundle $L$. The Reeb orbits on the boundary correspond exactly to the fibers of our line bundle, or Hopf fibers on $S^3 = \partial B^4$. Finite energy curves in $L$ are algebraic curves with poles corresponding to asymptotic limits on Reeb orbits.
If the projection of a curve to the zero section $\Sigma$ has degree $a$ then $a \ge 0$, and if the curve has $k$ poles and $l$ zeros counted with multiplicity, and does not cover $\Sigma$, then $k-l =a$. Curves of degree zero are just multiple covers of fibers. If we choose $\epsilon_{\Sigma}$ small with respect to $r$ then the symplectic area (see Definition \ref{symarea}) of a curve in $L$ of degree $a$ is approximately $a\pi r^2$. In particular curves with symplectic area less than $\pi r^2$ must cover fibres. (This fact informs our choice of $\delta< \pi r^2/3$ at the end of Step 2.) We also remark that the integrable complex structure on $L$ is regular in that somewhere injective curves appear in manifolds of dimension equal to their index, as do curves with poles constrained to lie on certain orbits. As well, multiply covered curves cover curves of strictly smaller index.

The other manifold obtained in the limit $M \to \infty$ is the (negative) completion of the complement of $V_{\Sigma}$. Holomorphically this is just $\CC P^2 \# \overline{\CC P^2} \smallsetminus \Sigma$ (in particular the negative cylindrical end is biholomorphic to our punctured ball) and under this identification finite energy curves are closed holomorphic curves in $\CC P^2 \# \overline{\CC P^2}$ with the preimage of $\Sigma$ removed.

If we also let $N \to \infty$ then we will also need to study $E \smallsetminus V_{\Sigma}$ with a positive completion at $\partial E$ and a negative completion at $\partial V_{\Sigma}$. We can identify this with $E^{\infty}_+ \smallsetminus \Sigma$ and so finite energy holomorphic curves here correspond to finite energy curves in $E^{\infty}_+$ with the preimage of $\Sigma$ removed. The index of these curves is given by the expression in Proposition \ref{indexE} where the term $H \cdot \Sigma$ now represents the number of negative ends counted with multiplicity.

\medskip

\noindent{\bf Point constraints in good position.}
To each moduli space ${\cal P}$ of finite energy curves in $E^{\infty}_+ \smallsetminus \Sigma$ of dimension zero we associate a family of moduli spaces of curves in $L$ each of which has a nonnegative even dimension. Let ${\cal B}^{\cal P}_n$ denote the class of curves in $L$ with one positive end (that is, one pole) and virtual deformation index $2(n+1)$ such that the positive end is asymptotic to a negative asymptotic orbit of a curve in ${\cal P}$. Set ${\cal B}_n = \cup_{\cal P} {\cal B}^{\cal P}_n$. As we assume that $J$ restricts to a regular almost-complex structure on $E^{\infty}_+ $ and hence $E^{\infty}_+ \smallsetminus \Sigma$, the set of possible limiting orbits coming from somewhere injective curves in some ${\cal P}$ is of dimension zero. The same is true of the ends coming from multiply covered curves by equation (\ref{coversofH}). Thus, by regularity of the standard complex structure on $L$ and because the limiting Reeb orbits appear in $2$ dimensional families, each moduli space ${\cal B}_n$ has virtual index $2n$.

Following Definition \ref{general2} we will say that
a collection of points $\{p_i\}_{i=1}^{2d}$ in $L$ is in {\it general position} if no somewhere injective finite energy holomorphic curve of genus zero and  virtual index $2k$ has image intersecting more than $k$ of the points.

\begin{definition}\label{good}
A set of points $\{p_i\}_{i=1}^{2d}$ in $V_{\Sigma}$ is in {\it good position relative to $J$} provided the following conditions hold.
\renewcommand{\theenumi}{(\roman{enumi})}
\begin{enumerate}
\item The set $\{p_i\}_{i=1}^{2d}$ is in general position relative to  $J^N_M$ on $(\CC P^2 \# \overline{\CC P^2})^N$ and to $J_M$ on $E^{\infty}_+$ for all  $N \in \Bbb N$ and $M \in [0,\infty)$.
\item The set $\{p_i\}_{i=1}^{2d}$ is in general position relative to the standard complex structure on  $L$.
\item No curve representing ${\cal B}_n$ has image intersecting $n+1$ of the $\{p_i\}_{i=1}^{2d}$.
\end{enumerate}
\end{definition}

\begin{lemma} \label{goodexist} A generic set of points $\{p_i\}_{i=1}^{2d}$ in $V_{\Sigma}$ is in good position relative to $J$.
\end{lemma}

\begin{proof}
We study the points which are not in good position. By Lemma \ref{regular1} the points which are not in good position for a given $J^N_M$ is of codimension $2$. Therefore letting $M$ vary, sets of points failing to satisfy condition $(i)$ for the $J^N_M$ are of codimension $1$. Similarly, by Lemma \ref{regular2}, recalling that each $J_M$ is regular (as they are biholomorphic to $J$), points which are not in good position for each $J_M$ are also of codimension $2$ and so points failing to satisfy condition $(i)$ for the $J_M$ are again of codimension $1$. As the complex structure on $L$ is regular, index decreases under multiple covers, and all moduli spaces have even dimension, the points failing to satisfy condition $(ii)$ have codimension $2$ as in Lemma \ref{regular2}. Finally, as ${\cal B}_{n}$ has dimension $2n$ the same argument as Lemma \ref{regular2} shows that points failing to satisfy $(iii)$ also have codimension $2$.
\end{proof}

 \begin{remark}\label{buildgs} We are really interested in curves in $L$ whose poles correspond to negative ends  of index zero holomorphic {\it buildings} in $E^{\infty}_+ \smallsetminus \Sigma$. We recall that a holomorphic building consists of finite energy curves in $E^{\infty}_+ \smallsetminus \Sigma$ and the symplectizations $SE$ and $SV_{\Sigma}$ of $\partial E$ and $\partial V_{\Sigma}$ respectively, with corresponding ends identified. As the Reeb orbits on $\partial V_{\Sigma} = S^3$ come in $2$-parameter families, the virtual index of such a holomorphic building is defined to be the sum of the indices of the various curves minus $2$ times the number of limits required to match on $\partial V_{\Sigma}$. With this definition, a convergent sequence of finite energy curves in $E^{\infty}_+ \smallsetminus \Sigma$ converges to a holomorphic building of the same index.

Now, by Proposition \ref{indexSE} curves in $SE$ have strictly positive index unless they are multiple covers of a trivial cylinder. Similarly, curves in $SV_{\Sigma}$ have index strictly greater than $2$ times the number of positive ends unless the curve is a cover of a trivial cylinder. (This can be seen intuitively without an explicit calculation. Indeed, for algebraic curves we always have freedom to move the locations of the poles in their $2$-dimensional families of Reeb orbits, and if the curve does not cover a trivial cylinder then scaling by non-zero complex numbers also acts nontrivially.) As the almost-complex structure on $E^{\infty}_+$ and hence $E^{\infty}_+ \smallsetminus \Sigma$ is assumed to be regular, by Lemma \ref{h} all finite energy curves have nonnegative index. Together, these inequalities imply that the negative limits of finite energy curves of index $0$ in $E^{\infty}_+ \smallsetminus \Sigma$ are exactly the negative limits of finite energy buildings of index $0$ in $E^{\infty}_+ \smallsetminus \Sigma$.
\end{remark}

\bigskip

\noindent{\bf Step 5.} {\it Proving Proposition \ref{oneend} by contradiction.} In this stage of the argument we begin the formal proof of a result which will imply  Proposition \ref{oneend} and hence Theorem \ref{key}. We also outline the argument to come and introduce a useful measure of distance.

Given Lemma \ref{goodexist}, we can choose points $\{p_i\}_{i=1}^{2d} \subset V_{\Sigma}$ in good position relative to $J$. With these points set, we view $F$ as the unique top level curve of the limit $\mathbf{F}$ of a convergent sequence of curves $f_N$ representing  ${\mathcal M}_d(J^N, p_1, \dots, p_{2d})$ as in Lemma \ref{move}. Let $K$ be the subset of $(\C P^2 \smallsetminus E)^{\infty}_-$ from Step 2. By Proposition \ref{alt} we will be done if we can prove the following.

\begin{proposition} \label{mkd} Each component of the complement of $D=F^{-1}(K)$ in $\mathcal{S}$ contains at least one marked point.
\end{proposition}

Arguing by contradiction, we assume that  $\mathcal{C}$, the collection of components of $\mathcal{S} \smallsetminus D$ which contain no marked points, is nonempty. We will derive a contradiction from this as follows. In the next step we will use the splitting along $\partial V_{\Sigma}$ to find, for any sufficiently small $\epsilon>0$, a sequence of curves $f^{\epsilon}_N$ a fixed {\it distance} $\epsilon$ from the candidate $F$.
In the last  step of the proof we will consider a general limit point  $\mathbf{ F}^{\epsilon}$ of the curves $f^{\epsilon}_N$ and prove that for a generic choice of $\epsilon>0$ no such limits can  exist. This  contradicts
the compactness theorem of \cite{BEHWZ}. The assumption that $\mathcal{C}$ is nonempty will be invoked twice in this process, once in each of the two steps to come.

Before proceeding, we observe some immediate implications of our assumption that $\mathcal{C}$ is nonempty. Let $\tilde{\mathcal{C}}$ be the remaining components of  $\mathcal{S} \smallsetminus D$ and set $\mathbf{ C} = \mathbf{ F}|_{\mathcal{C}}$  and $\tilde{\mathbf{ C}} = \mathbf{ F}|_{\tilde{\mathcal{C}}}$. As described in Remark \ref{new1}, every component of $\mathcal{C}$  must  include the domain of a curve of $\mathbf{ F}$ which intersects $\Sigma$. We denote by  $\mathbf{C}\cdot \Sigma $ the total number of these intersections, counted with multiplicity.  Since $J$ belongs to $\mathcal{J}^{\star}_E$, the exceptional divisor $\Sigma$ is itself $J$-holomorphic.
We also recall that, by Lemma \ref{delta} and our choice of $\delta<\pi r^2/3$, the total symplectic area of the curves in $\mathbf{ C}$ is less than $2\pi r^2/3$. In particular no curve of $\mathbf{ C}$ can cover $\Sigma$. So, by positivity of intersection, our present assumption implies that $\mathbf{C}\cdot \Sigma >0.$

\medskip

\noindent{\bf A measure of the distance from $F|_D$.} It will be useful to consider curves which are close to $F$ when restricted to the preimage of $K$. We make this precise using a map $d_K$ defined as follows.
Let $G$ be a smooth map from a Riemann surface to $(\C P^2 \smallsetminus E)^{\infty}_-$ or to $(\CC P^2 \# \overline{\CC P^2})^N$ where $N$ is large enough for this set to include $K$. Let $D_G = G^{-1}(K)$.
If $G$ does not intersect $\partial K$ transversally or $D_G$ is not diffeomorphic to $D$ then we set $d_K(G) = \infty$. Otherwise, we set  $$d_K(G) = \inf \|G\circ \sigma - F|_D\|_{C^{\infty}},$$
where the infimum is over all diffeomorphisms $\sigma : D \to D_G$ and the norm is defined by fixing metrics on $D$ and $K$.

For example, consider the case of the maps $f_N$. Since they converge to $\mathbf{F}$ in the sense of \cite{BEHWZ}, for $N$ large there are (approximately holomorphic) diffeomorphisms $\sigma_N : D \to f_N^{-1}(K)$ such that $f_N \circ \sigma_N$ is $C^{\infty}$-close to $F$.  In particular $d_K(f_N) \to 0$ as $N \to \infty$.

\bigskip

\noindent{\bf Step 6.} {\it Detecting the curves $f^{\epsilon}_N$.} At this point we establish the existence of the curves $f^{\epsilon}_N$ whose limits will lead us to our contradiction. For $\epsilon>0$, we define
$$
{\mathcal U}^N_M(\epsilon)=\{f\mid  [f] ={\mathcal M}_d(J^N_M, p_1, \dots, p_{2n}), \text{ and  } d_K(f) < \epsilon\}.
$$
Note that for $N$ sufficiently large $f_N \in {\mathcal U}^N_0(\epsilon)$.

\begin{proposition}\label{mty} Given $\epsilon>0$ sufficiently small, for all $N$ sufficiently large, there exists an $M_0=M_0(\epsilon,N)$ such that ${\mathcal U}^N_{M_0}(\epsilon)$ is empty.
\end{proposition}

\begin{proof}
Since  our points are in general position for each $J^N_M$, by Proposition \ref{path1}  we can find a smooth family of curves $f^N_M$ for $M \in [0,\infty)$ such that  $f^N_0 = f_N$ and  $[f^N_{M}] = {\mathcal M}_d(J^N_M, p_1, \dots, p_{2n})$. Here we recall that the $J^N_M$ are all biholomorphic to the fixed almost-complex structure $J^N$, and so varying $M$ is equivalent to moving the point constraints in $V_{\Sigma} \subset (\CC P^2 \# \overline{\CC P^2})^N$.

Arguing by contradiction we assume that the lemma is false. In this case, for an $N$ which we may assume to be arbitrarily large, $d_K(f^N_{M}) < \epsilon$ for all $M$, and so setting $D^N_{M} = (f^N_{M})^{-1}(K) \subset S^2$ there exist diffeomorphisms $\sigma^N_{M} \colon D \to D^N_{M}$ such that
\begin{equation}
\label{mclose}
\|f^N_{M} \circ \sigma^N_{M}-F|_D\|_{C^{\infty}} < \epsilon \,\,\,\text{  for all  } M.
\end{equation}

Now choose a sequence $M_i \to \infty$ and curves $f^N_{M_i} \in {\mathcal U}^N_{M_i}(\epsilon)$ which converge in the sense of \cite {BEHWZ} to a building $\mathbf{ F}^N$ with domain $\mathcal{S}^N$.
As the $f^N_{M_i}$ converge to $\mathbf{ F}^N$ uniformly on compact sets,
there is a curve $F^N$ of $\mathbf{ F}^N$ such that  $d_K(F^N|_{D^N})\leq \epsilon$ for $D^N = (F^N)^{-1}(K)$. Arguing as in Lemma \ref{components} one gets the following.
\begin{lemma}\label{componentsN}
For sufficiently large $N$, there is an unambiguously defined bijection between the components of $\mathcal{S} \smallsetminus D$  and the components of $\mathcal{S}^N \smallsetminus D^N$. Corresponding components have the same number of marked points and their images under $\mathbf{ F}$ and $\mathbf{ F}^N$ have same intersection numbers with $\Sigma$. The difference between the symplectic areas of corresponding components is of order $\epsilon$  as $N \to \infty$.
\end{lemma}

\begin{proof} By the nature of the convergence of the $f_N$ to $\mathbf{ F}$, for all sufficiently large $N$ there is natural bijection between the components of $\mathcal{S} \smallsetminus D$  and the components of $S^2 \smallsetminus (f_N)^{-1}(K)$ such that corresponding components have the same number of marked points, their images under $\mathbf{ F}$ and $f_N$ have same intersection numbers with $\Sigma$, and  the difference between the symplectic areas of corresponding components goes to zero as $N \to \infty$. Hence it suffices to prove the lemma with $\mathcal{S} \smallsetminus D$ replaced by $S^2 \smallsetminus (f_N)^{-1}(K)$ and $\mathbf{ F}$ replaced by $f_N$.

 Fixing a large $i$ one replaces the role of the family $\{f_N(t)\}_{t \in [0,1]}$ in the proof of Lemma \ref{components} with the family $\{f^N_M\}_{M \in [0,M_i]}$. Starting with inequality \eqref{mclose}, one can then establish the desired correspondences between the components of $S^2 \smallsetminus (f_N)^{-1}(K)$ and those of $S^2 \smallsetminus D^N_{M_i}$ as well as the equality of the intersection numbers with $\Sigma$ for corresponding components. The statement about the symplectic areas of corresponding components also follows easily from \eqref{mclose}. Passing to the limit $i \to \infty$ this time,  the proof is complete.

\end{proof}

The building $\mathbf{ F}^N$ has components mapping to $(\CC P^2 \# \overline{\CC P^2})^N \smallsetminus \Sigma$ and to $L$.
Denote by $\mathcal{C}^N$ and $\tilde{\mathcal{C}}^N$ the collections of the components of $\mathcal{S}^N \smallsetminus D^N$ which correspond, via the bijection of Lemma \ref{componentsN} to $\mathcal{C}$ and $\tilde{\mathcal{C}}$, respectively. Set
$\mathbf{ C}^N= \mathbf{ F}^N|_{\mathcal{C}^N}$  and $\tilde{\mathbf{ C}}^N = \mathbf{ F}^N|_{\tilde{\mathcal{C}}^N}$.

For $\epsilon>0$ sufficiently small, it follows from Lemma \ref{componentsN}, Lemma \ref{delta}, and our choice of $\delta$ in Step 2 that for all sufficiently large $N$ the total symplectic area of the curves in $\mathbf{ C}^N$ must also be less than $\pi r^2$. Henceforth we will assume that $\epsilon$ has  been chosen small enough for this to hold. With this, it follows that for large enough $N$ the curves of  $\mathbf{ C}^N$ with image in $L$ are all multiple covers of fibers (see Step 4).

Lemma \ref{componentsN} also implies that  $\mathbf{C}^N \cdot \Sigma = \mathbf{C}\cdot \Sigma$ where $\mathbf{C}^N \cdot \Sigma$ is the number of intersections between $\Sigma$ and the curves of $\mathbf{ C}^N$ counted with multiplicity.  Put another way, for large $N$  the curves of $\mathbf{C}^N$ which have image in $L$ are all covers of fibres and there are $\mathbf{C}\cdot \Sigma$ such fibers (when counted with multiplicity).

As $F$ is embedded and $F^N|_{D^N}$ is close to $F$, the boundaries of $\mathbf{ C}^N$ and $\tilde{\mathbf{ C}}^N$ are disjoint in $\partial K$.
We define the intersection number $\mathbf{ C}^N\cdot \tilde{\mathbf{ C}}^N$ by compactifiying the constituent curves to get maps into the complement $K^c$ of $K$ in $(\C P^2 \# \overline{\C P^2})^N$. Then, as the maps have disjoint boundaries the intersection number can be defined as usual and must be $0$ as we are dealing with limits of embedded curves.

Before completing the proof of Proposition \ref{mty}, we establish some additional restrictions on the curves of $\mathbf{ C}^N$  and $\tilde{\mathbf{ C}}^N$ with image in $L$.

\begin{lemma} \label{inL} For all sufficiently large $N$ the following statements hold. The curves of $\mathbf{C}^N$ with image in $L$ all cover the same fibre, $V$. All curves of  $\tilde{\mathbf{C}}^N$ with image in $L$ either have a single positive end asymptotic to the fiber $V$ or are covers of $\Sigma$. At least one of these curves has positive degree.
\end{lemma}

\begin{proof}
First we show that all curves of  $\tilde{\mathbf{C}}^N$ with image in $L$ and not covering $\Sigma$ have a single positive end. More generally, we show in fact that any component of $\tilde{\mathbf{C}}^N$ in $L$ consists of curves only one of which can have a single positive end. This follows from genus considerations as in Lemma \ref{limit}. Indeed, as $\mathbf{ F}^N$ is a limit of curves of genus zero the compactifications of its curves must fit together to form a map from $S^2$ to $(\C P^2 \# \overline{\C P^2})^N$. But suppose that a component with image in $L$ has more than one positive end. As $D^N$ is connected this implies that $\mathbf{ F}^N$ has a curve in $(\C P^2 \# \overline{\C P^2})^N \smallsetminus \Sigma$ with image contained in $K^c$. By the removable singularity theorem any such curves can be compactified to give a closed curve with image contained in $K^c$. This compactified closed curve  is therefore homologous to a multiple of $[\Sigma]$. However, it cannot cover $\Sigma$ itself (as the original  curve has image  in $(\C P^2 \# \overline{\C P^2})^N \smallsetminus \Sigma$).  This is a contradiction to positivity of intersection.

Fixing a suitably large $N$ we may assume that \ the curves of  $\mathbf{ C}^N$ with image in $L$ are all multiple covers of fibers. Let $V$ be one of these fibres. Consider a curve $H$ of $\tilde{\mathbf{ C}}^N$ with image in $L$. As described in Step 4, $H$ either has degree zero, in which case it covers a fibre, or it has positive degree.

Suppose $H$ has positive degree and is not a cover of $\Sigma$. We claim that $H$ and $V$ cannot intersect. Indeed, if there were an intersection point then as $H$ and $V$ have distinct images (as $H$ has positive degree) the intersection point must be isolated. Also as $H$ has positive degree it is the only curve in its component of $\tilde{\mathbf{C}}^N$ in $L$ with a positive end, and so the other curves are all covers of $\Sigma$. These also intersect $V$ at an isolated point. Therefore by positivity of intersection, any such intersections would imply  self-intersections of the $f^N_{M_i}$ for large $i$, and as the $f^N_{M_i}$ are embedded this is a contradiction. From this we see that
the single positive end of $H$ must be asymptotic to the fibre $V$. We also see that the corresponding component contains no copies of $\Sigma$.

Next suppose that $H$ is a cover of $\Sigma$. Its component of $\tilde{\mathbf{C}}^N$ in $L$ has a single curve with a positive end and from the last paragraph we see that this curve has degree $0$. By positivity of intersection again it must be a cover of the fiber $V$.

Similarly, any $H$ of positive degree not covering $\Sigma$ cannot intersect any fibre in $L$ covered by a degree $0$ curve in $\tilde{\mathbf{ C}}^N$, and curves which cover $\Sigma$ can only intersect degree $0$ curves of $\tilde{\mathbf{ C}}^N$ which cover $V$.

In summary, if the set of curves of  $\tilde{\mathbf{ C}}^N$ with image in $L$ and positive degree is nonempty, then every  curve of $\mathbf{C}^N$ with image in $L$ must cover the same fibre, $V$, every curve of  $\tilde{\mathbf{C}}^N$ with image in $L$ and degree zero must also cover $V$, and every curve of  $\tilde{\mathbf{C}}^N$ with image in $L$ and positive degree is either a cover of $\Sigma$ or has a single positive end which is asymptotic to the fiber $V$.

It remains to show that there must be a curve of $\tilde{\mathbf{ C}}^N$ with image in $L$ and positive degree. Assume that  all the curves of $\tilde{\mathbf{C}}^N$ in $L$ have degree zero and hence cover fibres of $L$. In this case, since we have $$\tilde{\mathbf{C}}^N \cdot \Sigma = d-1-\mathbf{C}^N \cdot \Sigma,$$  there are at most $d-1-\mathbf{C}^N \cdot \Sigma$ curves of $\tilde{\mathbf{C}}^N$ with image in $L$. Since our constraint points are in good position relative to $J$ it follows, from condition (ii) of the definition of being in good position, that each curve of  $\tilde{\mathbf{C}}^N$ with image in $L$ can intersect only one of the constraint points. Since the curves of $\tilde{\mathbf{C}}^N$ must hit all $2d$ constraints this is a contradiction.

\end{proof}

We can now complete the proof of Proposition \ref{mty}. Choose $\epsilon>0$ as above and let $N$ be large enough for Lemma \ref{inL} to hold.
By the removable singularity theorem, the curves of $\mathbf{ C}^N$ and $\tilde{\mathbf{ C}}^N$ in $(\C P^2 \# \overline{\C P^2})^N \smallsetminus \Sigma$ can be completed  to give collections of curves in $(\C P^2 \# \overline{\C P^2})^N$ with disjoint boundaries on $\partial K$, say $\underline{\mathbf{ C}}^N$ and $\underline{\tilde{\mathbf{ C}}}^N$. The intersection number of these curves is then given by the formula
\begin{equation}
\label{inter}
\underline{\mathbf{ C}}^N \cdot \underline{\tilde{\mathbf{ C}}}^N = \mathbf{ C}^N \cdot \tilde{\mathbf{ C}}^N - k \mathbf{ C}^N \cdot \Sigma.
\end{equation}
where  $k$ is the sum of the degrees of the curves of $\tilde{\mathbf{ C}}^N$ in $L$ (including any mapping to the symplectization $SV_{\Sigma} = L \setminus \Sigma$). This holds since $\underline{\tilde{\mathbf{ C}}}^N$ is homologous to $\tilde{\mathbf{ C}}^N - k \Sigma$ relative to its boundary.
As described above, we have $\mathbf{ C}^N \cdot \tilde{\mathbf{ C}}^N=0$ since $\mathbf{ F}^N$ is a limit of embedded curves. Lemma \ref{inL} implies that $k >0$ and Lemma \ref{componentsN}  implies that  $\mathbf{ C}^N \cdot \Sigma= \mathbf{ C} \cdot \Sigma$. Finally, our assumption that $\mathcal{C}$ is nonempty  implies that $\mathbf{ C} \cdot \Sigma >0$.
Hence, from equation \eqref{inter} we derive the inequality
 $$
 \underline{\mathbf{ C}}^N \cdot \underline{\tilde{\mathbf{ C}}}^N<0.
 $$
As $\underline{\mathbf{ C}}^N$ and $\underline{\tilde{\mathbf{ C}}}^N$ are genuine holomorphic curves this contradicts positivity of intersection and therefore  completes the proof of Proposition \ref{mty}.

\end{proof}

Using Proposition \ref{mty} we now detect the curves $f^{\epsilon}_N$. Consider again the smooth family of curves $f^N_M$ for $M \in [0,M_0]$ such that $[f^N_{M}]$ is equal to ${\mathcal M}_d(J^N_M, p_1, \dots, p_{2n})$.


\begin{lemma}\label{boundary} Given  $\epsilon>0$ sufficiently small, for all $N$ sufficiently large there exists an $M_1^{\epsilon}=M_1^{\epsilon}(N)>0$ and a curve $f^{\epsilon}_N$ such that $[f^{\epsilon}_N] = {\mathcal M}_d(J^N_{M_1^{\epsilon}}, p_1, \dots, p_{2n})$
and $d_K(f^{\epsilon}_N) =\epsilon$.
\end{lemma}

\begin{proof}
Let $M_1^{\epsilon}$ be the minimal $M \in [0,M_0]$ such that $d_K(f^N_M) \ge \epsilon$. Such an $M_1^{\epsilon}$ exists by Proposition \ref{mty}. We claim that $d_K(f^N_{M_1^{\epsilon}}) = \epsilon$ and so setting $f^{\epsilon}_N=f^N_{M_1^{\epsilon}}$ we will be done. By the definition of $d_K$, it is sufficient to show that $(f^N_{M_1^{\epsilon}})^{-1}(K)$ has a single component and $f^N_{M_1^{\epsilon}}$ is transverse to $\partial K$. By the definition of $M_1^{\epsilon}$ we know that $(f^N_M)^{-1}(K)$ has a single component for $M<M_1^{\epsilon}$. Hence if $(f^N_{M_1^{\epsilon}})^{-1}(K)$ had multiple components then all but one of these must be intersections with $K$ of holomorphic curves tangent to $\partial K$ from the outside. But, by the maximum principle, no such tangencies can occur since $K$ was chosen, in Step $2$, to be a subset of $(\C P^2 \smallsetminus E)^{\infty}_-$ of the form  $\C P^2 \smallsetminus E \cup( \partial E \times [\tau,0])$ . Next, if $f^N_{M_1^{\epsilon}}$ is somewhere tangent to $\partial K$ then, provided $\epsilon$ is chosen sufficiently small, as $F$ is transverse to $\partial K$ we must also have $d_K(f^N_M) \ge \epsilon$ for $M$ slightly less than $M_1^{\epsilon}$. This is also a contradiction.
\end{proof}

\bigskip

\noindent{\bf Step 7.} {\it Completion of the proof of Proposition \ref{mkd}.} To complete  the proof of Proposition \ref{mkd} we first make the following simple but crucial observation.

\begin{lemma}
For almost every  $\epsilon>0$ there are no  finite energy $J$-holomorphic curves $\widetilde{F}$  in $(\C P^2 \smallsetminus E)^{\infty}_-$ with deformation index zero and $d_K(\widetilde{F}) = \epsilon$.
\end{lemma}

\begin{proof} There are only countably many moduli spaces of such curves, which by regularity and Lemma \ref{f} are of dimension zero. Hence $d_K$ takes countably many finite values on these spaces. \end{proof}

With this, we can choose our $\epsilon>0$ to be arbitrarily small, so that, for example Lemma \ref{boundary} holds,  and we may assume that there are no  rigid  finite energy $J$-holomorphic curves $\widetilde{F}$  in $(\C P^2 \smallsetminus E)^{\infty}_-$ with  $d_K(\widetilde{F}) = \epsilon$. By  Lemma \ref{boundary},  there is a sequence of  curves $f^{\epsilon}_N$ with $[f^{\epsilon}_N] = {\mathcal M}_d(J^N_{M_1^{\epsilon}}, p_1, \dots, p_{2n})$ and $d_K(f^{\epsilon}_N)=\epsilon$. Passing to a subsequence, if necessary, we may assume that the $f^{\epsilon}_N$ converge as $N \to \infty$ to a limiting building $\mathbf{ F}^{\epsilon}$. We will prove that any such limit $\mathbf{ F}^{\epsilon}$ must have a rigid top level curve $F^{\epsilon}$ with $d_K(F^{\epsilon}) = \epsilon$. Since such curves are forbidden by our choice of $\epsilon$ we will have arrived at the desired contradiction.

There are two cases to consider.

\medskip

\noindent{\it Case 1.}
In the first case we suppose that  $M_1^{\epsilon}( N)$ remains bounded as $N \to \infty$. It can then be assumed to converge to, say,  $\overline{M}_1$.
The components of $\mathbf{ F}^{\epsilon}$ with image in $E^{\infty}_+$ are then holomorphic with respect to $J_{\overline{M}_1}$. As our $2d$ points are in general position with respect to this almost-complex structure (by condition $(i)$ of being in good position relative to $J$), it follows from Proposition \ref{index} that the component  of the limit with image in $(\C P^2 \smallsetminus E)^{\infty}_-$, say $F^{\epsilon}$,  is rigid. However,  by uniform convergence on compact subsets we also have $d_K(F^{\epsilon})= \epsilon$ which contradicts our choice of $\epsilon$ above.

\medskip

\noindent{\it Case 2.}
If  $M_1^{\epsilon}(N)$ is unbounded we may assume that $M_1^{\epsilon}(N) \to \infty$  as $N \to \infty$. In this case ${\mathbf{ F}}^{\epsilon}$ has components with images in $(\C P^2 \smallsetminus E)^{\infty}_-$, $E^{\infty}_+ \smallsetminus \Sigma$ and $L$.  It may also have components with images in the symplectizations $SE$ and $SV_{\Sigma}$ but, for convenience, we will group  these with the components in $E^{\infty}_+ \smallsetminus \Sigma$, see Remark \ref{buildgs}.

We first observe that $\mathbf{ F}^{\epsilon}$ has exactly one component with image in $(\C P^2 \smallsetminus E)^{\infty}_-$. To see this note
that at least one component of $\mathbf{ F}^{\epsilon}$ with image in $(\C P^2 \smallsetminus E)^{\infty}_-$, say $F^{\epsilon}$,  is close (in the sense of $d_K$) to $F$. The curve  $F^{\epsilon}$ therefore has degree $d$. Since $\mathbf{ F}^{\epsilon}$ has total degree $d$ and curves in $(\C P^2 \smallsetminus E)^{\infty}_-$ of nonpositive degree do not exist (they would have negative area), $F^{\epsilon}$ is the only such curve. As the almost-complex structure is regular, it also follows from Lemma \ref{f} that all the ends  of $F^{\epsilon}$ are asymptotic to covers of $\gamma_1$.
Again, by our choice of the curves $f^{\epsilon}_N$ with $d_K(f^{\epsilon}_N) =\epsilon$ and the fact that they converge  to $\mathbf{ F}^{\epsilon}$ uniformly on compact sets, we get  $d_K(F^{\epsilon}) =\epsilon$. To derive a contradiction as in the previous case it now suffices for us  to show that the index of $F^{\epsilon}$ is zero. To do this we must first manage the other curves of $\mathbf{ F}^{\epsilon}$ more carefully and then invoke our assumption that $\mathcal{C}$ is not empty.

 Let ${\mathcal{S}}^{\epsilon}$ be the domain of $\mathbf{ F}^{\epsilon}$.  Arguing as in the proof of Lemma \ref{componentsN}, we can conclude that  there is a subset ${D}^{\epsilon}$ of ${\mathcal{S}}^{\epsilon}$ which is diffeomorphic to $D$, and an unambiguously defined bijection between the components of $\mathcal{S} \smallsetminus D$ and those of ${\mathcal{S}}^{\epsilon} \smallsetminus {D}^{\epsilon}$.
 Denote the collections of the components of $\mathcal{S}^{\epsilon} \smallsetminus D^{\epsilon}$ which correspond to $\mathcal{C}$ and $\tilde{\mathcal{C}}$, by $\mathcal{C}^{\epsilon}$ and $\tilde{\mathcal{C}}^{\epsilon}$, respectively, and set
$\mathbf{ C}^{\epsilon}= \mathbf{ F}^{\epsilon}|_{\mathcal{C}^{\epsilon}}$  and $\tilde{\mathbf{ C}}^{\epsilon} = \mathbf{ F}^{\epsilon}|_{\tilde{\mathcal{C}}^{\epsilon}}$. As before, corresponding components have the same number of marked points,  the curves of $\mathbf{ C}^{\epsilon}$  intersect the exceptional divisor $\mathbf{ C} \cdot \Sigma$ times, counted with multiplicity, and the curves of $\mathbf{ C}^{\epsilon}$  have total area less than  $\pi r^2$ (as $\epsilon$ is arbitrarily small).  With this area bound  and the uniqueness of the top-level curve $F^{\epsilon}$ established above, one can argue precisely as in Lemma \ref{inL} to prove the following.

\begin{lemma} \label{inL2}  The curves of $\mathbf{C}^{\epsilon}$ with image in $L$ all cover the same fibre, $V$. All curves of  $\tilde{\mathbf{C}}^{\epsilon}$ with image in $L$ either have a single positive end asymptotic to the fiber $V$ or are covers of $\Sigma$, and at least one of these curves has positive degree.
\end{lemma}

We also have the following additional constraint on $\mathbf{C}^{\epsilon}$.

\begin{lemma} \label{domainzero}
The curves of $\mathbf{C}^{\epsilon}$ with image in  $E^{\infty}_+ \smallsetminus \Sigma$ have deformation index equal to zero.
\end{lemma}
\begin{proof}

As described above, $\mathbf{ F}^{\epsilon}$ has a single curve, $F^{\epsilon}$,   in $(\C P^2 \smallsetminus E)^{\infty}_-$ and its ends are all asymptotic to multiples of $\gamma_1$. Since $\mathbf{ F}^{\epsilon}$ is the limit of curves of genus zero, the curves of $\mathbf{ F}^{\epsilon}$ with image in $E^{\infty}_+ \smallsetminus \Sigma$ must all have a single positive end  and these must be asymptotic to multiples of $\gamma_1$. Let $G^{\epsilon}$ be a curve of $\mathbf{C}^{\epsilon}$ with image in  $E^{\infty}_+ \smallsetminus \Sigma$ and suppose that the positive end of $G^{\epsilon}$ covers $\gamma_1$ a total of $a^+$ times, and the number of negative ends of $G^{\epsilon}$, counted with multiplicity, is $a^-$.  The index of $G^{\epsilon}$ is then $2(a^+-a^-)$ (see Proposition \ref{indexE} and the description of $E^{\infty}_+ \smallsetminus \Sigma$ in Step 4). It now suffices to show that $a^+=a^-$. As it is a curve with image in $E^{\infty}_+ \smallsetminus \Sigma$,  the area of $G^{\epsilon}$ is $\pi a^+ / S^2 - \pi a^- r^2$. Since this area is bounded from above by  $\pi r^2$ and from below by $0$ we then have $$a^- r^2 S^2 < a^+ < (1 +a^- )r^2 S^2.$$ Thus, for $r$ close enough to $1/S$ the inequalities above imply that $a^+ = a^-$. (Recalling that $a^-$ is bounded by $\mathbf{C}\cdot \Sigma \le d-1$ it suffices to have $r^2  > \left(\frac{d-2}{d-1}\right)\frac{1}{S^2}$).
\end{proof}

With Lemma \ref{inL2} and Lemma \ref{domainzero} in hand we can now show that the index of $F^{\epsilon}$ must be zero, and thus derive the desired contradiction. Let $x$ and $y$  denote the sums of the indices of  the curves of $\mathbf{ F}^{\epsilon}$ with images in $E^{\infty}_+ \smallsetminus \Sigma$ and $L$, respectively. Here the contribution to $y$ of a curve in $L$ with marked points is defined to be its constrained index.
As described in Remark \ref{buildgs}, the fact that $\mathbf{F}^{\epsilon}$ has index zero  implies that $$\mathrm{index}(F^{\epsilon})+x+y-2p =0$$ where $p$ is the number of positive ends of the curves of $\mathbf{F}^{\epsilon}$ with image in $L$ (and hence the number of components in $L$).  Now, in the present context, our (still) standing assumption that $\mathcal{C}$ is nonempty implies that there is at least one curve, $G^{\epsilon}$, of  $\mathbf{C}^{\epsilon}$ with image in $E^{\infty}_+ \smallsetminus \Sigma$ and another curve, $H^{\epsilon}$, of  $\mathbf{C}^{\epsilon}$ with image in $L$ such that the positive end of $H^{\epsilon}$ is asymptotic to a negative end of $G^{\epsilon}$. Lemma \ref{inL2} applies here to say that every curve of $\mathbf{ F}^{\epsilon}$ with image in $L$ is either a cover of $\Sigma$ or has a single positive end and this is equal to the negative asymptotic limit of $G^{\epsilon} \in \mathbf{ C}^{\epsilon}$. By Lemma \ref{domainzero},  the index of $G^{\epsilon}$ is zero. Therefore,  it follows from condition $(iii)$ of  the definition of being in good position, that any curve of $\tilde{\mathbf{ C}}^{\epsilon}$ with image in $L$ and containing a marked point has constrained index at least $2$ (twice its number of positive ends). The curves of $\mathbf{ C}^{\epsilon}$ in $L$ also have index at least $2$, as do any curves of $\tilde{\mathbf{ C}}^{\epsilon}$ with a positive end but without marked points, and so summing over all curves of  $\mathbf{ F}^{\epsilon}$ with image in $L$ we get $$y-2p \ge 0.$$  This implies that $\mathrm{index}(F^{\epsilon})+x \le 0$. But as the almost complex structures on $(\C P^2 \smallsetminus E)^{\infty}_-$ and $E^{\infty}_+ \smallsetminus \Sigma$ are regular, by Lemma \ref{f} and Lemma \ref{h} both $\mathrm{index}(F^{\epsilon})$ and $x$ must be nonnegative. Thus, we have $\mathrm{index}(F^{\epsilon})=0$ and the desired contradiction.

\medskip

The contradictions at the ends of both these cases complete the proof of Proposition \ref{mkd} and hence Theorem \ref{key}.

\section[The proof of  Theorem \ref{thm2}]{The proof of  Theorem \ref{thm2}}

Let $M = \CC P^2(R) \times \CC^{n-2}$ and denote the obvious split symplectic form on $M$ by $\omega_R$. Let $E(a_1,a_2,\ldots,a_n)$ be the ellipsoid
$$\left\{ (z_1, \ldots, z_n) \in \CC^n \bigmid \sum_{i=1}^n\frac{|z_i|^2}{a_i^2} \leq1 \right\}.$$
Suppose that for any $S>0$ there exists a symplectic
embedding $$\phi(S):E(1,S,\ldots,S) \hookrightarrow M.$$
To prove Theorem \ref{thm2} we must show that this implies
that $R \geq \sqrt{3}$.

Fix an integer $d\geq1$, and a positive real number $S$ such that $S^2$ is irrational and $S^2> 3d$. Set $\phi= \phi(S)$.

\begin{lemma} \label{ellfam} For $t \in [1/S,1]$, there exists a smooth family of symplectic
embeddings
$$\phi_t : E_t=E(u(t),u(t)S,\dots,u(t)S)  \hookrightarrow M$$ such that:
\begin{itemize}
\item $u(t) \colon [1/S,1] \to (0,1]$ satisfies $u(t)=t$ for $t$ close to $1$ and $u(1/S)=1/S$;
  \item for $t$ in some  neighborhood of $1$ the embeddings $\phi_t$ are just the restrictions of $\phi$ to $E_t$;
  \item $\phi_{1/S}$ coincides with the inclusion $$i_S \colon E(1/S, 1, \dots, 1) \to E(1,1) \times \CC^{n-2} \subset  M.$$
\end{itemize}
\end{lemma}

\begin{proof}
This is a simple application of the Extension after Restriction Principle, \cite{eh1}.
For an $\epsilon  \in (0, 1/S]$, let  $u \colon [1/S,1] \to (0,1]$ be any smooth function which equals $1/S$ near $t=1/S$, is nonincreasing  on $[1/S, 1/3]$, is equal to $\epsilon$ on $[1/3,2/3]$, is increasing on $(2/3,1]$, and equals $t$ near $t=1$. This choice fixes the domains $E_t$ of the desired embeddings $\phi_t$.

Now choose an embedded Darboux ball $B^{2n}(\delta)$ in $M$ for some $\delta>0$. Without loss of generality we may assume that $\phi(0)=0 \in B^{2n}(\delta)$ and that the linearization of $\phi$ at zero, in the standard symplectic coordinates on  $E(1, S, \dots, S)$ and $B^{2n}(\delta)$, is the identity.

For $z \in E_t$, set
$$
\phi_t(z) =
\begin{cases}
i_S|_{E_t }     & \text{for $1/S \leq t \leq 1/3$}, \\
\frac{1}{3t-1}\phi((3t-1)z))      & \text{for $1/3 < t < 2/3$},\\
\phi|_{E_t}     & \text{for $2/3\leq t\leq1$}.
\end{cases}
$$
When  $\epsilon$ is sufficiently small the middle piece is then a well-defined continuous isotopy of smooth embeddings  from
$i_S|_{E_{1/3}}$ to $\phi|_{E_{2/3}} = \phi|_{E_{1/3}}$ (with images in  $B^{2n}(\delta)$), as constructed in the  Extension after Restriction Principle. Overall, we have a continuous isotopy of smooth embeddings with the desired properties. Reparameterizing the dependence of the $\phi_t$ on $t$ appropriately, as in say Appendix A of \cite{schl},  we then get the desired smooth isotopy.

\end{proof}

To the family of embeddings $\phi_t$  we will associate a family of moduli spaces of holomorphic curves.  Theorem \ref{key} will allow us to
prove that the moduli space corresponding to $t=1/S$ is nontrivial. With this we will prove that the moduli space for $t=1$ is also nontrivial.
The holomorphic curves which represent this nontrivial space will then yield the proof Theorem \ref{thm2}.

\subsection{Moduli spaces associated to $\phi_t$} \label{threeone}

Before defining our moduli spaces we first compactify an open subset of the target $(M, \omega_R)$ which is large enough to contain the desired curves. We do this so that we may later use the standard compactness theorems.

Let $\CC P^1(2T)$ denote $\CC P^1$ equipped with its standard  symplectic structure multiplied by $4T^2$. We complete the open subset $\CC P^2(R) \times (B^2(T))^{n-2}$ of $M$, by embedding each $B^2(T)$--factor into $\CC P^1(2T)$ as the lower hemisphere. We will denote the resulting manifold by
$$\widehat{M}(T)= \CC P^2(R) \times (\CC P^1(2T))^{n-2}.$$
When convenient we will
equip the $\CC P^1(2T)$--factors of $\widehat{M}(T)$ with complex coordinates $z_3, \dots, z_n$ such that $z_j=0$ corresponds to the center of the appropriate copy of $B^2(T)$.  We will also consider  the $(n-2)$-dimensional torus, $\mathbb{T}^{n-2}$, acting symplectically on $\widehat{M}(T)$ by acting on the $\CC P^1(2T)$ factors in the standard way, by rotations.

Now we need to set the size of $T$. Choose a $T_1=T_1(S)>0$ such that
$$\phi_t(E_t) \subset  \CC P^2(R) \times (B^2(T_1))^{n-2}$$ for all $t \in [1/S,1]$. We first assume that  $T>T_1$. With this, each $\phi_t$ can be viewed as a symplectic embedding of $E_t$ into $\widehat{M}(T)$. We can then consider the
negative symplectic completion of each $\widehat{M}(T) \smallsetminus E_t$ which, as a set, is given by
 $$
 (\widehat{M}(T) \smallsetminus E_t)^{\infty}_- = (\widehat{M}(T) \smallsetminus E_t) \cup (\partial E_t \times(-\infty,0]).
 $$
 (Here, and in what follows, we identify $E_t$ with its image $\phi_t(E_t)$.)

 Now let $$U(T) = \CC P^2(R) \times ((\C P^1(2T))^{n-2} \smallsetminus (B^2(T))^{n-2})$$ and
 $$U(T_1) = \CC P^2(R) \times ((\C P^1(2T))^{n-2} \smallsetminus (B^2(T_1))^{n-2}).$$
 These sets are contained in $\widehat{M}(T) \smallsetminus E_t$ and can thus be considered as subsets of $(\widehat{M}(T) \smallsetminus E_t)^{\infty}_-$.
 We will always view them as such.  However, it will be useful to sometimes identify the complements of these sets (in $(\widehat{M}(T) \smallsetminus E_t)^{\infty}_-$)  as subsets of $(M \smallsetminus E_t)^{\infty}_-$. In particular, we will choose $T$ so that certain holomorphic curves in $(\widehat{M}(T) \smallsetminus E_t)^{\infty}_-$ must be contained in $(U(T))^c$ and can thus be identified with curves in $(M \smallsetminus E_t)^{\infty}_-$.

Let $\mathcal{J}_t(T)$ be the space of almost-complex structures on $(\widehat{M}(T) \smallsetminus E_t)^{\infty}_-$ which turn $(\widehat{M}(T) \smallsetminus E_t)^{\infty}_-$ into an almost-complex manifold with cylindrical end and which are adjusted to the
symplectic form on $\widehat{M}(T) \smallsetminus E_t$ in the sense of \cite{BEHWZ}.


\begin{definition} \label{acs} Let $\mathcal{J}_{t,R}(T) \subset \mathcal{J}_t(T)$ be the collection of almost-complex structures $J$ such that any connected finite energy $J$-holomorphic (cusp) curve in $(\widehat{M}(T) \smallsetminus E_t)^{\infty}_-$  with at least one asymptotic end and area bounded by $d\pi R^2$ has image contained in the interior of $(U(T))^c$.
\end{definition}

\begin{lemma}\label{diamond}
For sufficiently large $T$, the space $\mathcal{J}_{t,R}(T)$ is open and nonempty.
\end{lemma}

\begin{proof}
The fact that each $\mathcal{J}_{t,R}(T)$ is open follows immediately from the compactness theorem for finite energy curves of bounded symplectic area.
It suffices to show that for large enough $T$, $\mathcal{J}_{t,R}(T)$ is nonempty.

Now choose $T$ sufficiently large that any point $p \in \partial (B^2(\frac{T+ T_1}{2})^{n-2}) \subset \C^{n-2}$ lies at the center of a ball of radius $R \sqrt{d}$. For example we may take $T = T_1 + 2R \sqrt{d}$.

Let $J$ be an almost complex structure on $M = \C P^2(R) \times \C^{n-2}$  of the form $J_R \oplus J_0$ where $J_0$ is the standard complex structure on $\C^{n-2}$.
Then the projection to the $\C^{n-2}$ factor of a holomorphic curve in $M$ is holomorphic. Suppose that such a holomorphic curve mapping to $\C^{n-2}$ has boundary contained in $(B^2(T_1))^{n-2} \cup (\C^{n-2} \smallsetminus (B^2(T))^{n-2})$. Then the interior of the curve necessarily intersects a point $p \in \partial (B^2(\frac{T+ T_1}{2})^{n-2}) \subset \C^{n-2}$ and its boundary is disjoint from the ball centered at $p$ of radius $R \sqrt{d}$. Thus by the monotonicity theorem (see for instance \cite{gr}, section $2.3.E'_2$) the curve has area at least $\pi d R^2$.


Note that $$U(T_1) \smallsetminus U(T) = (B^2(T))^{n-2} \smallsetminus (B^2(T_1))^{n-2}$$ can be viewed as a subset of both $\widehat{M}(T)$ and $M$. Choose a $J^T \in  \mathcal{J}_t(T)$ such that  $J^T$ has the form $J_R \oplus J_0$  on  $U(T_1) \smallsetminus U(T)$.  Since any connected finite energy $J^T$-holomorphic (cusp) curve in $\widehat{M}(T)$ with at least one asymptotic end must intersect the complement of $U(T_1)$, it follows from the monotonicity argument above  that if the symplectic area of such a curve is at most $d\pi R^2$ then  its image must be contained in the interior of $(U(T))^c$. Thus, $J^T$ is in $\mathcal{J}_{t,R}(T)$ and we are done.
\end{proof}

Fixing a $T$ such that  $\mathcal{J}_{t,R}(T)$ is nonempty we will henceforth simplify our notation by writing $\widehat{M}$ instead of $\widehat{M}(T)$, $\mathcal{J}_t$ instead of $\mathcal{J}_t(T)$, and $\mathcal{J}_{t,R}$ instead of $\mathcal{J}_{t,R}(T)$.

We now define a moduli space of curves in each  $(\widehat{M} \smallsetminus E_t)^{\infty}_-$.
First note that the standard contact form $\alpha_{E_t}$ on $\partial E_t$ has a nondegenerate closed Reeb orbit $\partial E_t\cap \{z_j=0; j \neq 1\}$ of shortest action and a Morse-Bott family of Reeb orbits in $\partial E_t \cap \{z_1=0\}$. For simplicity, we will denote the closed Reeb orbit on $\partial E_t$ with the smallest  action by $\gamma_1$, i.e., we suppress the dependence on $t$.  The action of $\gamma_1$ is $\pi u(t)^2$ (where $u(t)$ is the function from Lemma \ref{ellfam}) and the
Conley-Zehnder index of its $r$-fold cover, $\gamma_1^{(r)}$, is given by
$$\mu(\gamma_1^{(r)})=2r+ (n-1)\left(2 \Big\lfloor  \frac{r}{S^2}  \Big\rfloor +1\right).$$

%
For a $J_t \in \mathcal{J}_t$,
%
consider a $J_t$--holomorphic curve in $(\widehat{M} \smallsetminus E_t)^{\infty}_-$ which has genus zero,  first Chern number $e$, and  $s^-$
negative ends asymptotic to multiples of $\gamma_1$ such that  the $i^{th}$ such end  covers $\gamma_1$ a total of $a^-_i$ times.  (We refer to the discussion after Lemma \ref{even} for the definition of the Chern number of a finite energy curve. See also Remark \ref{fchern}, which is valid if our curve happens to lie in $M$.)
The virtual dimension of the moduli space represented by this curve is
\begin{equation}
\label{indexx}
2(n-3)(1-s^-) -2s^- + 2e -2\sum_{i=1}^{s^-}\left(a^-_i+(n-1)\left( \Big\lfloor  a^-_i/S^2  \Big\rfloor \right)\right).
\end{equation}

We define $\mathcal{K}_t$ to be the moduli space of somewhere injective $J_t$--holomorphic planes in $(\widehat{M} \smallsetminus E_t)^{\infty}_-$ which have finite energy, first Chern number $3d$, and whose negative end is asymptotic to
$\gamma_1^{(3d-1)}$. Since $S> \sqrt{3d}$, the formula above implies that the virtual dimension of each $\mathcal{K}_t$ is  zero.

\subsection{A compact cobordism}

Now, let $\{\mathcal{J}_{t,R}\}$ be the space of smooth $[1/S,1]$--families  of almost-complex structures $\{J_t\}$ such that $J_t$ belongs to $\mathcal{J}_{t,R}$ for all $t \in [1/S,1]$.
For $\{J_t\} \in \{\mathcal{J}_{t,R}\}$ set $$\mathcal{K} = \{ (C,t) \mid t \in [1/S, 1],\, C \in \mathcal{K}_t\}.$$
Any curve representing a class $C$ that appears in $\mathcal{K}$ must intersect
$(U(T_1))^c$. On this subset of $\widehat{M}$ we are free to perturb the family $\{J_t\}$ arbitrarily and still remain in $\{\mathcal{J}_{t,R}\}$.
Hence for a generic choice of the family $\{J_t\} \in \{\mathcal{J}_{t,R}\}$, the space $\mathcal{K}$ is  an oriented (from \cite{bormon} for instance), $1$-dimensional manifold with boundary equal to $\mathcal{K}_{1/S} \amalg \mathcal{K}_1$. By \cite{BEHWZ},
$\mathcal{K}$ is compact modulo convergence to equivalence classes of  holomorphic buildings in the spaces $(\widehat{M} \smallsetminus E_t)^{\infty}_-$. In this section we prove that in fact all buildings which occur in these limits represent classes in $\mathcal{K}$.

\begin{proposition}
\label{compactt}
For a generic choice of the family $\{J_t\} \in \{\mathcal{J}_{t,R}\}$ the space $\mathcal{K}$ is compact.
\end{proposition}

\subsubsection{Proof of Proposition \ref{compactt}}

\noindent{\bf Step 1.} We begin by specifying our choice of the family $\{J_t\} \in \{\mathcal{J}_{t,R}\}$. This choice is motivated by the desire to avoid certain holomorphic curves with negative virtual indices.
\begin{lemma}
\label{goodJt}
For a generic choice of the family  $\{J_t\} \in \{\mathcal{J}_{t,R}\}$, for each $t \in [1/S,1]$ every simple,  genus zero, finite energy,  $J_t$-holomorphic curve in $(M \smallsetminus E_t)^{\infty}_-$ which has negative ends, all of which are asymptotic to multiples of $\gamma_1$, has a nonnegative virtual index.
\end{lemma}

\begin{proof}

For any $e \in \mathbb{Z}$, $s^- \in \mathbb{N}$  and collection $$\vec{a}= (a^-_1, \dots, a^-_{s^-}) \in \mathbb{N}^{s^-}$$ let $\mathcal{K}_t(e, s^-,\vec{a})$ be the moduli space of simple genus zero $J_t$--holomorphic curves in $(M \smallsetminus E_t)^{\infty}_-$ which have finite energy, first Chern number $e$, and  $s^-$ negative ends, the $i^{th}$ of which is asymptotic to $\gamma_1^{(a^-_i)}$. For a given family $\{J_t\}$ set $$\mathcal{K}(e, s^-,\vec{a}) = \{ (C,t) \mid t \in [1/S, 1],\, C \in \mathcal{K}_t(e, s^-,\vec{a})\}.$$
For a generic choice of $\{J_t\} \in \{\mathcal{J}_{t,R}\}$ the space $\mathcal{K}(e, s^-,\vec{a})$ is a  manifold of dimension
\begin{equation*}
1+2(n-3)(1-s^-) + 2e -2\sum_{i=1}^{s^-}\left(a^-_i+(n-1)\left( \Big\lfloor  a^-_i/S^2  \Big\rfloor \right)\right).
\end{equation*}
This number is either negative, in which case $\mathcal{K}(e, s^-,\vec{a})$ is empty, or, as it is necessarily an odd number, it is strictly positive and hence we can show that for  any $t \in [1/S,1]$  the virtual index of every simple $J_t$-holomorphic curve which represents a class in  $\mathcal{K}_t(e, s^-,\vec{a})$  must be nonnegative. Since the collections $(e,s^-, \vec{a})$ are countable we are done.

\end{proof}

In what follows we will assume that  the family $\{J_t\}$ has been chosen as in Lemma \ref{goodJt}.

\bigskip

\noindent{\bf Step 2.}
Let $\mathbf{ F}$ be a holomorphic building in  $(\widehat{M} \smallsetminus E_t)^{\infty}_-$, for some $t \in [1/S,1]$, that represents a limit point of $\mathcal{K}$. Then $\mathbf{ F}$ consists of a finite collection of holomorphic curves with  images in either $(\widehat{M} \smallsetminus E_t)^{\infty}_-$ or $SE_t$,  the symplectization of $\partial E_t$ equipped with the cylindrical almost-complex structure induced by $J_t$.
We are still working with the notational convention established at the end of Section \ref{split}.  In particular by a curve of $\mathbf{ F}$ we mean a single component of the building $\mathbf{F}$ whose domain is a (possibly) punctured sphere.

To prove Proposition \ref{compactt} it suffices to prove that for our current (generic) choice of the family $\{J_t\}$ the following holds (see Remark \ref{uniquelimit}).
\begin{proposition}
\label{compactF}
The limiting building $\mathbf{F}$ consists of a single curve with image in $(\widehat{M} \smallsetminus E_t)^{\infty}_-$. It is simple, has one negative end, and this end is asymptotic to $\gamma^{(3d-1)}$.
\end{proposition}

In this second step, we establish some initial restrictions on the individual curves of $\mathbf{F}$ following Section \ref{finer}.  We first consider curves of $\mathbf{ F}$ with nonpunctured domains.

\begin{lemma} \label{Fclosed}
Let $F$ be a nonconstant closed curve of $\mathbf{F}$ with image  in $(\widehat{M} \smallsetminus E_t)^{\infty}_-$.  Then $c_1(F)>0$ and $\mathrm{index}(F)>0$.
\end{lemma}

\begin{proof}
By Definition \ref{acs}, the building $\mathbf{ F}$ is the limit of curves with image in
the interior of $(U(T))^c$. Hence, $F$ (and every other curve of $\mathbf{ F}$ with image in $(\widehat{M} \smallsetminus E_t)^{\infty}_-$)  has image contained
in  $(U(T))^c.$ By definition, this latter space is
simply the negative symplectic completion of $$(\CC P^2(R) \times (B^2(T))^{n-2}) \smallsetminus E_t.$$ Since $F$ is closed we have, by monotonicity of $\CC P^2(R)$,  $$c_1(F) = \frac{3}{\pi R^2} \omega_R(F)>0.$$ Then $\mathrm{index}(F) = 2(n-3) + 2c_1(F)>0$, since $n \geq3$.

\end{proof}

Now we consider curves with negative ends.

\begin{lemma} \label{F}
Let $F$ be a finite energy curve in $(\widehat{M} \smallsetminus E_t)^{\infty}_-$ of genus zero such that $F$ has at least one negative end and $c_1(F)=e \le 3d$. Then the ends of $F$ are all asymptotic to multiples of $\gamma_1$, $e>0$, and the total multiplicity of all negative ends is at most $e-1$.
\end{lemma}

\begin{proof}
Suppose that $F$ has $s^-_1 \ge 0$ negative ends asymptotic to multiples of $\gamma_1$,  and $s^-_2 \ge 0$ negative ends asymptotic to multiples of periodic orbits in the Morse-Bott family, where now $s^-_1 + s^-_2 \ge 1$.  Say that the $i^{th}$ negative end covering  $\gamma_1$ does so $a^-_i$ times,  and the $i^{th}$ negative end covering an orbit in the Morse-Bott family does so $b^-_i$ times. Then, as computed by Bourgeois
in \cite{bo}, the  virtual deformation index of $F$ (in the moduli space of finite energy curves with the same asymptotics, modulo reparameterization) is

\begin{eqnarray*}
\mathrm{index}(F) & = & (n-3)(2-s^-_1-s^-_2) +2e-\sum_{i=1}^{s^-_1}\left( 2a^-_i +(n-1)(2 \lfloor a^-_i / S^2 \rfloor +1)\right)\\
{} & {} & - \sum_{i=1}^{s^-_2} \left( 2b^-_i +2 \lfloor b^-_i  S^2 \rfloor +1\right) + \frac{1}{2} s^-_2 (2(n-2))\\
{} & = & 2(n-3)(1-s^-_1) - 2s^-_1 +2e-2\sum_{i=1}^{s^-_1}\left( a^-_i +(n-1) \lfloor a^-_i / S^2 \rfloor \right)\\
{} & {} &- 2\sum_{i=1}^{s^-_2} \left( b^-_i + \lfloor b^-_i  S^2 \rfloor \right)
\end{eqnarray*}
Since $S^2 > 3d$, it follows from this equation that if $\mathrm{index}(F) \geq 0$, then
$e>0$, $s^-_2 =0$ (and hence $s^-_1 \ge 1$), $\sum_{i=1}^{s^-_1} a^-_i \le e-1$ and
\begin{equation}
\label{positiveindex}
\mathrm{index}(F) = 2(n-3)(1-s^-_1) - 2s^-_1 +2e-2\sum_{i=1}^{s^-_1} a^-_i.
\end{equation}
If $F$ is simple, then our choice of the family $\{J_t\}$ implies that $\mathrm{index}(F)\geq 0$ and we are done.

Assume then that $F$ is not simple. By Proposition \ref{multiplecover}, the curve $F$ is  the $p$-fold cover of a simple curve $\widetilde{F}$ for some $p > 1$. The discussion above now implies that  $\widetilde{F}$ has no ends asymptotic to orbits in the Morse-Bott family, $c_1(\widetilde{F}) >0$, and if $\widetilde{F}$ has $\tilde{s}^-_1$ negative ends asymptotic to $\gamma_1$ with the $i^{th}$ such end covering it $\tilde{a}^-_i$ times, then
\begin{equation*}
\label{ }
\sum_{i=1}^{\tilde{s}^-_1} \tilde{a}^-_i \le e/p-1.
\end{equation*}
For $F$ itself, this implies $s^-_2=0$,  $e= p c_1(\widetilde{F})>0$, and
\begin{equation}
\label{cover}
\sum_{i=1}^{{s}^-_1}{a}^-_i = p \sum_{i=1}^{\tilde{s}^-_1} \tilde{a}^-_i \le e-p < e-1,
\end{equation}
as desired.
\end{proof}

\begin{lemma} \label{lowindex}
Let $F$ be a finite energy curve  in $(\widehat{M} \smallsetminus E_t)^{\infty}_-$ of genus zero such that $F$ has at least one negative end and $c_1(F)=e \le 3d$. If the number of negative ends of $F$ is $s$, then
\begin{equation}
\label{lowerindex}
\mathrm{index}(F) \ge 2(n-2) -2(n-2)s.
\end{equation}
If $s=1$, then $\mathrm{index}(F) \geq 0$ with equality if and only if $F$ is simple.
\end{lemma}

\begin{proof}
Lemma \ref{F}  implies that the ends of $F$ are all asymptotic to multiples of $\gamma_1$. It also implies that  the index formula  \eqref{positiveindex} holds (with $s^-_1$ replaced here by $s$) regardless as to whether or not the virtual index of $F$ is positive. By Proposition \ref{multiplecover} we may assume that $F$  is the $p$-fold cover of a simple curve $\widetilde{F}$ with $\tilde{s}$ negative ends. Using  formula \eqref{positiveindex} for $F$ and $\widetilde{F}$, together with
inequality \eqref{cover}, we then get
\begin{eqnarray*}
\mathrm{index}(F)  & \geq  & 2(n-3)[(p\tilde{s} -s) +(1-p)] +2(p\tilde{s} -s)  + p (\mathrm{index} (\widetilde{F}))\\
{}  & \geq  & p(2(n-2)(\tilde{s}-1) +2 \tilde{s}) - 2(n-2)(s -1) -2 s\\
{}  & \geq  & 2(n-2) - 2(n-2)s
\end{eqnarray*}
as $\mathrm{index}(\widetilde{F})\ge 0$ by our choice of $\{J_t\}$, $p\geq 1$, $s \ge 1$, and $ s \le p\tilde{s}$.

When $s=1$ the second to last inequality above yields $\mathrm{index}(F) \geq 2(p-1).$
This implies the last statement of Lemma \ref{lowindex}.
\end{proof}

\begin{lemma}
\label{G}
Let $G$ be a curve of $\mathbf{ F}$ with image  in the symplectization $SE_t$. The positive and negative ends of $G$ are all  asymptotic to some multiple of $\gamma_1$
and the positive ends cover $\gamma_1$ at least as many times as the negative ends.
\end{lemma}

\begin{proof}
This follows exactly as in the proof of Lemma \ref{g}. That is,  one needs only to
apply Stokes' Theorem twice and to invoke the fact, which follows from Lemma \ref{F}, that
the curves of   $\mathbf{ F}$ with image in $(M \smallsetminus E_t)^{\infty}_-$ have at most $3d-1$
total negative ends when counted with multiplicity, and all these ends are asymptotic to multiples of $\gamma_1$.

\end{proof}

The following lemma summarizes some of the key global features of $\mathbf{ F}$ which we can now infer.
\begin{lemma}\label{properties}
Every end of every nontrivial curve of $\mathbf{ F}$ is asymptotic to a multiple of $\gamma_1$. The Chern number of each curve of $\mathbf{ F}$ with image in  $(\widehat{M} \smallsetminus E_t)^{\infty}_-$ lies in $(0,3d]$, and the sum of all these Chern numbers is $3d$.
 \end{lemma}

 \begin{proof}
 By \cite{BEHWZ}, the sum of the Chern numbers of all curves of $\mathbf{ F}$ with image in  $(\widehat{M} \smallsetminus E_t)^{\infty}_-$ is equal to $3d$. Lemma \ref{Fclosed} and Lemma \ref{F} imply that the Chern number of  each such nontrivial curve is positive. The statement about the  Chern numbers then follows. With the upper bound on the Chern numbers in hand  we can invoke  Lemma \ref{F} and Lemma \ref{G} to obtain the first statement.
 \end{proof}

\bigskip

\noindent{\bf Step 3.}
We are now in a position to prove Proposition \ref{compactF} and hence Proposition \ref{compactt}. By \cite{BEHWZ},  the compactifications of the curves of $\mathbf{F}$ fit together to form a continuous map $\overline{\mathbf{ F}}$ from  the unit disc to the closure of $\widehat{M} \smallsetminus E_t$ which takes the boundary of the disc to
 $\gamma_1^{(3d-1)}$. Moreover, there is a single curve $B$ of $\mathbf{ F}$ at the lowest level (in our notation we call this level $1$)  which has a negative end. In particular, $B$ has exactly one negative end and this covers
$\gamma_1$ precisely $3d-1$ times.

If $B$ maps to $(\widehat{M} \smallsetminus E_t)^{\infty}_-$ then we are done. For, in this case, Lemma \ref{F} implies that $c_1(B) \ge 3d$.  It then follows  from  Lemma  \ref{properties}  that $c_1(B) = 3d$ and $B$ is the only curve of $\mathbf{F}$. Since $B$ has index zero it follows from Lemma \ref{lowindex} that it is also simple.

It remains for us to deal with the case in which  the special curve $B$ of $\mathbf{ F}$ maps to $SE_t$.
We begin by replacing $\mathbf{ F}$ by a holomorphic building $\mathbf{\widetilde{ F}}$ consisting of a subset of the curves from $\mathbf{ F}$. This is defined to be the smallest subset of curves of $\mathbf{ F}$  which contains $B$ and satisfies the following condition: if $F$ belongs to $\mathbf{\widetilde{F}}$ and $F'$ shares a matching asymptotic end with  $F$ then $F'$ belongs to $\mathbf{\widetilde{F}}$. By construction, $\mathbf{\widetilde{F}}$ has genus $0$,  the sum of the Chern numbers of the  curves of $\mathbf{\widetilde{F}}$, $c_1(\mathbf{\widetilde{F}})$,  is at most $3d$, and $\mathbf{\widetilde{F}}$ has a distinguished level $1$ curve, $B$, with negative end asymptotic to $\gamma_1^{(3d-1)}$. The definition implies that $\mathbf{\widetilde{F}}$ contains no closed curves (or ghost bubbles). The key observation is that since $\mathbf{F}$ has genus $0$, the curves of $\mathbf{\widetilde{F}}$ in a fixed level  now have no common nodal points. Together with the fact that all the ends of all the  curves of $\mathbf{\widetilde{F}}$ are asymptotic to multiples of $\gamma_1$, this implies that
\begin{equation}
\label{indextilde}
\mathrm{index}(\mathbf{\widetilde{F}}) = 2(c_1(\mathbf{\widetilde{F}}) -3d)
\end{equation}
where $\mathrm{index}(\mathbf{\widetilde{F}})$
%
is the sum of the virtual indices of the curves of $\mathbf{\widetilde{F}}$ and the right hand side is
the index of an abstract finite energy plane with Chern number $c_1(\mathbf{\widetilde{F}})$ and a single negative end asymptotic to $\gamma_1^{(3d-1)}$. This fact will eventually yield the proof of Proposition \ref{compactF}. To make proper use of it though we must first define a special partition of the  curves of $\mathbf{ \widetilde{F}}$.

To proceed we first partition the curves of $\mathbf{ \widetilde{F}}$ with image in $SE_t$ into the smallest possible number  of disjoint subsets, $\mathbf{G}_1, \dots, \mathbf{G}_x$,
such that the compactifications of the curves in each $\mathbf{G}_k$ fit together to
produce a continuous map $\overline{\mathbf{G}}_k$ to $\partial E_t$ with a connected domain.
Since $\mathbf{\widetilde{F}}$ inherits a compacification from that of  ${\mathbf{ F}}$ this smallest partition is unique. We will denote the curves in each $\mathbf{G}_k$ by
$$
\mathbf{G}_k = \{ G^k_1, \dots, G^k_{n_k}\}.
$$

In what follows it will be useful to view each $\mathbf{G}_k$ as a curve itself by identifying the matching ends of its constituent curves, the $G^k_j$. From this perspective, we can consider the remaining {\it nonmatched} ends of the constituent curves as the  negative and positive ends of $\mathbf{G}_k$.

Let $\mathbf{G}_1$ be the subset containing the special curve $B$. It is the only one of the $\mathbf{ G}_k$ with a negative end. Suppose  $\mathbf{G}_1$
has $t_1$ positive ends with the $i^{th}$ one covering $\gamma_1$ exactly $a^1_i$ times. Setting
$$\mathrm{index}(\mathbf{G}_1) \eqdef  \sum_{j=1}^{n_1} \mathrm{index}(G^1_j)$$
we then have
\begin{eqnarray*}
\mathrm{index}(\mathbf{G}_1) & = & (n-3)(2-t_1-1) + \sum_{i=1}^{t_1}(2a^1_i +(n-1)) - (2(3d-1) +(n-1)) \\
{} & = & \sum_{i=1}^{t_1}2a^1_i -6d +2t_1.
\end{eqnarray*}
In particular, the contributions of the matching ends of the $G^1_j$ cancel.

\begin{lemma}
\label{G1}
The index of $\mathbf{G}_1$ is nonnegative and is zero if and only if $\mathbf{G}_1$ is a (stacked collection of)
cylinder(s)  over $\gamma_1^{(3d-1)}$.
\end{lemma}

\begin{proof}
Integrating $d \alpha_{E_t}$ over $\mathbf{ G}_1$ it follows from Stokes' Theorem that $$\sum_{i=1}^{t_1}a^1_i \ge 3d-1.$$ Together with the index formula above this implies
$\mathrm{index}(\mathbf{G}_1)\ge 0$. If $\mathrm{index}(\mathbf{G}_1) = 0$, then we must have $t_1=1$ and
$a^1_1 =3d-1$. It then follows from the definition of $\mathbf{G}_1$ and Lemma \ref{G} that that curves of $\mathbf{G}_1$ are all cylinders  over $\gamma_1^{(3d-1)}$ stacked end-to-end with $B$ at the lowest level.
\end{proof}

For the other subsets $\mathbf{G}_k$ with $k>1$ we have
\begin{eqnarray*}
\mathrm{index}(\mathbf{G}_k) & \eqdef & \sum_{j=1}^{n_k} \mathrm{index}(G^k_j) \\
{} & = & (n-3)(2-t_k) + \sum_{i=1}^{t_1}(2a^k_i +(n-1))  \\
{} & = & 2(n-1) +2t_k +2\sum_{i=1}^{t_1}a^k_i -4
\end{eqnarray*}
where $t_k$ is the number of positive ends of $\mathbf{ G}_k$ and the $i^{th}$ such end covers $\gamma_1$  $a^k_i$ times. In particular, we have the inequality
\begin{equation}
\label{Gk}
\mathrm{index}(\mathbf{G}_k)  \geq  2(n-3) + 4t_k.
\end{equation}

Now, let $F_1, \dots F_y$ be the curves of $\mathbf{\widetilde{F}}$ with image in  $(\widehat{M} \smallsetminus E_t)^{\infty}_-$.
The $t_1$ positive ends of $\mathbf{ G}_1$ determine a unique partition of
$$
\{\mathbf{ G}_2, \dots \mathbf{ G}_x, F_1, \dots, F_y\}
$$
into $t_1$ nonempty subsets $\mathbf{ H}_1, \dots, \mathbf{ H}_{t_1}$ as follows. Labeling  the positive ends of
$\mathbf{ G}_1$, one  defines $\mathbf{ H}_i$ to be the unique subset of $\{\mathbf{ G}_2, \dots \mathbf{ G}_x, F_1, \dots, F_y\}$ such that the compactifications of the constituent curves fit together to produce a map, $\mathbf{ \overline{H}}_i$, from  the unit disc to the closure of ${\widehat{M} \smallsetminus E_t}$ which can be matched  continuously with ${\mathbf{ \overline{G}}}_1$ along the boundary component of its domain that corresponds to the $i^{th}$ positive end of $\mathbf{ G}_1$.

As we did for the $\mathbf{ G}_k$, we define  $\mathrm{index}(\mathbf{ H}_i)$ to be the sum of the indices of the constituent curves of $\mathbf{ H}_i$.
\begin{lemma}
\label{H}
For each $i = 1, \dots t_1$ the index of $\mathbf{ H}_i$ is nonnegative and is zero
if and only if $\mathbf{ H}_i = \{F_l\}$ for some curve $F_l$ of $\mathbf{\widetilde{F}}$ with image in
$(\widehat{M} \smallsetminus E_t)^{\infty}_-$ and exactly one negative end.
\end{lemma}

\begin{proof}
Relabeling the curves we may assume for simplicity that
$$
\mathbf{ H}_i =\{\mathbf{ G}_2, \dots, \mathbf{ G}_{x_i}, F_1, \dots, F_{y_i} \}
$$
for some $x_i \le x$ and $y_i \le y$. Let $t_k$ be the number of positive ends of $\mathbf{ G}_k$
and let $s_l$ be the number of negative ends of $F_l$. The fact that the domain of $\mathbf{ \overline{H}}_i$ has exactly one boundary component and that this corresponds to a negative end,  implies
\begin{equation*}
\label{ }
\sum_{l=1}^{y_i} s_l-1 = \sum_{k=2}^{x_i} t_k.
\end{equation*}
If we let $\kappa \geq 0$ be this common value, then the fact that the domain of $\mathbf{ \overline{H}}_i$  is the unit disc yields the additional identity
\begin{equation*}
\label{ }
x_i+y_i-\kappa =1.
\end{equation*}
In particular, the graph whose vertices correspond to the curves of $\mathbf{ H}_i$ and whose edges
correspond to the matching ends of these curves (of which there are $\kappa$), has Euler characterstic equal to that of a point.

With these identities in hand, it now follows from Lemma \ref{lowindex} and inequality \eqref{Gk} that
\begin{eqnarray*}
\mathrm{index}(\mathbf{ H}_i) & = & \sum_{k=2}^{x_i}\mathrm{index}(\mathbf{ G}_k) + \sum_{l=1}^{y_i} \mathrm{index}(\mathbf{ F}_l)  \\
{} & \geq & 2x_i(n-3)  +4 \sum _{k=2}^{x_i} t_k + 2y_i(n-2) -2(n-2)\sum_{l=1}^{y_i}s_l  \\
{} & = & 2(n-3)(\kappa+1) +2y_i + 4 \kappa  -2(n-2)(\kappa +1)\\
{} & = & 2\kappa +2(y_i-1)\\
{} & \geq & 0.
\end{eqnarray*}
Moreover, equality holds if and only if $y_i=1$ and $\kappa =0$. In this case  $\mathbf{ H}_i$ consists of one curve  of $\mathbf{F}$ with image  $(\widehat{M} \smallsetminus E_t)^{\infty}_-$ (since $y_i=1$ and $x_i=0$) and this curve has exactly one negative end (since $\kappa= \sum s_l =1$).

\end{proof}

We can now complete the proof of Proposition \ref{compactt}. By definition, the disjoint subsets
$\mathbf{ G}_1, \mathbf{ H}_1, \dots \mathbf{ H}_{t_1}$ contain all the curves of $\mathbf{\widetilde{F}}$. By equation \eqref{indextilde}, we then have
\begin{equation}
\label{eqnfor35}
 \mathrm{index}(\mathbf{ G}_1) +  \mathrm{index}(\mathbf{ H}_1) + \dots +  \mathrm{index} (\mathbf{ H}_{t_1}) =  2(c_1(\mathbf{\widetilde{F}}) -3d) \leq0,
\end{equation}
as $c_1(\mathbf{\widetilde{F}}) \leq 3d$.
By Lemma \ref{G1} and Lemma \ref{H}, the summands on the left are all nonnegative. Hence they are all zero and $c_1(\mathbf{\widetilde{F}}) = 3d$. This last equality implies $\mathbf{\widetilde{F}}=\mathbf{F}$. Indeed, by Lemma \ref{properties} there can be no nontrivial curves of $\mathbf{F}$ with image in  $(M \smallsetminus E_t)^{\infty}_-$ which don't belong to $\mathbf{ \widetilde{F}}$. By Lemma \ref{G} any curves with image in $SE_t$ are equivalent under matching ends to a nontrivial curve with image in $(M \smallsetminus E_t)^{\infty}_-$. Finally, as $\mathbf{F}$ has genus $0$ it cannot contain trivial curves (ghost bubbles) which identify nodes of curves in $\mathbf{\widetilde{F}}$.

Since the indices on the left of equation \eqref{eqnfor35} are all zero, Lemma \ref{G1} and Lemma \ref{H} also imply that $t_1=1$, the subset $\mathbf{ G}_1$ is a stacked collection of trivial cylinders over $\gamma_1^{(3d-1)}$,  the subset $\mathbf{ H}_1$ consists of a single curve $F$ of $\mathbf{\widetilde{F}} = \mathbf{ F}$ with image  $(\widehat{M} \smallsetminus E_t)^{\infty}_-$, the curve $F$ has exactly one negative end,  and this end covers $\gamma_1$ precisely $3d-1$ times.  Since  $\mathrm{index}(F) = 0$, Lemma \ref{lowindex}  implies that $F$ is also simple and, with this, the proof of Proposition \ref{compactF}, and hence Proposition \ref{compactt} is complete.

\subsection{The space $\mathcal{K}_{1/S}$.} \label{threethree}

Here we prove the following result.

\begin{proposition}
\label{nonempty}
For sufficiently large $S>0$ and a generic almost-complex structure $J_{1/S}$ in $\mathcal{J}_{1/S,R}$, the corresponding moduli space $\mathcal{K}_{1/S}$ is an oriented,
compact, zero-dimensional manifold whose oriented cobordism class is nontrivial.
\end{proposition}

Much of this has already been established. We have already shown that the virtual dimension of $\mathcal{K}_{1/S}$ is $0$. For  a generic $J_{1/S}$ in $\mathcal{J}_{1/S,R}$, $\mathcal{K}_{1/S}$ is then a zero-dimensional manifold which can be oriented as in \cite{bormon}, see also \cite{EGH}. The compactness  of $\mathcal{K}_{1/S}$ follows as in Proposition \ref{compactt}.  It just remains to verify the nontriviality of the oriented cobordism class.

By our compactness result, Proposition \ref{compactt}, its oriented cobordism class is independent of  the choice of a generic  almost-complex structure in $\mathcal{J}_{1/S,R}$. So, it suffices to show that this class is nontrivial for a specific
regular almost-complex structure in $\mathcal{J}_{1/S,R}$. By our choice of $\phi_{1/S}$ the manifold $(\widehat{M} \smallsetminus E_{1/S})^{\infty}_-$ inherits the $\mathbb{T}^{n-2}$--action from $\widehat{M} =\CC P^4 (R) \times (\CC P^1(2T))^{n-2}$.  We will restrict our attention to
the subset $\bar{\mathcal{J}}_{1/S,R}$ of  $\mathcal{J}_{1/S,R}$ consisting of  almost-complex structures  which are invariant under this action.  Note that for  $J_{1/S} \in  \bar{\mathcal{J}}_{1/S,R}$ the submanifold $$( (\CC P^2(R) \smallsetminus E(1/S, 1)) \times \{0\})^{\infty}_-  \subset (\widehat{M} \smallsetminus E_{1/S})^{\infty}_-$$ is $J_{1/S}$-holomorphic. We
 say that  $J_{1/S} \in \bar{\mathcal{J}}_{1/S,R}$ is \emph{suitably restricted} if its restriction to $( (\CC P^2(R) \smallsetminus E(1/S, 1)) \times \{0\})^{\infty}_-  $ renders $\mathbb{C}P^1(\infty) \times \{0\}$ holomorphic, and is regular for all somewhere injective, finite energy curves of genus zero in $( (\CC P^2(R) \smallsetminus E(1/S, 1)) \times \{0\})^{\infty}_-  $.

The following two results will immediately imply Proposition \ref{nonempty}.

\begin{proposition}
\label{nontrivial}
If $J_{1/S} \in  \bar{\mathcal{J}}_{1/S,R}$ is regular for $\mathcal{K}_{1/S}$ and is suitably restricted, then the cobordism class of $\mathcal{K}_{1/S}$ is nontrivial.
\end{proposition}

\begin{proposition}
\label{regular}
There exists a $J_{1/S} \in  \bar{\mathcal{J}}_{1/S,R}$ which is regular for $\mathcal{K}_{1/S}$  and is suitably restricted.
\end{proposition}

The condition that $\bar{\mathcal{J}}_{1/S,R}$ is regular for $\mathcal{K}_{1/S}$ of course must also imply that curves appearing in holomorphic buildings in the boundary of $\mathcal{K}_{1/S}$ are also regular (so that Proposition \ref{compactt} ensures compactness).

\subsubsection{Proof of Proposition \ref{nontrivial}.}
We first note  that since $J_{1/S}$ is  suitably restricted the space $\mathcal{K}_{1/S}$ is nonempty. In particular,  Theorem \ref{key} yields  a curve with image in
$$
( (\CC P^2(R) \smallsetminus E(1/S, 1)) \times \{0\})^{\infty}_-  \subset (\widehat{M} \smallsetminus E_{1/S})^{\infty}_-
$$
that represents a class in $\mathcal{K}_{1/S}$, (see Remark \ref{up-reg}). It now suffices to show that all curves representing a
class in $\mathcal{K}_{1/S}$ have the same orientation.

Let $F$ be a curve representing a class in $\mathcal{K}_{1/S}$.  Since  $J_{1/S} \in \bar{\mathcal{J}}_{1/S,R}$ is regular for $\mathcal{K}_{1/S}$, the image of $F$ must be  contained in $( (\CC P^2(R) \smallsetminus E(1/S, 1)) \times \{0\})^{\infty}_- .$ For, if not,   the $\mathbb{T}^{n-2}$-action would produce a family of curves in $\mathcal{K}_{1/S}$, including $F$, of dimension at least $n-2 \ge 1$. As  $\mathrm{index}(F)=0$ this would contradict the regularity of $J_{1/S}$.

As curves $F$ in $\mathcal{K}_{1/S}$ are somewhere injective and of index $0$, they must be immersed, see \cite{wendl}, Corollary $3.17$. Now let $\nu$ be the subbundle of $F^*(T(\widehat{M} \smallsetminus E_{1/S})^{\infty}_- )$ corresponding to the normal bundle of the image of $F$. Denote the almost-complex structure induced by $J_{1/S}$ on the total space of $\nu$ by the same symbol, that is, $J_{1/S}$ is the almost-complex structure such that infinitesimal deformations of $F$ correspond to holomorphic sections of $\nu$.

\begin{lemma}
\label{reg}
The  bundle $\nu$  splits as a sum
of almost-holomorphic line bundles, that is, complex subbundles whose total spaces are invariant under the action of $J_{1/S}$. In particular, $\nu = H \oplus V_3 \oplus \ldots \oplus
V_n,$ where $H$ is the subbundle of normal vectors parallel to $\{z_j=0\mid
j =3, \dots n\}$ and $V_j$ is the subbundle of normal vectors parallel to
the $z_j$ factor.
\end{lemma}

\begin{proof}

It is clear that $H$ is an almost-holomorphic subbundle. It suffices to show that each $V_j$ is too.
Fixing a $p =F(z)$ we see that since $J_{1/S}$ belongs to $\bar{\mathcal{J}}_{1/S}(T_1)$, the space $J_{1/S}(p)(V_j(z))$ is an $S^1$-invariant, $2$-dimensional subspace of $T_p((M \smallsetminus E_{1/S})^{\infty}_-)$ which is transverse to $H(z) \oplus_{i \neq j} V_i(z)$. Here the $S^1$-action is the derivative of rotations in the $z_j$ plane acting on the tangent space at the fixed point $p$, and transversality follows because the action is by isomorphisms. Since the only such $2$-dimensional subspace is $V_j(z)$ itself we see that $V_j(z)$ is a complex subspace of $T_p((\widehat{M} \smallsetminus E_{1/S})^{\infty}_-)$ and hence the vector bundle $V_j$ has complex fibers. To see that the bundle is almost-holomorphic, we choose a connection on $V_j$ whose horizontal subspaces are all
invariant under our $S^1$ action of rotations in the $z_j$ plane. Then let $v$ be the horizontal lift of a vector in the image of $dF_{z}$. We can think of $v$ as a vector field tangent to the total space of $V_j$ defined along the fiber $V_j(p)$. Then $J_{1/S}(v)$ is a vector field tangent to the total space of $\nu$ defined along $V_j(p)$. Using the connection we can project $J_{1/S}(v)$ to the fibers of $\nu$ and then project along $V_j$ to $H \oplus_{i \neq j} V_i$. This gives a linear map $L_v$ from $V_j(p)$ to $H \oplus_{i \neq j} V_i(p)$, and since $J_{1/S}$ is invariant under rotations the map takes points on the same $S^1$-orbit to the same image. By linearity, this forces the map to be identically zero. As the horizontal planes of the connection are tangent to the total space of $V_j$ we conclude that $J_{1/S}(v)$ is also tangent to $V_j$ and it follows that $V_j$ is an almost-holomorphic subbundle as required.

\end{proof}

At this point we can make a key observation of the proof, that a version of automatic regularity can be applied in our present setting. Let $F$ be any  curve which  represents a class in $\mathcal{K}_{1/S}$. By  Lemma \ref{reg}, the linearized Cauchy-Riemann operator along $F$ (giving infinitesimal deformations of the finite energy plane) splits. This allows us to apply the (four-dimensional) automatic regularity results of \cite{wendl}. (The curve $F$ satisfies the hypotheses of Theorem 1 of  \cite{wendl}  since the normal first Chern number term, $c_N$, is negative in our case.) In particular,  no factor of the operator can have a nontrivial cokernel, and thus the Cauchy-Riemann operator itself is surjective.

To complete the proof of Proposition \ref{nontrivial} let us now suppose that there are at least two distinct classes in the zero dimensional moduli space $\mathcal{K}_{1/S}$. To determine the difference in orientation of these classes  we fix representative curves, $F$ and $F'$, identify their normal bundles, and choose a family of linear Cauchy-Riemann operators  interpolating between the induced operators
for $F$ and $F'$. The difference  in  orientations is then given by a sum of crossing numbers evaluated at parameter values for which the associated Cauchy-Riemann operators have nontrivial cokernel (see, for example, \cite{msa} Remark 3.2.5). But our version of  automatic regularity here means that, provided we choose our interpolation to preserve the splitting, there are no such  singular parameter values. Hence,  all classes in  $\mathcal{K}_{1/S}$ have the same orientation.

\subsubsection{Proof of Proposition \ref{regular}.}

For a fixed $J_{1/S} \in \bar{\mathcal{J}}_{1/S,R}$, a finite energy curve in $(\widehat{M} \smallsetminus E_{1/S})^{\infty}_-$ will be called {\it orbitally simple} if it intersects at least one
orbit of the $\mathbb{T}^{n-2}$--action exactly once, and furthermore the tangent space to the curve and the tangent space to the orbit together span a subspace of maximal dimension, namely $n$. There is  a subset of $ \bar{\mathcal{J}}_{1/S,R}$ of second category which consists of suitably restricted almost-complex structures for which all orbitally simple curves  are regular. This
follows from the standard methods, exactly as in, say, Section $3.2$ of \cite{msa}. Here, the condition of orbital simplicity replaces the assumption in \cite{msa} that all curves are simple. In particular, the analogue of Proposition $3.2.1$  of \cite{msa} allows one  to construct  sections of bundles over $(\widehat{M }\smallsetminus E_{1/S})^{\infty}_-$ which are both $\mathbb{T}^{n-2}$--invariant and, when restricted to the image of a given orbitally simple curve, have support contained in the neighborhood of a single point.

Thus, to detect the desired almost-complex structure $J_{1/S}  \in \bar{\mathcal{J}}_{1/S,R}$ it will suffice to find an open subset of $ \bar{\mathcal{J}}_{1/S,R}$ such that every curve for the corresponding spaces $\mathcal{K}_{1/S}$ is either contained in $( (\CC P^2(R) \smallsetminus E(1/S, 1)) \times \{0\})^{\infty}_- $ or is orbitally simple. In the case of a building in the boundary of $\mathcal{K}_{1/S}$ this should apply to every curve in the building mapping to $(\widehat{M }\smallsetminus E_{1/S})^{\infty}_-$ (then Proposition \ref{compactt} ensures compactness of $\mathcal{K}_{1/S}$). In fact, as orbital simplicity is an open  property  (by compactness), it will suffice to construct a single such almost-complex structure. Hence, Proposition \ref{regular} will be implied by the following result.

\begin{proposition}\label{ex}
 There exists a suitably restricted $J_{1/S}  \in \bar{\mathcal{J}}_{1/S,R}$ such that all curves (or buildings) which represent classes in $\mathcal{K}_{1/S}$ and which do not lie entirely in $( (\CC P^2(R) \smallsetminus E(1/S, 1)) \times \{0\})^{\infty}_-$ must intersect some orbit of the $\mathbb{T}^{n-2}$--action exactly once, and at the given intersection point the tangent space to the curve and the tangent space to the orbit span an $n$-dimensional subspace.
\end{proposition}

\noindent{\bf Proof of Proposition \ref{ex}.} For a real number $a$ slightly larger than $1$, let $\Upsilon_a$ denote the hypersurface $\{S^2|z_1|^2 + |z_2|^2 =a^2\} \subset (\widehat{M} \smallsetminus E_{1/S})^{\infty}_-$ which divides $(\widehat{M} \smallsetminus E_{1/S})^{\infty}_-$ into two regions; $V$ which contains the cylindrical concave end corresponding to $\partial E_{1/S}$, and $W = \{S^2|z_1|^2 + |z_2|^2 >a^2\}$. The hypersurface $\Upsilon_a$ is not of contact type, although its intersections with the levels $\{z_j=c_j | j=3,\dots ,n\}$ for $(c_3,\dots ,c_n) \in (\CC P^1(2T))^{n-2}$, are. The hypersurface $\Upsilon_a$ is {\it stable} in the sense that the symplectic form $\omega_R$ when restricted to $\Upsilon_a$ gives a {\it stable Hamiltonian structure}, for this see for example the discussion at the start of section $2$ of the paper \cite{ciemon}. Indeed, let $\lambda$ be the pull-back to $\Upsilon_a$ of the standard Liouville form under the projection from $\Upsilon_a$ into $B^4(R) \subset \CC P^2(R)$. We observe that $L=\text{ker}(\omega_R |_{\Upsilon_a})$ is $1$-dimensional (it is known as the Hamiltonian line field) and contained in $\text{ker} d \lambda$. Moreover, $\lambda |_L \neq 0$. Given this, the section $v$ of the Hamiltonian line field determined by $\lambda(v) \equiv 1$ is called the Reeb vector field corresponding to the stable Hamiltonian structure. In our case, $v$ preserves the levels $\{z_j=c_j | j=3,\dots ,n\}$ and restricted to each level is the Reeb vector field corresponding to the contact form there.
Observe then that the  flow of the Reeb vector field $v$ on $\Upsilon_a$ contains precisely two $2(n-2)$ parameter families of closed Reeb orbits, one family  of period $\pi a^2/S^2$ corresponding to translations of $a\gamma_1$ and another  {\it long} family  with period $\pi a^2$.

Given the stable Hamiltonian structure on $\Upsilon_a$ described above, we can define compatible almost-complex structures as in Section $2.2$ of \cite{BEHWZ} or section $2.5$ of \cite{ciemon}, see also Section \ref{split} above, and can invoke the compactness theorem of \cite{BEHWZ} (or \cite{ciemon}) as we split  $(\widehat{M} \smallsetminus E_{1/S})^{\infty}_-$ along $\Upsilon_a$. (The compactness theorem is valid provided we split along a stable hypersurface.) Indeed, we will need to consider this splitting to find the almost-complex structure in Proposition \ref{ex}. But first we must start with a rather special almost-complex structure on $(\widehat{M} \smallsetminus E_{1/S})^{\infty}_-$.

\begin{lemma}
There is a  $J$ in  $\bar{\mathcal{J}}_{1/S,R}$ which is suitably restricted, compatible with $\Upsilon_a$ and such that $\CC P^1(\infty) \times (\CC P^1(2T))^{n-2}$ is holomorphic, and the projection of $W$ onto $\{z_j=0 \mid j=3, \dots, n\}$ is holomorphic  with respect to the almost-complex structure induced by $J$. Furthermore the same holds for the almost-complex structures $J^N$ stretched to length $N \in \Bbb N$ along $\Upsilon_a$.
\end{lemma}

Here, $\CC P^1(\infty) \subset \CC P^2(R)$ denotes the line at infinity, and we use the fact that $W$ can be viewed as a subset of $\widehat{M}$  to which $J$ can be restricted and on which the projection to $\{z_j=0 \mid j=3, \dots, n\}$ is defined.

\begin{proof} Before describing how such a $J$ can be constructed we first enumerate its required properties beginning with the three which are implied by the restriction that it belong to  $\bar{\mathcal{J}}_{1/S,R}$.
\begin{enumerate}
  \item[(i)] $(\widehat{M} \smallsetminus E_{1/S})^{\infty}_-$ equipped with $J$ is an almost-complex manifold with a cylindrical end.
  \item[(ii)] $J$ is in ${\mathcal{J}}_{1/S,R}$,  that is, every connected finite energy $J$-holomorphic (cusp) curve in $(\widehat{M}(T) \smallsetminus E_{1/S})^{\infty}_-$  with at least one asymptotic end and area bounded by $d\pi R^2$ has image contained in the interior of $(U(T))^c$.
  \item[(iii)] $J$ is preserved by the  action $\mathbb{T}^{n-2}$ of  on $(\widehat{M} \smallsetminus E_{1/S})^{\infty}_-$.
  \item[(iv)] $J$ is compatible with $\Upsilon_a$.
  \item[(v)] $\CC P^1(\infty) \times (\CC P^1(2T))^{n-2}$ is $J$-holomorphic.
  \item[(vi)] $J$ is suitably restricted.
  \item[(vii)] The projection of $W$ onto $\{z_j=0 \mid j=3, \dots, n\}$ is $J$-holomorphic.
\end{enumerate}

Let $J_R$ be an almost-complex structure on $\C P^2(R)$ such that
\begin{enumerate}
  \item[(R1)] $J_R$ equals the standard complex structure on $E(1/S,1)$.
  \item[(R2)] $J_R$ is compatible with the hypersurface $\{S^2|z_1|^2 + |z_2|^2 =a^2\}$.
  \item[(R3)] The line at infinity, $\C P^1(\infty)$ , is $J_R$-holomorphic.
  \item[(R4)] The restriction of $J_R$ to $\C P^2(R) \smallsetminus E(1/S,1)$ induces an almost complex structure on $(\C P^2(R) \smallsetminus E(1/S,1))^{\infty}_-$ which is regular for somewhere injective, finite energy curves of genus zero.
\end{enumerate}
Then if $J_T$ is the standard complex structure on $(\C P^1(2T))^{n-2}$ the almost-complex structure $J=J_R \oplus J_T$ on $\widehat{M} = \C P^2(R) \times (\C P^1(2T))^{n-2}$ has a restriction to $\widehat{M} \smallsetminus E_{1/S}$ which is compatible with the boundary and so defines an almost-complex structure on $(\widehat{M} \smallsetminus E_{1/S})^{\infty}_-$ satisfying $(\mathrm{i})$. It follows from the argument of Lemma \ref{diamond}, that property $(\mathrm{ii})$ is also satisfied, as is $(\mathrm{iii})$ since $\partial E_{1/S}$ is itself $\mathbb{T}^{n-2}$ invariant. Properties $(\mathrm{iv})$  and $(\mathrm{v})$ are  implied by $(\mathrm{R}2)$ and $(\mathrm{R}3)$, respectively. Property $(\mathrm{vi})$ follows from condition $(\mathrm{R}4)$, and property $(\mathrm{vii})$ follows simply from the fact that $W$ is a product manifold and $J$ is split.

Finally we observe that if $J^N_R$ denotes $J_R$ stretched to length $N$ along $\{S^2|z_1|^2 + |z_2|^2 =a^2\}$ then $J^N = J^N_R \oplus J_T$ is the result of stretching $J$ to length $N$ along $\Upsilon_a$. Hence if we assume that $J^N_R$ also induces regular almost-complex structures on $(\C P^2(R) \smallsetminus E(1/S,1))^{\infty}_-$ we have constructed a $J$ as required.
\end{proof}

Fix a $J$ as in the previous lemma and as in the lemma denote by $J^N$ the corresponding sequence of almost-complex structures starting with $J$ which are stretched to a length $N$ along $\Upsilon_a$.
Letting $N \to \infty$, for the limiting almost-complex structure the submanifold  $X^{\infty}_- = W^{\infty}_- \cap \{z_j=0 \mid j=3, \dots, n\}$ will remain complex and is a manifold with cylindrical end (it can be written $(\CC P^2(R) \smallsetminus E(a/S, a))^{\infty}_-$) in its own right; the projection $\pi: W^{\infty}_- \to X^{\infty}_-$ will be holomorphic.

We now show that one of these $J^N$ will satisfy the requirements of Proposition \ref{ex}.

Arguing by contradiction, let us assume that for all such almost-complex structures $J^N$ there exist curves $F_N$ which represent classes in $\mathcal{K}_{1/S}$ (or possibly holomorphic buildings in the boundary) and  intersect the $\mathbb{T}^{n-2}$--orbits in multiple points, or not at all (or in the building case this is true for at least one component). Passing to a subsequence if necessary, we may assume by  \cite{BEHWZ}  that  these curves $F_N$ converge to a holomorphic building $\mathbf{F}$.

Let $G_k$ for $1 \le k \le K$ denote the components of $\mathbf{F}$ in $W^{\infty}_-$. As $\pi$ is holomorphic the projections $H_k = \pi \circ G_k$ are also finite energy curves in $X^{\infty}_-$. The idea of the proof is that although we cannot assume any regularity properties for the almost-complex structure on $W^{\infty}_-$ (as it is carefully chosen to be invariant under $\mathbb{T}^{n-2}$ and have a holomorphic projection) we are assuming regularity for $X^{\infty}_-$ and so can conclude index inequalities for the $H_i$.

The key result will be the following.

\begin{lemma} \label{ending}
All ends of all $G_k$ are asymptotic to translations of covers of $a \gamma_1$. Counting with multiplicity, these ends cover $a \gamma_1$ a total of at most $3d-2$ times.
\end{lemma}

\begin{proof}

Let $e_k$ be the Chern class of $H_k$ (which is the same as that of $G_k$). As the $F_N$ have Chern class $3d$ we deduce that $\sum_{k=1}^K e_k = 3d$. Therefore by Lemma \ref{new2} all ends of the $H_k$ are asymptotic to multiples of $a \gamma_1$, and hence the same is also true for the $G_k$.

Suppose that $G_k$ (and hence also $H_k$) has $s_k^-$ negative ends with the $i^{th}$ such end covering $a\gamma_1$ a total of  $c_{k,i}$ times. Using formula \eqref{indexf}, the deformation index of $H_k$ is
\begin{eqnarray*}
\mathrm{index}(H_k) & = & -2 + 2e_k - 2\sum_{i=1}^{s_k^-}c_{k,i}.
\end{eqnarray*}
By Lemma \ref{new2} this is nonnegative (and in fact is strictly positive if $H_k$ is a multiple cover), and so for $k=1, \dots K$, we have
\begin{equation}
\label{gk}
\sum_{i=1}^{s_k^-}c_{k,i} \le e_k-1.
\end{equation}

By summing inequality \eqref{gk} over $k$ we see immediately that if $K \ge 2$ then since $\sum_{k=1}^K e_k = 3d$ the total number of ends is at most $3d-2$ as required.

Suppose then that $K=1$. Then $H_1$ is necessarily a multiply covered curve. Indeed, the number of preimages of a point $z$ in the image of $H_1$ is at least the number of times $G_1$ intersects the fiber of $\pi$ over $z$. This number can be $1$ only if $G_1$ intersects the fiber in a single point and the fiber and the tangent space to $G_1$ span a subspace of maximal dimension. But if this were the case then the same would be true for the intersection of the $F_N$ with some nearby fibers when $N$ is very large, contradicting our hypothesis on the $F_N$.
\end{proof}

There is a unique curve of $\mathbf{F}$ at the lowest level (and hence with image not in $W^{\infty}_-$) which has a negative end, and this negative end is asymptotic to $\gamma_1^{(3d-1)}$.
Moreover,  $\mathbf{F}$ must include a curve in the symplectic completion of $V$ and this curve must have strictly positive area. Hence, it follows from Stokes' Theorem, applied here to the curves of $\mathbf{F}$ not in  $W^{\infty}_-$,  that for an $a$ sufficiently close to $1$ the negative ends of the curves of $\mathbf{F}$ in $W^{\infty}_-$ cover the corresponding translations of $a\gamma_1$ a total of at least $3d-1$ times. This is the desired contradiction to Lemma \ref{ending}.

\subsection{The completion of the proof}

 Theorem \ref{thm2} now follows almost immediately from Propositions \ref{compactt} and \ref{nonempty}. For any $S>0$ with $S^2 \in \mathbb{R} \smallsetminus \mathbb{Q}$,  and any positive integer $d$ satisfying $3d < S^2$, it follows from Proposition \ref{compactt} that for a generic family $J_t$  the space $\mathcal{K} = \{ \mathcal{K}_t \mid t \in [1/S, 1]\}$ is  a compact,  oriented, $1$-dimensional manifold whose boundary is $\mathcal{K}_{1/S} \amalg \mathcal{K}_1$. Proposition \ref{nonempty}  implies  that the moduli space $\mathcal{K}_{1/S}$ represents a nontrivial cobordism class, and so $\mathcal{K}_1$ is also nonempty. Hence, there exists  a holomorphic plane in $(\widehat{M} \smallsetminus E_1)^{\infty}_-$
whose negative end is asymptotic to $\gamma_1^{(3d-1)}$. The symplectic area of this curve is positive and equal to  $d\pi R^2 - (3d-1)\pi$. Thus, $R^2 >
\frac{3d-1}{d}$, and taking the limit as $d$ (and hence $S$) goes to  $\infty$ we have
$R^2 \ge 3$.

\section{Symplectic embeddings}

In this section we prove Theorems \ref{embed2} and \ref{embed}.

\subsection{The proof of Theorem \ref{embed}}

Let $\Sigma(\epsilon)$ be a punctured torus, i.e., a surface of genus one with one boundary component,
equipped with a symplectic form of total area $\epsilon$.
The Main Lemma of \cite{guth} implies the following.

\begin{proposition} For any $S, \epsilon>0$, there exists a symplectic embedding of
$B^{2(n-1)}(S)$ into $\Sigma(\epsilon) \times \RR^{2(n-2)}$.
\end{proposition}

To establish Theorem \ref{embed} it suffices to find an embedding
of $\Sigma(\epsilon) \times B^2(1)$ into $B^2(\sqrt{R}) \times B^2(\sqrt{R})$, where $R>2$ is fixed and $\epsilon$ can
be arbitrarily small. The existence of such an embedding is essentially contained in
Lemma $3.1$ of \cite{guth}. We review the construction here for the sake of
completeness and because it will play an important role in the proof of Theorem \ref{embed2}.

Fix $R = 2 + 2\delta$ for $\delta>0$. We begin with the following elementary result.
\begin{lemma}
\label{push}
For every $\delta >0$ there is  a nonnegative function $H$ whose support is contained in $B^2(\sqrt{2+2\delta})$, whose maximum value is less than $\pi + \delta$, and whose time-$t$ Hamiltonian flow,
$\phi^t_H$, satisfies $$\phi^t_H(B^2(1)) \subset B^2(\sqrt{(1+t)(1 + \delta/\pi)})$$ and $$\phi^1_H(B^2(1))\cap B^2(1) = \emptyset.$$
\end{lemma}

\begin{proof}


Let $U$ be any set whose closure is contained in the interior of the square $[0, \sqrt{\pi + \delta}] \times [0, \sqrt{\pi + \delta}] \subset \R^2.$ The time-$t$ Hamiltonian flow of the function
$K(x, y) = (\sqrt{\pi + \delta})x$ on $\R^2$ is given by $$\phi^t_K(x,y) = (x, y+ t \sqrt{\pi + \delta}).$$ (We use the convention that the Hamiltonian vectorfield,  $X_K$,  of $K$ is defined by the equation $i_{X_K} \omega_0 = -dK$.)
Hence, for all $t>0$ the set $\phi^t_K(U)$ is contained in the interior of the rectangle
$[0, \sqrt{\pi + \delta}] \times [ 0, (1+t)\sqrt{\pi + \delta}] $ and $\phi^1_K(U) \cap U =\emptyset$.
 Cutting $K$ off appropriately near $\bigcup_{t \in [0,1]} \phi^t_K(U)$, we get a nonnegative function $\hat{K}$ whose Hamiltonian flow still has these properties, but is now supported in $[0, \sqrt{\pi + \delta}] \times [ 0, 2\sqrt{\pi + \delta}] $ and satisfies
$\max(\hat{K}) < \pi +\delta$.

One can construct a symplectic diffeomorphism $\psi$ of $\R^2$ which maps $[0, \sqrt{\pi + \delta}] \times [ 0, 2\sqrt{\pi + \delta}] $ into $B^2(\sqrt{2+2\delta})$ and for $t \in [0,1]$ maps arbitrarily large subsets of each
rectangle $[0, \sqrt{\pi + \delta}] \times [ 0, (1+t)\sqrt{\pi + \delta}]$ onto balls centered at the origin. (Such maps are described and illustrated explicitly in Section $3.1$ of \cite{schl}.) We choose these arbitrarily
large subsets of the rectangles $[0, \sqrt{\pi + \delta}]\times [ 0, (1+t)\sqrt{\pi + \delta}]$ so that they contain $\phi^t_K(U)$
for all $t \in [0,1]$. Setting $U = \psi^{-1}(B^2(1))$ and $H = \hat{K} \circ \psi$, we are done.
\end{proof}

\begin{remark}
\label{error1}
It is clear from the definition of $H$ in terms of $\hat{K}$ that for all $t\in [0,1]$ the distance between $\phi^t_H(B^2(1))$ and the boundary of $B^2(\sqrt{(1+t)(1 + \delta/\pi)})$ is greater than zero and of order $\delta$.
\end{remark}

Consider an immersion $i_{\delta}$ of $\Sigma(\epsilon)$ into $\R^2$, as sketched in
Figure $1$, with the following properties:
\begin{itemize}
\item the double points are concentrated in an
arbitrarily small region around the origin $(0,0)$.
\item  the vertical and horizontal sections
crossing at $(0,0)$ are arbitrarily thin (and hence arbitrarily long if need be).
\item the areas of the regions $A$ and $B$ are both equal to $\pi+2\delta$.
\end{itemize}
By the last of these properties we may assume that, for sufficiently small $\epsilon>0$,  the immersion lies
 in a region symplectomorphic to
$B^2(\sqrt{2 + 2\delta})$.

\begin{figure}[h]
\label{torus1}
\begin{center}
\psfrag{B}[][][0.8]{$B$}
\psfrag{A}[][][0.8]{$A$}
\psfrag{C}[][][0.8]{$C^{\Lambda}$}
\psfrag{(0,0)}[][][0.8]{$(0,0)$}
\includegraphics{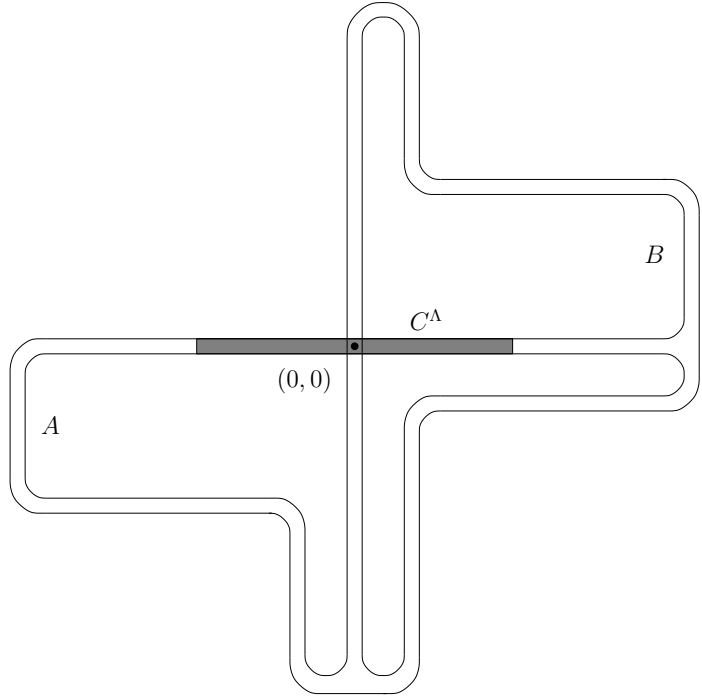} \caption{The symplectic immersion $i_{\delta}$ of the
punctured torus.}
\end{center}
\end{figure}

Let $I^0_{\delta}$ be the symplectic immersion of $\Sigma(\epsilon) \times B^2(1)$ into $B^2(\sqrt{2+2\delta})\times B^2(\sqrt{2+2\delta})$ which acts as $i_{\delta}$ on the first factor and as inclusion on the second.
We now alter the image of $I^0_{\delta} $ to obtain the desired embedding. In what follows, we will use coordinates $(x_1,y_1)$ on the plane containing the first copy of  $B^2(\sqrt{2+2\delta})$
 and coordinates $(x_2,y_2)$ on the plane  containing the second copy. The projection from $\R^4$ to the $x_1y_1$-plane will be denoted by $pr_1.$

The self-intersections of $I^0_{\delta}$ project under $pr_1$ to the self intersections of $i_{\delta}$.  The $x_1$-coordinates of these points take values in an interval of the form $[-\lambda, \lambda]$.
For $\Lambda> \lambda$, let  $C^{\Lambda}$ be the horizontal portion of the image of $i_{\delta}$ which  passes through the
 origin and whose first coordinates satisfy $x_1 \in [-\Lambda,\Lambda]$.
To remove the intersections of $I^0_{\delta}$,
we consider the Hamiltonian $\hat{H}= \chi(x_1)H(x_2,y_2)$ where  $H$ is the Hamiltonian from Lemma \ref{push} and $\chi$
is a bump-function which equals $1$ for $|x_1| \leq \lambda$ and equals $0$ for $|x_1| \geq \Lambda$.  The  time-$1$ Hamiltonian flow of $\hat{H}$ is
 $$\phi_{\hat{H}}(x_1,y_1,x_2,y_2)=(x_1, y_1 + \chi'(x_1) H(x_2,y_2), \phi^{\chi(x_1)}_{H}(x_2, y_2))
.$$ Let $I^1_{\delta}$ be the symplectic immersion of
  $\Sigma(\epsilon) \times B^2(1)$ into $B^2(\sqrt{2+2\delta})\times B^2(\sqrt{2+2\delta})$ obtained by
applying  $\phi_{\hat{H}}$ to $C^{\Lambda} \times B^2(1).$
Clearly, $I^1_{\delta}$ agrees with $I^0_{\delta}$ away from $i_{\delta}^{-1}(C^{\Lambda}) \times B^2(1)$, and shares none of the original double points of $I^0_{\delta}$. The immersion $I^1_{\delta}$ can only have new  double points in $\phi_{\hat{H}}(C^{\Lambda}_{\pm} \times B^2(1))$
where  $C^{\Lambda}_+$  and $C^{\Lambda}_-$ are the portions of $C^{\Lambda}$  corresponding to points with $x_1$-values in $[\lambda, \Lambda]$ and  $[-\Lambda,-\lambda]$,
respectively.
We now show that, for an appropriate choice of $i_{\delta}$ and $\chi$,  $I^1_{\delta}$ can be adjusted on $\phi_{\hat{H}}(C^{\Lambda}_{\pm} \times B^2(1))$ so that no new double points occur.

In the $x_1y_1$-plane, $\phi_{\hat{H}}$ only moves  points in $C^{\Lambda}_{\pm}$ and does so only in the $y_1$--direction.
Moreover, the maximum displacement in this direction  is bounded from above by $$\left|\int_0^1 \chi'(x_1) \max(H)\, dx_1\right| < \max(|\chi'|)(\pi + \delta).$$
Hence, the image of $\phi_{\hat{H}}(C^{\Lambda}_{+} \times B^2(1))$ under $pr_1$ is contained in the set $$C^{\Lambda}_+  + \left([\lambda, \Lambda] \times \left[0, \max(|\chi'|)(\pi+\delta)\right]\right),$$
 and the projection of $\phi_{\hat{H}}(C^{\Lambda}_{-} \times B^2(1))$ is contained in  $$C^{\Lambda}_-  + \left( [-\Lambda, -\lambda] \times \left[ -\max(|\chi'|)(\pi+\delta), 0\right]\right).$$
Choosing the width of $C^{\Lambda}$ to be sufficiently small,  and $\max(|\chi'|)$ sufficiently close to $ \frac{1}{\Lambda - \lambda}$,
we may assume that both $\phi_{\hat{H}}(C^{\Lambda}_{\pm} \times B^2(1))$ project to regions in the $x_1y_1$-plane whose
area is  less than $\pi + 2\delta$ and hence  less than the area of each of the regions
$A$ and $B$. Acting on $\phi_{\hat{H}} (C^{\Lambda}_{-} \times B^2(1))$ and  $\phi_{\hat{H}} (C^{\Lambda}_{+} \times B^2(1))$
by another symplectic diffeomorphism which acts nontrivially only in the $x_1y_1$--directions, we may then ensure that  they are mapped by $pr_1$ into $A$ and $B$, respectively.
The resulting symplectic immersion therefore has no double points and is the
desired embedding of Theorem \ref{embed}.

\subsection{The proof of Theorem \ref{embed2}}

Scaling things appropriately, the argument  above yields  a symplectic embedding of $\Sigma(\epsilon) \times B^2(r)$ into $B^2(\sqrt{2}) \times B^2 (\sqrt{2})$ for any $r<1$ provided that $\epsilon$ is sufficiently small.
In this section, we show that the previous embedding procedure can be refined to obtain
a symplectic embedding of  $\Sigma(\epsilon) \times B^2(r)$  into $B^4(\sqrt{3})$. This will prove Theorem \ref{embed2} which implies that Theorem \ref{thm} is sharp.

\begin{remark}
The bi-disc $B^2(\sqrt{2}) \times B^2 (\sqrt{2})$ can be symplectically embedded into $B^4(2)$, by inclusion. The second Ekeland-Hofer capacity implies that this is optimal in the sense that $B^2(\sqrt{2}) \times B^2 (\sqrt{2})$ can not be symplectically embedded into a smaller ball.
\end{remark}

As in the proof of Theorem \ref{embed}, we start with a symplectic immersion $I^0$ of  $\Sigma(\epsilon) \times B^2(r)$ into $B^2(\sqrt{2}) \times B^2 (\sqrt{2})$ which acts by an immersion $i \colon \Sigma(\epsilon) \hookrightarrow B^2(\sqrt{2})$ in the first factor, and by inclusion on the second factor. The immersion $i$ is chosen so that for some $\lambda >0$ we have:
\begin{itemize}
\item the vertical and horizontal crossing portions of the image have length equal to $2\pi /\lambda$.
\item the regions $A$ and $B$ have areas in the interval $(\pi r^2, \pi)$ and  all but an arbitrarily small amount of this area is concentrated within a distance $\lambda$ of the horizontal crossing portion, $C$.
\end{itemize}
\noindent See Figure $2$.

\begin{figure}[h]
\label{distance}
\begin{center}
\psfrag{B}[][][0.8]{$B$}
\psfrag{A}[][][0.8]{$A$}
\includegraphics{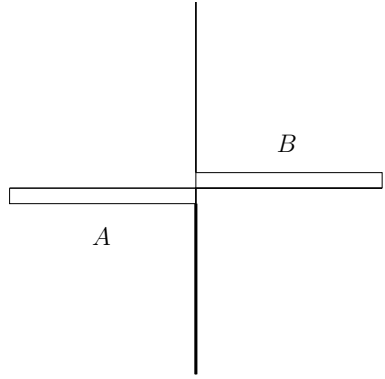} \caption{The symplectic immersion $i$, from a distance.}
\end{center}
\end{figure}

Set $\Lambda = \pi / \lambda$, so that $C^{\Lambda}=C$. Defining $\chi$ and $\hat{H}$ as in the proof of Theorem \ref{embed}, we remove the double points of $I^0$  by
applying the time-$1$
flow of $\hat{H}$ to $C\times B^2(r)$ to obtain a new immersion $I^1$. Choosing $\lambda$ and  $\left(\max(|\chi'|) - \frac{1}{\Lambda-\lambda}\right)$ to be sufficiently small, we may assume that the projection
$pr_1$ maps  $\phi_{\hat{H}}(C\times B^2(r))$ to the interior of the following  region of the $x_1y_1$-plane
\begin{equation*}
\label{region}
\mathbf{C}=C + \left\{([-\pi/\lambda, -\lambda] \times \left[ -\lambda, 0\right]) \cup ([\lambda,\pi/\lambda] \times \left[0, \lambda \right]) \right\}.
\end{equation*}
When the width of $C$ is small enough, the area of $\mathbf{C}$ is less than $2\pi$. As in the proof of Theorem \ref{embed} we can then apply a suitable symplectic map to $\phi_{\hat{H}}(C\times B^2(r))$
to shift the relevant parts of its projection into $A$ and $B$ and hence obtain a symplectic embedding of $\Sigma(\epsilon) \times B^2(r)$ into $B^2(\sqrt{2}) \times B^2 (\sqrt{2})$.

We now refine this embedding procedure by choosing a new immersion of $\Sigma(\epsilon)$. We  begin by analyzing the fibres of the projection map $pr_1$ acting on
$I^1(\Sigma(\epsilon) \times B^2(r))$. The  points of
$I^1(\Sigma(\epsilon) \times B^2(r))$ not in $\phi_{\hat{H}}(C\times B^2(r))$ belong to  fibres which can all be identified with  $B^2(r)$. The points in $\phi_{\hat{H}}(C\times B^2(r))$ belong to fibres of $pr_1$  determined
by the $x_1$-component
of their projections.  That is, the fibres of $pr_1$ corresponding to a fixed value of $x_1$ can all be included in a fixed subset of
$B^2(\sqrt{2})$, which we denote  by $F(x_1)$.  For $|x_1|\leq \lambda$, $F(x_1)$ is a fixed subset of the interior of $B^2(\sqrt{2})$. For  $\lambda < |x_1| \leq \pi/ \lambda$, each $F(x_1)$ is contained in the interior of
$B^2(\sqrt{1+\chi(x_1)})$ (see Lemma \ref{push}). We can choose the bump function $\chi(x_1)$ so that on $[-\pi/\lambda, -\lambda] \cup [\lambda, \pi/\lambda]$ it is arbitrarily $C^0$-close to the function $x_1 \mapsto 1- \frac{|x_1|-\lambda}{\pi/\lambda -\lambda}$. For all sufficiently small $\lambda>0$ we may then assume that  $F(x_1)$ is contained in the interior of  $B^2(\sqrt{2 -|x_1| \lambda/\pi})$ for all $x_1 \in [-\pi/\lambda, \pi/\lambda]$. By Remark \ref{error1},  the distance from $F(x_1)$ to the  boundary of $B^2(\sqrt{1+\chi(x_1)})$ is bounded from below by a positive constant which is independent of $\lambda$ and which goes to zero as $r$ approaches $\sqrt{2}$. For
$\chi$ as above, and $\lambda$ sufficiently small, we may assume the same is true of the   distance from $F(x_1)$ to the  boundary  of  $B^2(\sqrt{2 -|x_1| \lambda/\pi})$. In this case we denote the lower bound for this distance by
$D_r$.

%

%




We now apply a symplectomorphism $w$  to the
$x_1y_1$-plane which  winds $\mathbf{C}$
around itself as sketched in Figure $3$.
\begin{figure}[h]
\begin{center}
\label{twist}
\psfrag{C}[][][0.8]{$\mathbf{C}$}
\psfrag{w(C)}[][][0.8]{$w(\mathbf{C})$}
\includegraphics{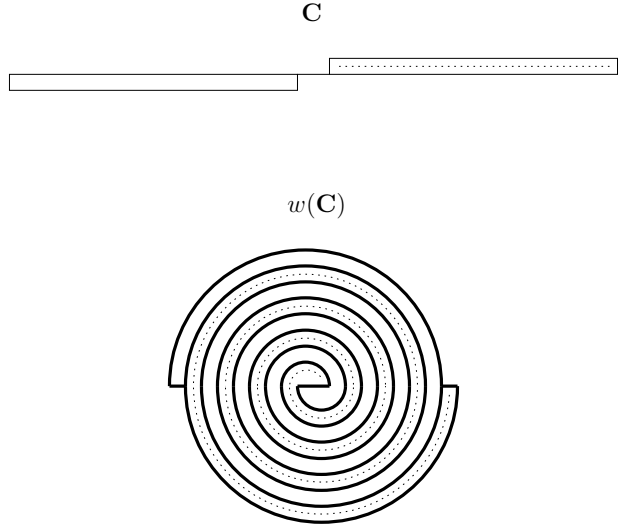} \caption{The winding map $w$ acting on $\mathbf{C}$.}
\end{center}

\end{figure}
This winding is nearly tight  but includes small  gaps (represented by the thicker lines in Figure $3$) to accommodate the winding of the
rest of $i(\Sigma(\epsilon))$.  Since the area of $\mathbf{C}$ is less than $2\pi$, for sufficiently small $\epsilon>0$
we may assume that the image of $i(\Sigma(\epsilon)) \cup \mathbf{C}$ under the winding map still lies in the ball $B^2(\sqrt{2})$.
Replacing the embedding $i$ in the previous construction with the
composition $w\circ i$, we get a new symplectic embedding $$I_w  \colon \Sigma(\epsilon) \times B^2(r) \hookrightarrow B^2(\sqrt{2}) \times B^2 (\sqrt{2}).$$
We now show that the image of $I_w$ is contained in
$B^4(\sqrt{3})$.

Let $(z_1, z_2) \in \C^2 \equiv \mathbb{R}^4$ be any point in the image of $I_w$. Then $z_1= w(x_1, y_1)$ for a unique point $(x_1, y_1)$ in $i(\Sigma(\epsilon))$
and $z_2$ belongs to $F(x_1)$.  Since $w(i(\Sigma(\epsilon)) \cup \mathbf{C}) \subset  B^2(\sqrt{2})$ we have
\begin{equation}
\label{1}
|z_1|< \sqrt{2}.
\end{equation}
By the analysis of the fibres  $F(x_1)$ above, we have
\begin{equation}
\label{error2}
|z_2| \leq \sqrt{2 -|x_1| \lambda/\pi}- D_r.
\end{equation}
On the other hand it follows from the definition  of the winding map $w$  that
\begin{equation}
\label{2}
|z_1| = \sqrt{\frac{2\lambda}{\pi }|x_1|} + O(\lambda) + O(\epsilon \lambda).
\end{equation}
The first approximation here comes from equating the area of the  portion of $C$ determined by $x_1$, $2|x_1| \lambda$,
with  $\pi |z_1|^2.$  The  error terms of \eqref{2} correspond, respectively,  to the discrepancy caused by the width of $\mathbf{C}$, and the discrepancy caused by the gap in the winding.

Together, equations \eqref{error2} and \eqref{2} yield
\begin{equation*}
\label{ }
(|z_2| + D_r)^2 +\frac{1}{2} (|z_1| - O(\lambda) - O(\epsilon \lambda))^2 \leq 2
\end{equation*}
The fact that $D_r$ is positive and independent of  $\lambda$, for sufficiently small $\lambda>0$, implies that
\begin{equation}
\label{3}
|z_2|^2   +  \frac{1}{2} |z_1|^2 \leq  2.
\end{equation}
Together, inequalities  \eqref{1} and  \eqref{3} then yield
\begin{equation*}
\label{ }
|z_1|^2 + |z_2|^2 \leq 3,
\end{equation*}
as desired.


\begin{thebibliography} {99}

\bibitem{bo}
F. Bourgeois,
A Morse-Bott approach to contact homology, PhD Thesis, Stanford University, 2002.

\bibitem{BEHWZ}
F. Bourgeois, Y. Eliashberg, H. Hofer, K. Wysocki, and E. Zehnder,
 Compactness results in symplectic field theory,
{\em Geom. Topol.}, \textbf{7} (2003), 799--888.

\bibitem{bormon} F. Bourgeois and K. Mohnke, Coherent orientations in symplectic field theory, {\em Math. Z.}, 248 (2004), no. 1, 123--146.

\bibitem{ciemon} K. Cieliebak and K. Mohnke, Compactness for punctured holomorphic curves, Conference on Symplectic Topology, {\em J. Symplectic Geom.}, 3 (2005), no. 4, 589--654.


\bibitem{eh1} I. Ekeland and H. Hofer, Symplectic topology and
Hamiltonian dynamics, {\it Math. Z.}, \textbf{200} (1989), 355--378.

\bibitem{ekehof} I. Ekeland and H. Hofer, Symplectic topology and
Hamiltonian dynamics II, {\it Math. Z.}, \textbf{203} (1990), 553--567.



\bibitem{EGH} Y. Eliashberg, A. Givental and H. Hofer, Introduction to symplectic field theory, GAFA 2000 (Tel Aviv, 1999), {\it Geom. Funct. Anal.}, 2000, Special Volume, Part II, 560--673.


\bibitem{gr} M. Gromov, Pseudo-holomorphic curves in symplectic manifolds, {\it Inv. Math.}, \textbf{82} (1985), 307--347.

\bibitem{guth} L. Guth, Symplectic embeddings of polydisks,
{\it Inv.  Math}, \textbf{172} (2008), 477--489.




\bibitem{hls} H. Hofer, V. Lizan and J.-C. Sikorav, On genericity of holomorphic curves in $4$-dimensional almost complex manifolds, {\it J. Geom. Anal.}, \textbf{7} (1997), 149--159.

\bibitem{hofa} H. Hofer, K. Wysocki and E. Zehnder, Properties of pseudoholomorphic curves in symplectisations I: Asymptotics, {\it Ann. Inst. H. Poincar\'{e} Anal. Non Lineaire}, \textbf{13} (1996) , 337--379.

\bibitem{hofi} H. Hofer, K. Wysocki and E. Zehnder, Properties of pseudoholomorphic curves in symplectisations II: Embedding controls and algebraic invariants, {\it Geom. Funct. Anal.}, \textbf{5} (1995),  337--379.





\bibitem{msa} D. McDuff and D. Salamon, $J$-holomorphic curves and
symplectic topology. American Mathematical Society Colloquium
Publications, 52. American Mathematical Society, Providence, RI,
2004.

\bibitem{pn} \'A. Pelayo, S. V\~u Ngoc, Sharp symplectic embeddings of cylinders, Preprint 2013, arXiv:1304.5250.

\bibitem{rs} J. Robbin and D. Salamon, The Maslov index for paths, {\it Topology}, \textbf{32} (1993), 827--844.


\bibitem{schl} F. Schlenk, Embedding problems in symplectic geometry
De Gruyter Expositions in Mathematics 40. Walter de Gruyter Verlag, Berlin. 2005.




\bibitem{tr} L. Traynor, Symplectic packing constructions, \emph{J. Diff. Geom.}, \textbf{42} (1995), 411--429.

\bibitem{wendl} C. Wendl, Automatic transversality and orbifolds
of punctured holomorphic curves in dimension four,  {\it Comment. Math. Helv.}, \textbf{85} (2010), 347--407.
\end{thebibliography}
\end{document}